\documentclass[aihp]{imsart}

\RequirePackage{amsthm,amsmath,amsfonts,amssymb}
\RequirePackage[numbers]{natbib}
\RequirePackage[colorlinks,citecolor=blue,urlcolor=blue]{hyperref}
\RequirePackage{graphicx}
\RequirePackage{stmaryrd}
\RequirePackage{esint}
\RequirePackage{enumitem}

\startlocaldefs

\theoremstyle{plain}
\newtheorem{prop}{Proposition}[section]
\newtheorem{lem}{Lemma}[section]
\newtheorem{thm}{Theorem}

\theoremstyle{remark}
\newtheorem{rem}{Remark}

\endlocaldefs

\begin{document}

\begin{frontmatter}

\title{Fractal properties of Aldous--Kendall random metric}
\runtitle{Fractal properties of Aldous--Kendall random metric}

\begin{aug}
\author[A]{\inits{G.}\fnms{Guillaume}~\snm{Blanc}\ead[label=e1]{guillaume.blanc1@universite-paris-saclay.fr}}
\address[A]{Université Paris-Saclay, Orsay, France\printead[presep={,\ }]{e1}}
\end{aug}

\begin{abstract}
Investigating a model of scale-invariant random spatial network suggested by Aldous, Kendall constructed a random metric $T$ on $\mathbb{R}^d$, for which the distance between points is given by the optimal connection time, when travelling on the road network generated by a Poisson process of lines with a speed limit.
In this paper, we look into some fractal properties of that random metric.
In particular, although almost surely the metric space $\left(\mathbb{R}^d,T\right)$ is homeomorphic to the usual Euclidean $\mathbb{R}^d$, we prove that its Hausdorff dimension is given by $(\gamma-1)d/(\gamma-d)>d$, where $\gamma>d$ is a parameter of the model; which confirms a conjecture of Kahn.
We also find that the metric space $\left(\mathbb{R}^d,T\right)$ equipped with the Lebesgue measure exhibits a multifractal property, as some points have untypically big balls around them.
\end{abstract}

\begin{abstract}[language=french]
En étudiant un modèle de ``scale-invariant random spatial network'' suggéré par Aldous, Kendall a construit une métrique aléatoire $T$ sur $\mathbb{R}^d$, pour laquelle la distance entre les points est donnée par le temps de trajet optimal, lorsqu'on se déplace sur le réseau de routes engendré par un processus de Poisson de droites avec une limitation de vitesse.
Dans cet article, nous nous intéressons aux propriétés fractales de cette métrique aléatoire.
En particulier, bien que presque sûrement l'espace métrique $\left(\mathbb{R}^d,T\right)$ soit homéomorphe à l'espace euclidien $\mathbb{R}^d$, nous montrons que sa dimension de Hausdorff est donnée par $(\gamma-1)d/(\gamma-d)>d$, où $\gamma>d$ est un paramètre du modèle ; cela confirme une conjecture de Kahn.
Nous montrons par ailleurs que l'espace métrique $\left(\mathbb{R}^d,T\right)$ muni de la mesure de Lebesgue est multifractal, puisque certains points se trouvent être au centre de boules atypiquement grosses.
\end{abstract}

\begin{keyword}[class=MSC]
\kwd{60D05}
\end{keyword}

\begin{keyword}
\kwd{Random geometry}
\kwd{Poisson process}
\kwd{Hausdorff dimension}
\end{keyword}

\end{frontmatter}


\section*{Introduction and main results}

In this paper, we are interested in fractal properties of a self-similar random metric on $\mathbb{R}^d$, which was constructed by Kendall in the pioneer paper \cite{kendall}.
Investigating a model suggested by Aldous in \cite[Subsection 4.1]{aldous}, the author of \cite{kendall} considers a Poisson random measure $\Pi$ with intensity proportional to $\mu_d\otimes v^{-\gamma}\mathrm{d}v$ on $\mathbb{L}_d\times\mathbb{R}_+^*$, where $\mathbb{L}_d$ is the space of affine lines in $\mathbb{R}^d$ and $\mu_d$ is its unique (up to multiplicative constant) invariant measure, and $\gamma>d$ is a parameter of the model.
Viewing each element $(\ell,v)$ of $\mathbb{L}_d\times\mathbb{R}_+^*$ as a road in $\mathbb{R}^d$, with $v$ the speed limit on the line $\ell$, the atoms of the measure $\Pi$ are seen as the roads of a network.
Kendall shows that every pair of points in $\mathbb{R}^d$ can be connected using this random road network, by paths that respect the speed limits; thus, denoting by $T(x,y)$ the optimal connection time between points $x,y\in\mathbb{R}^d$, the function $T:(x,y)\mapsto T(x,y)$ defines a metric on $\mathbb{R}^d$.
By construction, the random metric $T$ is invariant in distribution under rotations and translations, and satisfies a scaling property: for every $x,y\in\mathbb{R}^d$, we have the equality in distribution ${T(x,y)\overset{\text{\tiny law}}{=}|x-y|^{(\gamma-d)/(\gamma-1)}\cdot T(0,e_d)}$, where $e_d=(0,\ldots,0,1)$ denotes the $d$-th vector of the canonical basis of $\mathbb{R}^d$.
(The model will be presented in more detail in Section \ref{secrappels}.)

Although the random metric space $\left(\mathbb{R}^d,T\right)$ is almost surely homeomorphic to the usual Euclidean $\mathbb{R}^d$, in this paper we prove the following, confirming a conjecture of Kahn (see \cite[Section 7]{kahn}).
\begin{thm}\label{thmhausdim}
Almost surely, the metric space $\left(\mathbb{R}^d,T\right)$ has Hausdorff dimension 
\[\frac{(\gamma-1)d}{\gamma-d}>d.\]
\end{thm}
This fractal property is reminiscent of the Brownian sphere, for which the Hausdorff dimension $4$ is greater than the ``topological dimension'' $2$ (see, e.g, \cite[Theorem 2 and Theorem 3]{geobrown}).

To compute the above Hausdorff dimension, we use the Lebesgue measure as mass distribution on $\left(\mathbb{R}^d,T\right)$.
This choice intuitively makes sense given that $T$ is invariant (in distribution) under rotations and translations.
Implementing the standard ``energy method'' leads us to estimating the \emph{quick connection probability}, i.e, the probability $\mathbb{P}(T(0,e_d)\leq t)$ for $0$ and $e_d$ --- two points at unit Euclidean distance --- to be connected by the road network in time $t$, as $t\to0^+$.
We obtain the following sharp estimate.
\begin{thm}\label{thmquickco}
There exists constants $c,C>0$ such that
\[\text{$c\cdot t^\sigma\leq\mathbb{P}(T(0,e_d)\leq t)\leq C\cdot t^\sigma$ for all $t\in[0,1]$,\quad where $\sigma=\frac{(\gamma-1)(d+\gamma-2)}{\gamma-d}$.}\]
\end{thm}

Then, we pursue our analysis by investigating the Lebesgue measure of balls for the metric $T$.
Denoting by ${\overline\Gamma(x,t)=\left\{y\in\mathbb{R}^d:T(x,y)\leq t\right\}}$ the closed $T$-ball centered at $x\in\mathbb{R}^d$ and with radius $t>0$, we are interested in the behaviour of the quantity $\lambda\left(\overline\Gamma(x,t)\right)$ as $t\to0^+$.
In particular, we prove the following, where $\mathcal{L}$ denotes the union of the lines $\ell$ over the roads $(\ell,v)$ of $\Pi$.
\begin{thm}\label{thmvolboules}
Almost surely, the following holds:
\begin{enumerate}
\item for $\lambda$-almost every $x\in\mathbb{R}^d$, we have
\[\text{$\lambda\left(\overline\Gamma(x,t)\right)=t^{s^*+o(1)}$ as $t\to0^+$,\quad where $s^*=\frac{(\gamma-1)d}{\gamma-d}$,}\]
\item for every $x\in\mathcal{L}$, we have 
\[\text{$\lambda\left(\overline\Gamma(x,t)\right)=t^{s_*+o(1)}$ as $t\to0^+$,\quad with $s_*=\frac{(\gamma-1)d}{\gamma-d}-\frac{d-1}{\gamma-d}<s^*$,}\]
\item for every $x\in\mathbb{R}^d$, we have 
\[t^{s^*+o(1)}\leq\lambda\left(\overline\Gamma(x,t)\right)\leq t^{s_*-o(1)}\quad\text{as $t\to0^+$.}\]
\end{enumerate}
\end{thm}
In words, although for \emph{typical} points $x$ we have $\lambda\left(\overline\Gamma(x,t)\right)\approx t^{s^*}$ for small $t$, with $s^*$ the Hausdorff dimension of $\left(\mathbb{R}^d,T\right)$, there exists special points $x$ --- namely points \emph{on roads} --- for which $\lambda\left(\overline\Gamma(x,t)\right)\approx t^{s_*}$ for small $t$, with a smaller exponent $s_*$ (see the opening of Section \ref{secvolboules} for a heuristic).
This \emph{multifractal} property of the measured metric space $\left(\mathbb{R}^d,T,\lambda\right)$ distinguishes it for instance from the Brownian sphere equipped with its volume measure, for which the single exponent $4$ describes the behaviour of the volume of balls around \emph{all} points (see \cite[Proposition 6.1, equations $(38)$ and $(39)$, and Theorem 7.2]{legall}).
Although it can be the case in $\left(\mathbb{R}^d,T,\lambda\right)$ that other points than points on roads exhibit a different behaviour than typical points (see the results of Section \ref{secdisc} below), the third item of Theorem \ref{thmvolboules} above tells us that the two behaviours of typical points and points on roads are extremal.

\paragraph{Organisation of the paper.}

In a first section, we present the model in detail and recall the construction of Kendall's metric $T$.
Then, we present the proof of Theorem \ref{thmhausdim} and Theorem \ref{thmquickco} in Section \ref{secthm12}.
In Section \ref{secvolboules}, we study the Lebesgue measure of balls for the metric $T$, and --- in particular --- prove Theorem \ref{thmvolboules}.
We conclude this paper in Section \ref{secdisc}, where we discuss some natural questions our results leave open.

\begin{figure}[ht]
\centering
\begin{tabular}{ccc}
\includegraphics[width=0.25\linewidth]{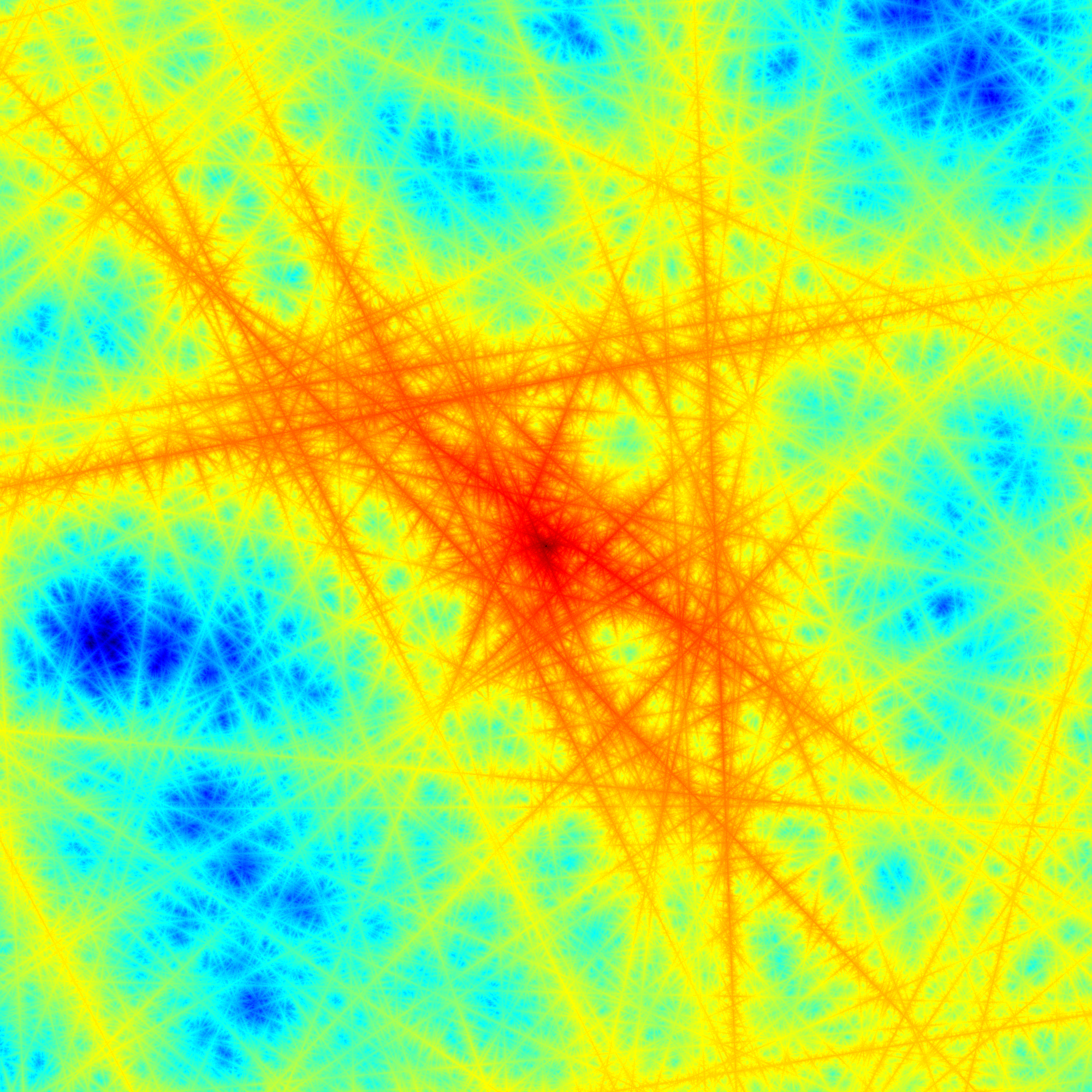}&\includegraphics[width=0.25\linewidth]{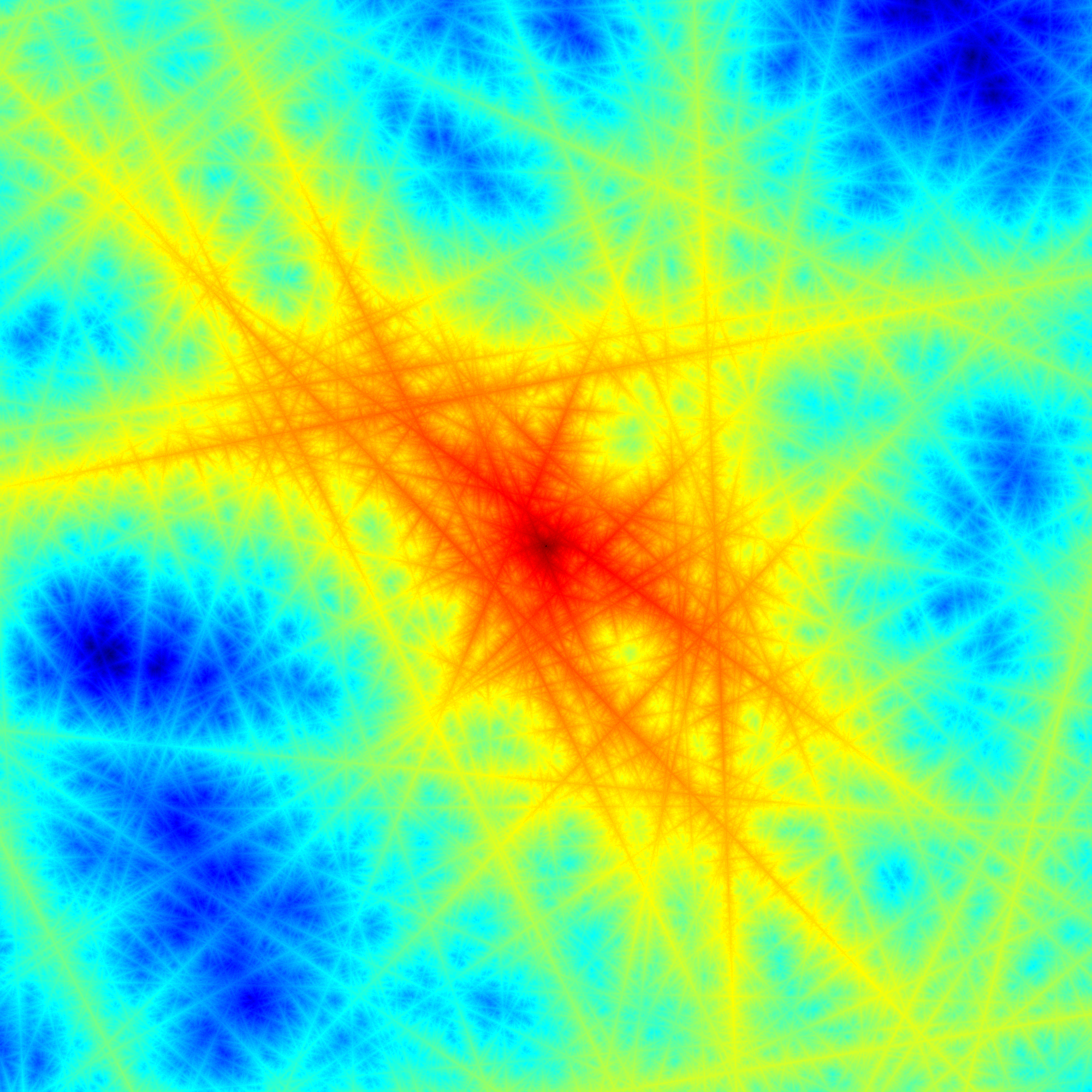}&\includegraphics[width=0.25\linewidth]{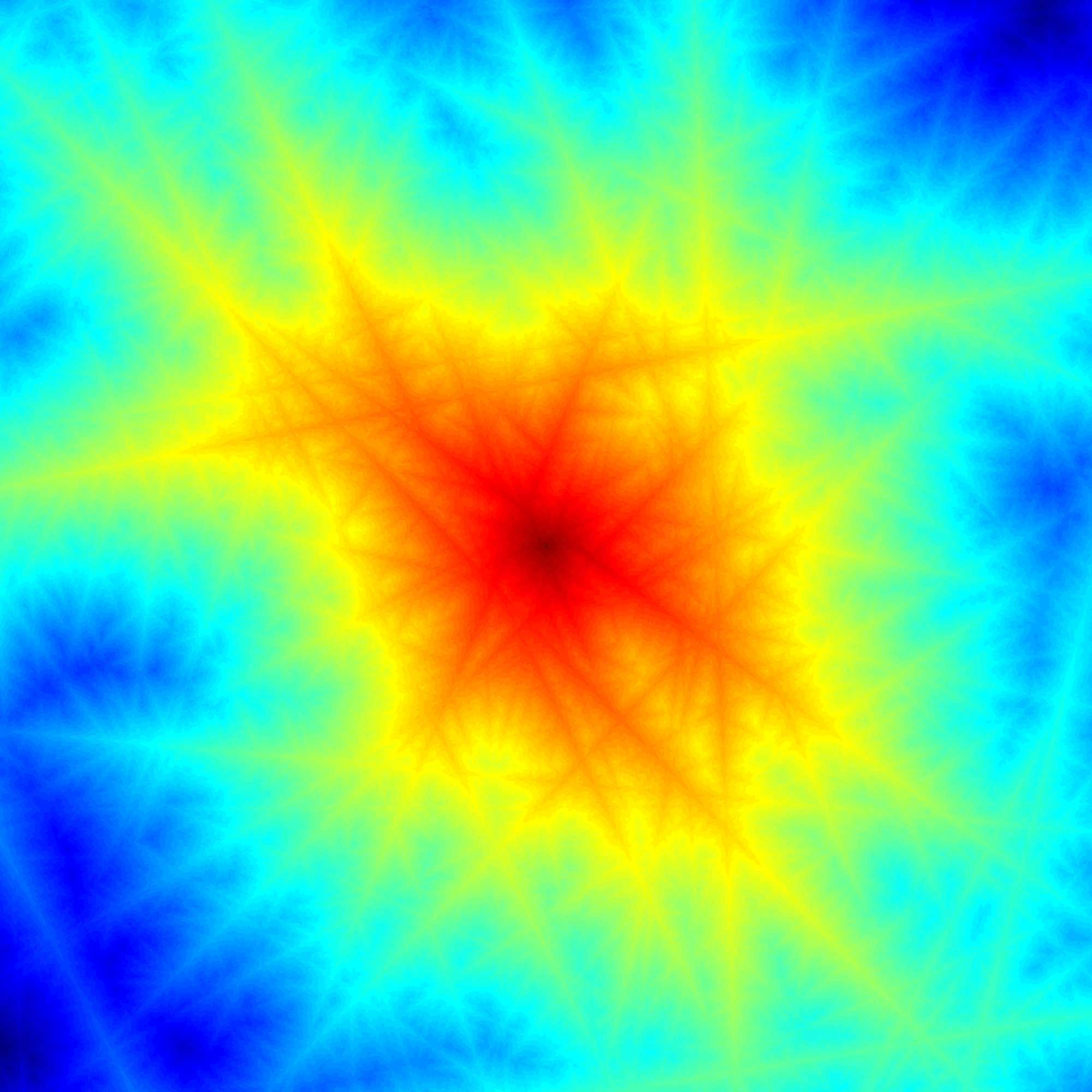}
\end{tabular}
\caption{Simulations by Arvind Singh of the distance-to-the-origin function, for the metric $T$ induced by the Poisson process $\Pi$, in dimension $d=2$. In each picture, the origin is the dark (red) point at the centre, and all points are coloured (from red to blue) according to their distance to the origin: the isocolour sets correspond to level sets for the distance-to-the-origin function. This naturally reveals the roads of the network. From left to right, the parameter $\gamma$ increases (a natural coupling for the speed limits allows to use the same lines).}\label{figsimu}
\end{figure}

\section{Introduction to Kendall's metric}\label{secrappels}

In this section, we present the model in detail, recalling the construction of Kendall's metric.
For more details, we refer to the original paper \cite{kendall}.
See also Kahn's subsequent paper \cite{kahn}.

\subsection{Definition of the Poisson process of roads}

We start by introducing the Poisson process of roads $\Pi$.
We call \emph{road} a couple $(\ell,v)$, with $\ell\subset\mathbb{R}^d$ an affine line, and $v\in\mathbb{R}_+^*$ the speed limit on $\ell$.
Denoting by $\mathbb{L}_d$ the set of affine lines in $\mathbb{R}^d$, the process $\Pi$ will be defined as a Poisson random measure on the space of roads $\mathbb{L}_d\times\mathbb{R}_+^*$.
To describe its intensity measure, let us recall the following.

\paragraph{The space of lines.}
In this paragraph, we briefly present the usual topology and invariant measure on the affine Grassmannian $\mathbb{L}_d$. We refer to \cite[Section 13.2]{stochintgeo} for more details, see also \cite[Section 8.2]{kendalletal}.
The space of lines can conveniently be described as follows.
Choosing a reference line, for instance the line $\ell_d=\{xe_d=(0,\ldots,0,x);~x\in\mathbb{R}\}$, and denoting by ${\ell_d^\perp=\mathbb{R}^{d-1}\times\{0\}}$ its orthogonal hyperplane, consider the surjective mapping
\[\begin{matrix}
\Phi:&\ell_d^\perp\times\mathbf{SO}\left(\mathbb{R}^d\right)&\longrightarrow&\mathbb{L}_d\\
&(w,g)&\longmapsto&g(w+\ell_d).
\end{matrix}\]
The space $\mathbb{L}_d$ is endowed with the finest topology that makes $\Phi$ into a continuous map (final topology), and with the corresponding Borel $\sigma$-algebra.

\subparagraph{Invariant measure.} 
On $\mathbb{L}_d$ there exists a unique, up to multiplicative constant, locally finite Borel measure which is invariant under rotations and translations.
It can be described as the pushforward $\mu_d=\Phi_*\left[\lambda_d^\perp\otimes\eta\right]$ of the measure $\lambda_d^\perp\otimes\eta$ by the map $\Phi$, where $\lambda_d^\perp$ denotes the Lebesgue measure on the $(d-1)$-dimensional Euclidean space $\ell_d^\perp$, and $\eta$ is the Haar probability on the compact group $\mathbf{SO}\left(\mathbb{R}^d\right)$.
The Borel measure $\mu_d$ satisfies the following invariance and scaling property: for $x\in\mathbb{R}^d$, $r>0$ and $h\in\mathbf{SO}\left(\mathbb{R}^d\right)$, the pushforward of $\mu_d$ by the continuous function
\[\begin{matrix}
f:&\mathbb{L}_d&\longrightarrow&\mathbb{L}_d\\
&\ell&\longmapsto&x+r\cdot h(\ell)\\
\end{matrix}\]
is the measure $f_*\mu_d=r^{-(d-1)}\cdot\mu_d$.
For a compact subset $K\subset\mathbb{R}^d$, we denote by $[K]=\{\ell\in\mathbb{L}_d:\ell\cap K\neq\emptyset\}$ the (closed) subset of lines that hit $K$.
If $L\subset\mathbb{R}^d$ is another compact subset, we simply write $[K;L]$ instead of $[K]\cap[L]$.
Writing ${\overline{B}(x,r)=\left\{y\in\mathbb{R}^d:|x-y|\leq r\right\}}$ for the closed Euclidean ball centered at $x\in\mathbb{R}^d$ with radius $r>0$, we have 
\begin{equation}\label{eqnormmud}
\mu_d\left[\overline{B}(x,r)\right]=\upsilon_{d-1}r^{d-1},
\end{equation}
where the constant $\upsilon_s=\frac{\pi^{s/2}}{\Gamma(s/2+1)}$ corresponds to the Lebesgue measure of the unit Euclidean ball in $\mathbb{R}^s$.
In complement to \eqref{eqnormmud}, we have the following.
\begin{lem}\label{lemdroitesboules}
There exists constants $C,c>0$ such that for every $x\neq y\in\mathbb{R}^d$:
\[\mu_d\left[\overline{B}(x,r)~;~\overline{B}(y,s)\right]\leq C\cdot\frac{r^{d-1}\cdot s^{d-1}}{|x-y|^{d-1}}\quad\text{for all $r,s>0$,}\]
and
\[\mu_d\left[\overline{B}(x,r)~;~\overline{B}(y,s)\right]\geq c\cdot\frac{r^{d-1}\cdot s^{d-1}}{|x-y|^{d-1}}\quad\text{for all $0<r,s\leq|x-y|$.}\]
\end{lem}

For the convenience of the reader, a proof is provided in the \nameref{appendix}.

\paragraph{The Poisson process of roads $\Pi$.}

Fix a parameter $\gamma>d$, and let $\Pi$ be a Poisson random measure with intensity ${c\cdot\mu_d\otimes v^{-\gamma}\mathrm{d}v}$ on $\mathbb{L}_d\times\mathbb{R}_+^*$, where the normalising constant $c=\upsilon_{d-1}^{-1}\cdot(\gamma-1)$ is chosen so as to have, in view of \eqref{eqnormmud}:
\begin{equation}\label{eqnorm}
c\cdot\mu_d\left[\overline{B}(x,r)\right]\cdot\int_{v_0}^\infty\frac{\mathrm{d}v}{v^\gamma}=r^{d-1}\cdot v_0^{-(\gamma-1)}\quad\text{for all $x\in\mathbb{R}^d$, $r>0$ and $v_0\in\mathbb{R}_+^*$.}
\end{equation}
The atoms of the measure $\Pi$ are seen as the roads of a network: for $(\ell,v)\in\mathbb{L}_d\times\mathbb{R}_+^*$, the positive real number $v$ represents the speed limit on the line $\ell$, and we say that $(\ell,v)$ is a road of $\Pi$ when $(\ell,v)$ is an atom of the measure $\Pi$.
We also use the notation $\Pi$ for the set of atoms of the measure $\Pi$.
For instance, we write $\mathcal{L}=\bigcup_{(\ell,v)\in\Pi}\ell$ for the union of the lines $\ell$ over the roads $(\ell,v)$ of $\Pi$.
For $v_0\in\mathbb{R}_+^*$, we denote by $\Pi_{v_0}$ the restriction of the Poisson random measure $\Pi$ to the Borel subset $\mathbb{L}_d\times[v_0,\infty[$.

\subparagraph{First properties.}
The road network generated by the process almost surely has the following properties:
\begin{enumerate}
\item[(/)] we have $\Pi\left(\mathbb{L}_d\times\{v\}\right)\leq1$ for all $v\in\mathbb{R}_+^*$,
\item[$(\varnothing)$] for every $R>0$ and $v_0\in\mathbb{R}_+^*$, there are only finitely many roads of $\Pi_{v_0}$ that pass through the ball $\overline{B}(0,R)$,
\item[$(\sharp)$] for every $x\in\mathbb{R}^d$ and $r>0$, there are infinitely many roads of $\Pi$ that pass through the ball $\overline{B}(x,r)$.
\end{enumerate}
In addition,
\begin{enumerate}
\item[$(\flat)$] for any given $x\in\mathbb{R}^d$, almost surely no road of $\Pi$ passes through $x$,
\item[$(\flat\flat)$] in dimensions $d\geq3$, almost surely the roads of $\Pi$ do not intersect,
\item[$(*)$] in dimension $d=2$, almost surely there is no crossing of three roads from $\Pi$.
\end{enumerate}
For the convenience of the reader, a proof of $(\flat\flat)$ and $(*)$ is provided in the \nameref{appendix}.

\subparagraph{Speed limits.}
The speed limits on the lines naturally extend to a (random) speed-limit function $V:\mathbb{R}^d\rightarrow\mathbb{R}_+$, defined by:
\[V(x)=\sup\{v;~(\ell,v)\in\Pi:\ell\ni x\}\quad\text{for all $x\in\mathbb{R}^d$,}\]
with the convention $\sup\emptyset=0$.
If no road of $\Pi$ passes through $x$ then $V(x)=0$, otherwise $x$ lies on at most two roads of $\Pi$, and $V(x)$ corresponds to the maximal speed limit.
For $x\in\mathbb{R}^d$ and $r>0$, we denote by $V^x_r=\sup_{\overline{B}(x,r)}V$ the speed limit of the fastest road passing through the ball $\overline{B}(x,r)$.
For every $v_0\in\mathbb{R}_+^*$, we have
\begin{equation}\label{eqVxr}
\mathbb{P}\left(V^x_r\geq v_0\right)=\mathbb{P}\left(\Pi\left(\left[\overline{B}(x,r)\right]\times[v_0,\infty[\right)>0\right)=1-\exp\left[-r^{d-1}\cdot v_0^{-(\gamma-1)}\right].
\end{equation}
In particular, we have $V^x_r\overset{\text{\tiny law}}{=}r^{(d-1)/(\gamma-1)}\cdot V^0_1$. 
Let us keep in mind that $V^x_r$ is of order $r^{(d-1)/(\gamma-1)}$.
For $x,y\in\mathbb{R}^d$ and $r,s>0$, we denote by $V^{x,y}_{r,s}$ the speed limit of the fastest road passing through both balls $\overline{B}(x,r)$ and $\overline{B}(y,s)$, writing $V^{x,y}_r$ instead of $V^{x,y}_{r,r}$ for short.

\subsection{Construction of the metric $T$}

In this subsection, we recall Kendall's construction of the metric $T$.
The distance $T(x,y)$ will be given by the optimal connection time between points $x,y\in\mathbb{R}^d$, using the random road network generated by $\Pi$, respecting the speed limits.

\paragraph{$\Pi$-paths.}\label{parpipaths}

To travel on the road network generated by the process, we use Kendall's $\Pi$-paths.
The reason for their definition is that one wants to connect every pair of points in $\mathbb{R}^d$ by continuous paths that come with a notion of instantaneous speed, and the usual piecewise $C^1$ paths fail at that task because of the properties $(\flat)$ and $(\flat\flat)$ of the road network.
We call \emph{regular path} a continuous function $\xi:[0,T]\rightarrow\mathbb{R}^d$ for which there exists an integrable function $\dot\xi:[0,T]\rightarrow\mathbb{R}^d$ such that 
\[\xi(t)=\xi(0)+\int_0^t\dot\xi(s)\mathrm{d}s\quad\text{for all $t\in[0,T]$.}\]
Note that by the Lebesgue differentiation theorem (see, e.g, \cite[Theorem 7.7]{rudin}), the instantaneous speed $\dot\xi$ of a regular path $\xi$ can be recovered as ${\dot\xi(t)=\lim_{\delta\to0^+}(\xi(t+\delta)-\xi(t-\delta))/(2\delta)}$ for almost every $t\in{]0,T[}$.
Now, given a realisation of $\Pi$, a \emph{$\Pi$-path} is a regular path $\xi:[0,T]\rightarrow\mathbb{R}^d$ which respects the speed limits, i.e, such that
\[\left|\dot\xi(t)\right|\leq V(\xi(t))\quad\text{for almost every $t\in[0,T]$.}\]

\subparagraph{A random metric on $\mathbb{R}^d$.}

Kendall's \cite[Theorem 3.6]{kendall} shows that almost surely, every pair of points in $\mathbb{R}^d$ can be connected by a $\Pi$-path. 
This is not trivial in view of properties $(\flat)$ and $(\flat\flat)$ of the road network, and it gives meaning to the definition:
\begin{equation}\label{eqoptconnt}
T(x,y)=\inf\left\{T>0:\text{there exists a $\Pi$-path $\xi:[0,T]\rightarrow\mathbb{R}^d$ connecting $x$ to $y$}\right\}\quad\text{for all $x,y\in\mathbb{R}^d$.}
\end{equation}
Then, it is not difficult to check that the (random) function
\[\begin{matrix}
T:&\mathbb{R}^d\times\mathbb{R}^d&\longrightarrow&\mathbb{R}_+\\
&(x,y)&\longmapsto&T(x,y)
\end{matrix}\]
is a metric.
At least, the fact that $T(x,x)=0$ for all $x\in\mathbb{R}^d$, the symmetry, and the triangle inequality, are immediate.
For the positivity, consider the following simple argument, which is illustrated in Figure \ref{figkahn}.
Fix a realisation of $\Pi$:
any $\Pi$-path connecting points $x\neq y\in\mathbb{R}^d$ has to travel an Euclidean distance at least $r=|x-y|$ inside the ball $\overline{B}(x,r)$, where the speed is limited at $V^x_r$.
Thus, the optimal connection time $T(x,y)$ is bounded from below by the ratio $r/V^x_r$, which is positive.
\begin{figure}[ht]
\centering
\includegraphics[width=0.45\linewidth]{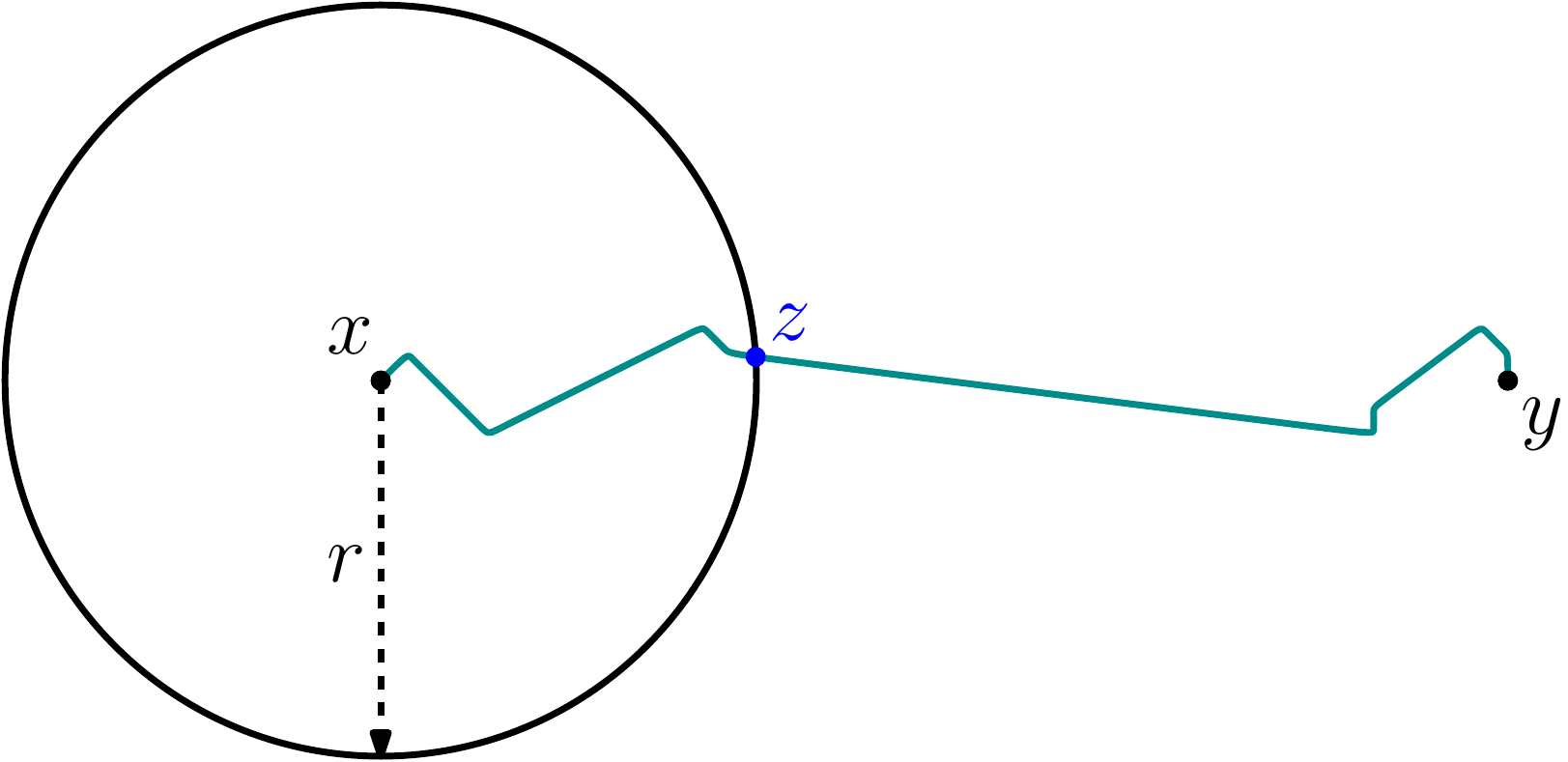}
\caption{Sketch of a $\Pi$-path connecting points $x\neq y\in\mathbb{R}^d$. For $0<r\leq|x-y|$, let $z\in\partial B(x,r)$ be the first exit point of the $\Pi$-path from the ball $\overline{B}(x,r)$: between $x$ and $z$, the $\Pi$-path travels an Euclidean distance $r$ inside the ball $\overline{B}(x,r)$, at speed at most $V^x_r$. Thus, we have $T(x,y)\geq r/V^x_r$.}\label{figkahn}
\end{figure}
(Note that the positivity of $T$ also follows from a more delicate result of Kendall we mention below, which is the existence of $\Pi$-geodesics.)

\begin{rem}
Equivalently, the metric $T$ can be seen as a first-passage percolation metric.
With the random field ${W:x\in\mathbb{R}^d\mapsto V(x)^{-1}\in[0,\infty]}$ we have, for all $x,y\in\mathbb{R}^d$:
\[T(x,y)=\inf\left\{\int_0^1W(\varsigma(s))\left|\dot\varsigma(s)\right|\mathrm{d}s,\quad\text{$\varsigma:[0,1]\rightarrow\mathbb{R}^d$ regular path connecting $x$ to $y$}\right\}.\]
Indeed, given a $\Pi$-path $\xi:[0,T]\rightarrow\mathbb{R}^d$ connecting $x$ to $y$, the reparametrised regular path ${\varsigma:s\in[0,1]\mapsto\xi(sT)}$ satisfies $\int_0^1W(\varsigma(s))\left|\dot\varsigma(s)\right|\mathrm{d}s\leq T$.
In the other direction, let ${\varsigma:[0,1]\rightarrow\mathbb{R}^d}$ be a regular path connecting $x$ to $y$, such that $\int_0^1W(\varsigma(s))\left|\dot\varsigma(s)\right|\mathrm{d}s<\infty$. 
Calling this quantity $T$, consider
\[\psi:t\in[0,T]\mapsto\inf\left\{u\in[0,1]:\int_0^uW(\varsigma(s))\left|\dot\varsigma(s)\right|\mathrm{d}s\geq t\right\}.\]
The reparametrisation ${\xi:t\in[0,T]\mapsto\varsigma(\psi(t))}$ is a $\Pi$-path that connects $x$ to $y$ in time $T$.
\end{rem}

\subparagraph{$\Pi$-geodesics.}
\emph{$\Pi$-geodesics} are defined as $\Pi$-paths travelling in optimal time: the $\Pi$-path $\xi:[0,T]\rightarrow\mathbb{R}^d$ is a $\Pi$-geodesic when no $\Pi$-path $\varsigma:[0,S]\rightarrow\mathbb{R}^d$ connecting $\xi(0)$ to $\xi(T)$ satisfies $S<T$.
Kendall's \cite[Corollary 3.5]{kendall} shows that the infimum in \eqref{eqoptconnt} is in fact a minimum: for every $x,y\in\mathbb{R}^d$, there exists a $\Pi$-geodesic connecting $x$ to $y$ in optimal time $T(x,y)$.

\paragraph{$\varepsilon$-$\Pi$-paths.}

A natural idea to make sense of $\Pi$-paths is to approximate them by simple sequential paths that nearly respect the speed limits.
Following Kendall, for $\varepsilon>0$ we call \emph{$\varepsilon$-$\Pi$-path} a continuous path $\xi:[0,T]\rightarrow\mathbb{R}^d$ for which there exists a subdivision $0\leq r_1<s_1\leq\ldots\leq r_k<s_k\leq T$ of the interval $[0,T]$, and a collection $(\ell_1,v_1);\ldots;(\ell_k,v_k)$ of roads of $\Pi$ such that:
\begin{itemize}
\item for each $i\in\llbracket1,k\rrbracket$, the points $\xi(r_i)$ and $\xi(s_i)$ lie on $\ell_i$, and are connected linearly by $\xi$ at speed $v_i$ over the time interval $[r_i,s_i]$,
\item for each $i\in\llbracket0,k\rrbracket$, with the conventions $s_0=0$ and $r_{k+1}=T$, the points $\xi(s_i)$ and $\xi(r_{i+1})$ are connected linearly by $\xi$ at speed $\varepsilon$ over the time interval $[s_i,r_{i+1}]$.
\end{itemize}
This definition is illustrated in Figure \ref{figepsilonpipath}.
\begin{figure}[ht]
\centering
\includegraphics[width=0.5\linewidth]{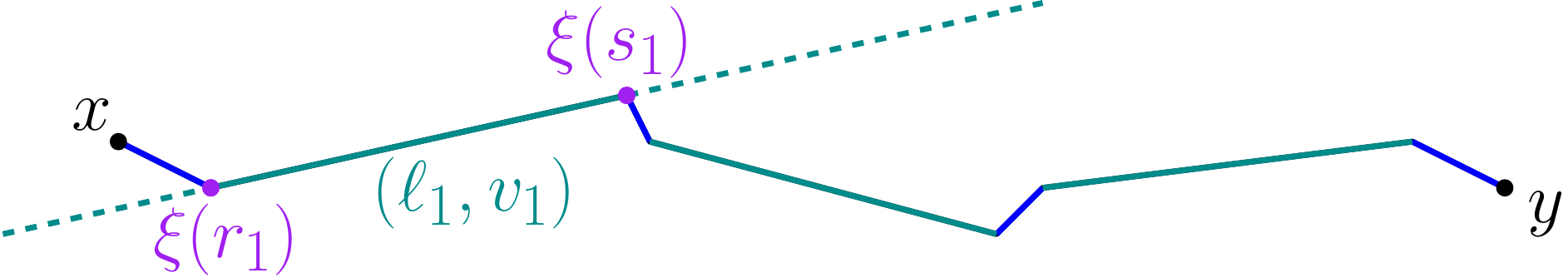}
\caption{Sketch of an $\varepsilon$-$\Pi$-path connecting points $x$ and $y$ in $\mathbb{R}^d$. The big light (cyan) segments are travelled on a road of $\Pi$ at maximal speed, and the smaller dark (blue) segments are travelled between roads at speed $\varepsilon$.}\label{figepsilonpipath}
\end{figure}
In particular, any $\varepsilon$-$\Pi$-path $\xi:[0,T]\rightarrow\mathbb{R}^d$ is a regular path such that
\[\left|\dot\xi(t)\right|\leq\varepsilon\vee V(\xi(t))\quad\text{for almost every $t\in[0,T]$.}\]
With this cutoff, setting
\[T_\varepsilon(x,y)=\inf_{\substack{k\geq0\\(\ell_1,v_1);\ldots;(\ell_k,v_k)\in\Pi\\a_1,b_1\in\ell_1;\ldots;a_k,b_k\in\ell_k}}\left\{\frac{|x-a_1|}{\varepsilon}+\frac{|a_1-b_1|}{v_1}+\ldots+\frac{|a_k-b_k|}{v_k}+\frac{|b_k-y|}{\varepsilon}\right\}\quad\text{for all $x,y\in\mathbb{R}^d$}\]
defines a metric $T_\varepsilon$ on $\mathbb{R}^d$.
Given that for each $x,y\in\mathbb{R}^d$, the quantity $T_\varepsilon(x,y)$ is nondecreasing as $\varepsilon$ decreases, one defines an extended metric ${\overline{T}:\mathbb{R}^d\times\mathbb{R}^d\rightarrow[0,\infty]}$, by setting:
\[\overline{T}(x,y)=\lim_{\varepsilon\downarrow0}T_\varepsilon(x,y)\quad\text{for all $x,y\in\mathbb{R}^d$.}\]
Kendall's results \cite[Theorem 3.8]{kendall} and \cite[Theorem 3.11]{kendall} show that in fact this definition coincides with that of $T$, i.e, for every $x,y\in\mathbb{R}^d$, we have
\begin{equation}\label{eqapproxT}
T_\varepsilon(x,y)\uparrow T(x,y)\quad\text{as $\varepsilon\downarrow0$}.
\end{equation}

\subparagraph{Measurability.} A byproduct of this approximation argument is the measurability of $T$ as a function of $\omega;x,y$, which justifies every application of Fubini's theorem in this paper.
For each $\varepsilon>0$, it is possible to write
\begin{equation}\label{eqmeasurabilityTe}
T_\varepsilon(x,y)=\tau_\varepsilon(\Pi;x,y)\quad\text{for all $x,y\in\mathbb{R}^d$,}
\end{equation}
for some measurable function $\tau_\varepsilon$, and it follows that\begin{equation}\label{eqmeasurabilityT}
T(x,y)=\tau(\Pi;x,y)\quad\text{for all $x,y\in\mathbb{R}^d$,}
\end{equation}
for some measurable function $\tau$.

\subsection{Fundamental properties of $T$}

In this section, we present some fundamental properties of the metric $T$; namely, its invariance and scaling property, its mixing, and the probabilistic estimate of Kahn.
We do not claim any new results here.
In particular, the estimates below show that $\left(\mathbb{R}^d,T\right)$ is almost surely homeomorphic to the usual Euclidean $\mathbb{R}^d$, which was already noted by Kendall (see \cite[Section 7]{kendall}) and Kahn (see \cite[Remark 5.1]{kahn}).

\paragraph{Invariance and scaling.}
The invariance and scaling properties of the process's intensity measure naturally extend to the random metric $T$.
Note that $T$ may be viewed as a random variable with values in the space $\mathbf{F}$ of functions from $\mathbb{R}^d\times\mathbb{R}^d$ to $\mathbb{R}$, endowed as usual with the product $\sigma$-algebra.

\begin{prop}\label{propscaling}
For every $x\in\mathbb{R}^d$, $r>0$ and $h\in\mathbf{SO}\left(\mathbb{R}^d\right)$, we have the equality in distribution
\[(T(x+r\cdot h(y),x+r\cdot h(z)))_{y,z\in\mathbb{R}^d}\overset{\text{\tiny law}}{=}\left(r^\frac{\gamma-d}{\gamma-1}\cdot T(y,z)\right)_{y,z\in\mathbb{R}^d}.\]
\end{prop}
\begin{proof}
Since the intensity measure of $\Pi$ is left invariant by the continuous map
\[\begin{matrix}
f:&\mathbb{L}_d\times\mathbb{R}_+^*&\longrightarrow&\mathbb{L}_d\times\mathbb{R}_+^*\\
&(\ell,v)&\longmapsto&\left(x+r\cdot h(\ell),r^\frac{d-1}{\gamma-1}\cdot v\right),
\end{matrix}\]
the Poisson random measure $\Pi'=f_*\Pi$ has the same distribution as $\Pi$.
Therefore, denoting by ${T'(\cdot,\cdot)=\tau(\Pi';\cdot,\cdot)}$ the random metric induced by the optimal connection time using $\Pi'$-paths, where $\tau$ is the measurable function of \eqref{eqmeasurabilityT}, we have on the one hand the equality in distribution
\[\left(T'(x+r\cdot h(y),x+r\cdot h(z))\right)_{y,z\in\mathbb{R}^d}\overset{\text{\tiny law}}{=}(T(x+r\cdot h(y),x+r\cdot h(z)))_{y,z\in\mathbb{R}^d}.\]
On the other hand we have, deterministically:
\[\tau\left(\Pi';x+r\cdot h(y),x+r\cdot h(z)\right)=\frac{r}{r^{(d-1)/(\gamma-1)}}\cdot\tau(\Pi;y,z)=r^\frac{\gamma-d}{\gamma-1}\cdot\tau(\Pi;y,z)\quad\text{for all $y,z\in\mathbb{R}^d$.}\]
This concludes the proof.
\end{proof}

\paragraph{Mixing.}
In the course of the proof above, we have seen that the process $\Pi$ is stationary with respect to the translations
\begin{equation}\label{eqdeff_x}
\left(\begin{matrix}
f_x:&\mathbb{L}_d\times\mathbb{R}_+^*&\longrightarrow&\mathbb{L}_d\times\mathbb{R}_+^*\\
&(\ell,v)&\longmapsto&(x+\ell,v)
\end{matrix}\right)_{x\in\mathbb{R}^d},
\end{equation}
in the sense that for every $x\in\mathbb{R}^d$, we have $f_x^*\Pi\overset{\text{\tiny law}}{=}\Pi$.
Furthermore, it is mixing, as shown in the following proposition.
Let us denote by $\mathbb{M}$ the space of atomic measures on $\mathbb{L}_d\times\mathbb{R}_+^*$.
The Poisson random measure $\Pi$ is a random variable with values in $\mathbb{M}$.

\begin{prop}\label{propmixing}
For every bounded measurable functions $\varphi$ and $\psi$ from $\mathbb{M}$ to $\mathbb{R}$, we have
\[\mathbb{E}[\varphi(\Pi)\cdot\psi(f_x^*\Pi)]\underset{|x|\to\infty}{\longrightarrow}\mathbb{E}[\varphi(\Pi)]\cdot\mathbb{E}[\psi(\Pi)],\]
where $f_x$ has been defined in \eqref{eqdeff_x}.
\end{prop}
We are indebted to an anonymous referee for pointing out that the above result holds in more generality than as we were stating in an earlier version of this paper, and for providing helpful references.
\begin{proof}
The proof uses a classical approximation argument from measure theory, similarly as in \cite[proof of Theorem 3.5]{kendall2}.
This proof strategy also appears elsewhere in the literature, see for instance \cite[proof of Proposition 2.3.7]{gaf}.

Let $\varphi$ and $\psi$ be bounded measurable functions from $\mathbb{M}$ to $\mathbb{R}$.
For $n\in\mathbb{N}^*$, let us denote by $\rho_n:\mathbb{M}\rightarrow\mathbb{M}$ the restriction map to the Borel subset ${\left[\overline{B}(0,n)\right]\times[1/n,\infty[}\subset\mathbb{L}_d\times\mathbb{R}_+^*$, and by ${\mathcal{A}_n=\sigma(\rho_n(\Pi))}$ the $\sigma$-algebra generated by $\rho_n(\Pi)$.
Fix $\varepsilon>0$.
As $\sigma(\Pi)$ coincides with the $\sigma$-algebra generated by $\bigcup_{n\geq1}\mathcal{A}_n$, by a standard measure theoretic argument (see, e.g, \cite[Lemma 3.16]{kallenberg}), there exists $n\in\mathbb{N}^*$, and bounded measurable functions $\varphi_n$ and $\psi_n$ from $\mathbb{M}$ to $\mathbb{R}$, such that
\begin{equation}\label{eqproofmixing1}
\mathbb{E}[|\varphi(\Pi)-\varphi_n(\rho_n(\Pi))|]\leq\varepsilon\quad\text{and}\quad\mathbb{E}[|\psi(\Pi)-\psi_n(\rho_n(\Pi))|]\leq\varepsilon.
\end{equation}
Denoting by $C>0$ a constant which dominates $\varphi$, $\psi$, $\varphi_n$ and $\psi_n$, we have on the one hand:
\begin{equation}\label{eqproofmixing2}
|\mathbb{E}[\varphi(\Pi)\cdot\psi(f_x^*\Pi)]-\mathbb{E}[\varphi_n(\rho_n(\Pi))\cdot\psi_n(\rho_n(f_x^*\Pi))]|\leq2C\varepsilon.
\end{equation}
On the other hand, the random measures $\rho_n(\Pi)$ and $\rho_n(f_x^*\Pi)$ are ``asymptotically independent''.
Indeed, denote by $G_x$ the event: ``no road of $\Pi_{1/n}$ passes through both balls $\overline{B}(0,n)$ and $\overline{B}(-x,n)$''.
The following holds:
\begin{itemize}
\item we have $\mathbb{P}(G_x)\rightarrow1$ as $|x|\to\infty$ \big(indeed, the integer $n$ is fixed, and we have $\mu_d\left[\overline{B}(0,n)~;~\overline{B}(-x,n)\right]\rightarrow0$ as $|x|\to\infty$, for instance by Lemma \ref{lemdroitesboules}\big),
\item on $G_x$, the random measure $\rho_n(\Pi)$ coincides with the restriction $\Pi_{0\setminus{-x}}$ of the Poisson random measure $\Pi$ to the Borel subset ${\left(\left[\overline{B}(0,n)\right]\middle\backslash\left[\overline{B}(-x,n)\right]\right)\times[1/n,\infty[}$, and $\rho_n(f_x^*\Pi)$ coincides with the restriction $\Pi_{{-x}\setminus0}$ of $\Pi$ to ${\left(\left[\overline{B}(-x,n)\right]\middle\backslash\left[\overline{B}(0,n)\right]\right)\times[1/n,\infty[}$,
\item the random measures $\Pi_{0\setminus{-x}}$ and $\Pi_{{-x}\setminus0}$ are independent.
\end{itemize}
With these three ingredients, standard manipulations show that
\[\mathbb{E}[\varphi_n(\rho_n(\Pi))\cdot\psi_n(\rho_n(f_x^*\Pi))]\underset{|x|\to\infty}{\longrightarrow}\mathbb{E}[\varphi_n(\rho_n(\Pi))]\cdot\mathbb{E}[\psi_n(\rho_n(\Pi))].\]
Then, equations \eqref{eqproofmixing1} and \eqref{eqproofmixing2} allow to conclude: as $|x|\to\infty$, we have
\[\mathbb{E}[\varphi(\Pi)\cdot\psi(f_x^*\Pi)]\longrightarrow\mathbb{E}[\varphi(\Pi)]\cdot\mathbb{E}[\psi(\Pi)].\]
\end{proof}

\paragraph{Probabilistic estimates.}
Finally, we present some fundamental probabilistic estimates on the metric $T$.

\subparagraph{Kahn's estimate.}
First, let us state Kahn's estimate on the $T$-diameter of Euclidean balls.
For $x\in\mathbb{R}^d$ and $r>0$, consider the random variable\footnote{As a consequence of Proposition \ref{propunifcontrolT} below, almost surely the function $T:\mathbb{R}^d\times\mathbb{R}^d\rightarrow\mathbb{R}_+$ is continuous, thus the supremum can be reduced to points $y,z$ having rational coordinates.}
\[T^*_{x,r}=\sup_{y,z\in\overline{B}(x,r)}T(y,z).\]
By invariance and scaling (Proposition \ref{propscaling}), we have the equality in distribution $T^*_{x,r}\overset{\text{\tiny law}}{=}T^*_{0,1}\cdot r^{(\gamma-d)/(\gamma-1)}$.
Kahn's \cite[Theorem 3.1]{kahn} shows the existence of constants $J,\kappa>0$ such that
\begin{equation}\label{eqkahn}
\mathbb{P}\left(T^*_{0,1}>t\right)\leq J\exp\left[-\kappa t^{\gamma-1}\right]\quad\text{for all $t\in\mathbb{R}_+^*$}.
\end{equation}

\subparagraph{A uniform control.}
Intuitively, for given $x,y\in\mathbb{R}^d$, the random variable $T(x,y)$ is of order $|x-y|^{(\gamma-d)/(\gamma-1)}$, since by invariance and scaling, we have the equality in distribution ${T(x,y)\overset{\text{\tiny law}}{=}|x-y|^{(\gamma-d)/(\gamma-1)}\cdot T(0,e_d)}$.
Slightly adapting Kahn's proof of \cite[Theorem 3.1]{kahn}, one can obtain the following estimate.

\begin{prop}\label{propunifcontrolT}
Almost surely, for every $R>0$, there exists a constant $C=C(\omega,R)>0$ such that
\begin{equation}\label{equnifcontrolT}
T(x,y)\leq C\cdot|x-y|^\frac{\gamma-d}{\gamma-1}\cdot\ln\left(\frac{4R}{|x-y|}\right)^\frac{1}{\gamma-1}\quad\text{for all $x\neq y\in\overline{B}(0,R)$}.
\end{equation}
In particular, the map $\mathrm{id}:\left(\overline{B}(0,R),|\cdot|\right)\rightarrow\left(\overline{B}(0,R),T\right)$ is $\beta$-Hölder, for every $\beta\in{]0,(\gamma-d)/(\gamma-1)[}$.
\end{prop}
\begin{proof}
The idea (of Kahn) is to combine Kendall's construction of $\Pi$-paths between pairs of points (\cite[Theorem 3.1]{kendall}), that we recall for convenience, with a uniform lower bound on the speed limits of the roads involved (Lemma \ref{lemunifcontrolV} below).

Fix $R\in\mathbb{N}^*=\{1,2,\ldots\}$, and let $\alpha\in{]0,1[}$ be a small parameter to be adjusted throughout the proof.
Let $(x,y)$ be a pair of points in $\overline{B}(0,R)$: Kendall constructs a $\Pi$-path connecting $x$ to $y$ in the following recursive way.
If $x=y$ then there is nothing to do, otherwise consider the fastest road $(\ell,v)$ of $\Pi$ that passes through both balls $\overline{B}(x,\alpha|x-y|)$ and ${\overline{B}(y,\alpha|x-y|)}$, and denote by $x'$ (resp. $y'$) the orthogonal projection of $x$ (resp. $y$) onto the line $\ell$.
The road $(\ell,v)$ allows to connect $x'$ to $y'$ in time $|x'-y'|/v$, and it remains to connect $x$ to $x'$, and $y'$ to $y$. 
For this, recursively apply the procedure to the pairs of points $(x,x')$ and $(y',y)$.
To encode the construction, let us mark the nodes $\varnothing;1,2;\ldots$ of the infinite (complete) binary tree $\mathbb{T}$ in the following way.
Initially, set $(x_\varnothing,y_\varnothing)=(x,y)$.
Then, for every ${n\in\mathbb{N}=\{0,1,\ldots\}}$ such that the pairs of points $((x_u,y_u);~u\in\mathbb{T}:|u|=n)$ have been constructed, proceed as follows.
For each node $u\in\mathbb{T}$ such that $|u|=n$, consider the fastest road $(\ell_u,v_u)$ of $\Pi$ that passes through both balls $\overline{B}(x_u,\alpha|x_u-y_u|)$ and ${\overline{B}(y_u,\alpha|x_u-y_u|)}$, and denote by $x_u'$ (resp. $y_u'$) the orthogonal projection of $x_u$ (resp. $y_u$) onto the line $\ell_u$.
Recalling the notation introduced above, note that $v_u=V^{x_u,y_u}_{\alpha|x_u-y_u|}$.
Record the time $T_u=|x_u'-y_u'|/v_u$ it takes to connect $x_u'$ to $y_u'$ using the road $(\ell_u,v_u)$, and define the pairs of points associated to the children $u1$ and $u2$ of $u$ in $\mathbb{T}$ as $(x_{u1},y_{u1})=(x_u,x_u')$ and ${(x_{u2},y_{u2})=(y_u',y_u)}$.
From all these ``middle'' connections $x_u'\to y_u'$, we claim that it is indeed possible to obtain a $\Pi$-path connecting $x$ to $y$, which travels in time
\begin{equation}\label{eqintime}
\sum_{u\in\mathbb{T}}T_u=\sum_{u\in\mathbb{T}}\frac{|x_u'-y_u'|}{v_u}\leq\sum_{u\in\mathbb{T}}\frac{|x_u-y_u|}{V^{x_u,y_u}_{\alpha|x_u-y_u|}}.
\end{equation}
Therefore, the above right hand side provides an upper bound on $T(x,y)$.
Now, consider the following lemma.
Note that, assuming $\alpha\leq1/3$, we have $x_u,y_u\in\overline{B}(0,2R)$ for all $u\in\mathbb{T}$.
\begin{lem}\label{lemunifcontrolV}
Almost surely, there exists a constant $c>0$ such that
\begin{equation}\label{equnifcontrolV}
V^{x,y}_{\alpha|x-y|}\geq c\cdot\frac{|x-y|^\frac{d-1}{\gamma-1}}{\ln(4R/(\alpha|x-y|))^\frac{1}{\gamma-1}}\quad\text{for all $x\neq y\in\overline{B}(0,2R)$}.
\end{equation}
\end{lem}
Then, equation \eqref{equnifcontrolT} will be obtained by combining \eqref{eqintime} and \eqref{equnifcontrolV}.
\begin{proof}
The proof relies on a discretisation argument, together with the Borel-Cantelli lemma.

For each $n\in\mathbb{N}$, set $r_n=\alpha^n\cdot4R$, and denote by $\left(\overline{B}(z,r_n/2)\right)_{z\in\mathcal{Z}_n}$ a covering of $\overline{B}(0,2R)$ by balls of radius $r_n/2$, with centres $z\in\overline{B}(0,2R)$ more than $r_n/2$ apart from each other.
In particular, the balls $\overline{B}(z,r_n/4)$ for $z\in\mathcal{Z}_n$ are disjoint and included in $\overline{B}(0,2R+r_n/4)$, and we have ${\#\mathcal{Z}_n\leq(8R/r_n+1)^d=\left(2\alpha^{-n}+1\right)^d}$.
Let $x\neq y\in\overline{B}(0,2R)$, and let $n\in\mathbb{N}^*$ be such that $r_n\leq|x-y|\leq r_{n-1}$.
There exists $x',y'\in\mathcal{Z}_{n+1}$ such that $x\in\overline{B}(x',r_{n+1}/2)$ and ${y\in\overline{B}(y',r_{n+1}/2)}$. 
Since $r_{n+1}=\alpha r_n\leq\alpha|x-y|$, by the triangle inequality we have $V^{x,y}_{\alpha|x-y|}\geq V^{x',y'}_{r_{n+1}/2}$.
Now, let us control the $V^{x',y'}_{r_{n+1}/2}$ uniformly in $x',y'$.
First, note that $x'$ and $y'$ are at distance at most
\[|x'-x|+|x-y|+|y-y'|\leq\frac{r_{n+1}}{2}+r_{n-1}+\frac{r_{n+1}}{2}=\left(\frac{\alpha}{2}+\frac{1}{\alpha}+\frac{\alpha}{2}\right)\cdot r_n=:C\cdot r_n\]
from each other.
Using Lemma \ref{lemdroitesboules}, we find that there exists a constant $c>0$ such that
\[\mathbb{P}\left(V^{x',y'}_{r_{n+1}/2}<r_n^\frac{d-1}{\gamma-1}\cdot\zeta_n^{-1}\right)\leq\exp\left[-c\cdot r_n^{d-1}\cdot\left(r_n^\frac{d-1}{\gamma-1}\cdot\zeta_n^{-1}\right)^{-(\gamma-1)}\right]=\exp\left[-c\zeta_n^{\gamma-1}\right],\]
where the $\zeta_n$ are positive real numbers which will be adjusted later.
Denoting by $A_n$ the event: ``there exists ${x',y'\in\mathcal{Z}_{n+1}}$ with $|x'-y'|\leq Cr_n$ such that ${V^{x',y'}_{r_{n+1}/2}<r_n^{(d-1)/(\gamma-1)}\cdot \zeta_n^{-1}}$'', by a union bound we get
\[\mathbb{P}(A_n)\leq\#\mathcal{Z}_n^2\cdot\exp\left[-c\zeta_n^{\gamma-1}\right]\leq\left(2\alpha^{-n}+1\right)^{2d}\cdot\exp\left[-c\zeta_n^{\gamma-1}\right].\]
Then, setting
\[\zeta_n=\left(\frac{3d\ln(1/\alpha)}{c}\cdot n\right)^\frac{1}{\gamma-1},\quad\text{so that}\quad\exp\left[-c\zeta_n^{\gamma-1}\right]=\alpha^{3dn},\]
we obtain $\sum_{n\geq1}\mathbb{P}(A_n)<\infty$.
By the Borel-Cantelli lemma, almost surely the event $A_n$ fails to be realised for all sufficiently large $n$, hence there exists a constant $c'>0$ such that for every $n\in\mathbb{N}^*$:
\[V^{x',y'}_{r_{n+1}/2}\geq c'\cdot r_n^\frac{d-1}{\gamma-1}\cdot n^{-\frac{1}{\gamma-1}}\quad\text{for all $x',y'\in\mathcal{Z}_{n+1}$ such that $|x'-y'|\leq Cr_n$.}\]
It follows that
\[V^{x,y}_{\alpha|x-y|}\geq c'\cdot(\alpha|x-y|)^\frac{d-1}{\gamma-1}\cdot\log_{1/\alpha}\left(\frac{4R}{\alpha|x-y|}\right)^{-\frac{1}{\gamma-1}}\quad\text{for all $x\neq y\in\overline{B}(0,2R)$,}\]
which proves \eqref{equnifcontrolV}.
\end{proof}
Back to the proof of the proposition, let us plug \eqref{equnifcontrolV} into \eqref{eqintime}: this yields
\[T(x,y)\leq\frac{1}{c}\sum_{u\in\mathbb{T}}|x_u-y_u|^\frac{\gamma-d}{\gamma-1}\cdot\ln\left(\frac{4R}{\alpha|x_u-y_u|}\right)^\frac{1}{\gamma-1}=\frac{1}{c}\sum_{u\in\mathbb{T}}\phi(|x_u-y_u|),\]
with $\phi:s\mapsto s^{(\gamma-d)/(\gamma-1)}\cdot\ln(4R/(\alpha s))^{1/(\gamma-1)}$.
Note that $|x_u-y_u|\leq\alpha^{|u|}\cdot|x-y|\leq2R$ for all $u\in\mathbb{T}$.
Now, a straightforward analysis shows that, assuming $\alpha\leq2e^{-1/(\gamma-d)}$, the function $\phi$ is nondecreasing over the interval $]0,2R]$: using this, we get
\[\begin{split}
T(x,y)&\leq\frac{1}{c}\sum_{u\in\mathbb{T}}\phi\left(\alpha^{|u|}\cdot|x-y|\right)\\
&=\frac{1}{c}\sum_{n\geq0}2^n\cdot\phi\left(\alpha^n\cdot|x-y|\right)=\frac{1}{c}\sum_{n\geq0}\left(2\alpha^\frac{\gamma-d}{\gamma-1}\right)^n\cdot|x-y|^\frac{\gamma-d}{\gamma-1}\cdot\ln\left(\frac{4R}{\alpha^{n+1}\cdot|x-y|}\right)^\frac{1}{\gamma-1}.
\end{split}\]
Finally, we have
\[\ln\frac{4R}{\alpha^{n+1}\cdot|x-y|}=(n+1)\ln\frac{1}{\alpha}+\ln\frac{4R}{|x-y|}\leq\left((n+1)\cdot\frac{\ln(1/\alpha)}{\ln 2}+1\right)\cdot\ln\frac{4R}{|x-y|}=:a_n\cdot\ln\frac{4R}{|x-y|}\]
for all $n\in\mathbb{N}$.
Thus, we obtain
\[T(x,y)\leq\frac{1}{c}\sum_{n\geq0}\left(2\alpha^\frac{\gamma-d}{\gamma-1}\right)^n\cdot a_n^\frac{1}{\gamma-1}\cdot|x-y|^\frac{\gamma-d}{\gamma-1}\cdot\ln\left(\frac{4R}{|x-y|}\right)^\frac{1}{\gamma-1}\quad\text{for all $x\neq y\in\overline{B}(0,R)$.}\]
Assuming $\alpha<2^{-(\gamma-1)/(\gamma-d)}$, we see that $\sum_{n\geq0}\left(2\alpha^\frac{\gamma-d}{\gamma-1}\right)^n\cdot a_n^\frac{1}{\gamma-1}<\infty$, which achieves to prove \eqref{equnifcontrolT}.
The second assertion of the proposition is a straightforward consequence of that equation.
\end{proof}

In particular, almost surely, for every $R>0$, the map ${\mathrm{id}:\left(\overline{B}(0,R),|\cdot|\right)\rightarrow\left(\overline{B}(0,R),T\right)}$ is continuous.
By the compactness of $\overline{B}(0,R)$, this map is in fact a homeomorphism (see, e.g, \cite[Theorem 26.6]{munkres}), and globally so is the map ${\mathrm{id}:\left(\mathbb{R}^d,|\cdot|\right)\rightarrow\left(\mathbb{R}^d,T\right)}$.
Therefore, the metric $T$ may be viewed as a random variable with values in the space $\mathbf{C}$ of continuous functions from $\mathbb{R}^d\times\mathbb{R}^d$ to $\mathbb{R}$, endowed as usual with the topology of uniform convergence on compact subsets.

\section{Hausdorff dimension and quick connection probability}\label{secthm12}

In this section, we compute the Hausdorff dimension of the random metric space $\left(\mathbb{R}^d,T\right)$, and provide sharp estimates on the \emph{quick connection probability}, i.e, the probability $\mathbb{P}(T(0,e_d)\leq t)$ for two points at unit Euclidean distance to be connected by the road network in time $t$, as $t\to0^+$.
Actually, doing the former is quite straightforward once the latter has been established, so we consider it first.

\subsection{Hausdorff dimension of $\left(\mathbb{R}^d,T\right)$}

Theorem \ref{thmhausdim} states that the Hausdorff dimension of $\left(\mathbb{R}^d,T\right)$ is given by ${(\gamma-1)d/(\gamma-d)>d}$: this confirms a conjecture of Kahn, stated in \cite[Section 7]{kahn}.
For a reference on Hausdorff dimension see, e.g, \cite[Chapter 4]{brownianmotion}.
The proof below is conditional on the estimate of Theorem \ref{thmquickco}, which we prove in the next subsection; namely, we assume here that the following holds: there exists a constant $C>0$ such that
\begin{equation}\label{eqbsoptconntime}
\text{$\mathbb{P}(T(0,e_d)\leq t)\leq C\cdot t^\sigma$ for all $t\in[0,1]$,\quad where $\sigma=\frac{(\gamma-1)(d+\gamma-2)}{\gamma-d}$.}
\end{equation}

\begin{proof}[Proof of Theorem \ref{thmhausdim} assuming \eqref{eqbsoptconntime}]
\small\textsc{Upper bound\textbf{.}} 
\normalsize By Proposition \ref{propunifcontrolT}, almost surely, for each $R>0$, the map ${\mathrm{id}:\left(\overline{B}(0,R),|\cdot|\right)\rightarrow\left(\overline{B}(0,R),T\right)}$ is $\beta$-Hölder, for every ${\beta\in{]0,(\gamma-d)/(\gamma-1)[}}$.
Therefore, we have
\[\dim_H\left(\overline{B}(0,R),T\right)\leq\frac{\dim_H\left(\overline{B}(0,R),|\cdot|\right)}{\beta}\quad\text{for every $\beta\in\left]0,\frac{\gamma-d}{\gamma-1}\right[$,}\]
hence $\dim_H\left(\overline{B}(0,R),T\right)\leq(\gamma-1)/(\gamma-d)\cdot d$.
By the countable stability of Hausdorff dimension, we obtain 
\[\dim_H\left(\mathbb{R}^d,T\right)\leq\frac{(\gamma-1)d}{\gamma-d}.\]

\small\textsc{Lower bound\textbf{.}}
\normalsize Let us now turn to the converse inequality.
The energy method (see \cite[Theorem 4.27]{brownianmotion}) states that if the energy integral $\int_{\overline{B}(0,1)}\int_{\overline{B}(0,1)}T(x,y)^{-s}\mathrm{d}y\mathrm{d}x$ is finite, then $\dim_H\left(\overline{B}(0,1),T\right)\geq s$.
Therefore, it suffices to show that for every $s\in[0,(\gamma-1)d/(\gamma-d)[$, we have $\int_{\overline{B}(0,1)}\int_{\overline{B}(0,1)}T(x,y)^{-s}\mathrm{d}y\mathrm{d}x<\infty$ almost surely.
Let then ${s\in[0,(\gamma-1)d/(\gamma-d)[}$: by Fubini's theorem, and invariance and scaling, we have
\[\begin{split}
\mathbb{E}\left[\int_{\overline{B}(0,1)}\int_{\overline{B}(0,1)}\frac{\mathrm{d}y}{T(x,y)^s}\mathrm{d}x\right]&=\int_{\overline{B}(0,1)}\int_{\overline{B}(0,1)}\mathbb{E}\left[T(x,y)^{-s}\right]\mathrm{d}y\mathrm{d}x\\
&=\int_{\overline{B}(0,1)}\int_{\overline{B}(0,1)}\frac{\mathrm{d}y}{|x-y|^{(\gamma-d)s/(\gamma-1)}}\mathrm{d}x\cdot\mathbb{E}\left[T(0,e_d)^{-s}\right].
\end{split}\]
As $(\gamma-d)s/(\gamma-1)<d$, we have $\int_{\overline{B}(0,1)}\int_{\overline{B}(0,1)}|x-y|^{-(\gamma-d)s/(\gamma-1)}\mathrm{d}y\mathrm{d}x<\infty$; and by assumption, the other term
\[\mathbb{E}\left[T(0,e_d)^{-s}\right]=s\int_0^\infty\mathbb{P}\left(T(0,e_d)\leq t\right)\cdot t^{-s-1}\mathrm{d}t\leq s\int_0^1C\cdot t^\sigma\cdot t^{-s-1}\mathrm{d}t+1\]
is finite as well, since $\sigma-s>(\gamma-1)(\gamma-2)/(\gamma-d)>0$.
\end{proof}

To actually complete the above proof, we need to prove Theorem \ref{thmquickco}, which is the subject of the next subsection.

\subsection{Sharp estimates for quick connection probability}

In this subsection, we prove Theorem \ref{thmquickco}.
The lower bound is Proposition \ref{propbi} below, and the more delicate upper bound is Proposition \ref{propbs}.

\paragraph{Lower bound.}

A simple way of connecting $0$ and $e_d$ quickly is to use a fast road passing close to both points.
This basic idea gives the following lower bound.

\begin{prop}\label{propbi}
There exists a constant $c>0$ such that
\[\text{$\mathbb{P}(T(0,e_d)\leq t)\geq c\cdot t^\sigma$ for all $t\in[0,1]$,\quad where $\sigma=\frac{(\gamma-1)(d+\gamma-2)}{\gamma-d}$.}\]
\end{prop}
\begin{proof}
Let $t\in{]0,1]}$, and let $r=r(t)>0$ be a small parameter to be adjusted later.
We have (see Figure \ref{figproplb})
\[T(0,e_d)\leq T^*_{0,r}+\frac{1}{V^{0,e_d}_r}+T^*_{e_d,r},\]
hence
\[\mathbb{P}(T(0,e_d)\leq t)\geq\mathbb{P}\left(T^*_{0,r}\leq\frac{t}{3}~;~V^{0,e_d}_r\geq\frac{3}{t}~;~T^*_{e_d,r}\leq\frac{t}{3}\right).\]
\begin{figure}[ht]
\centering
\includegraphics[width=0.55\linewidth]{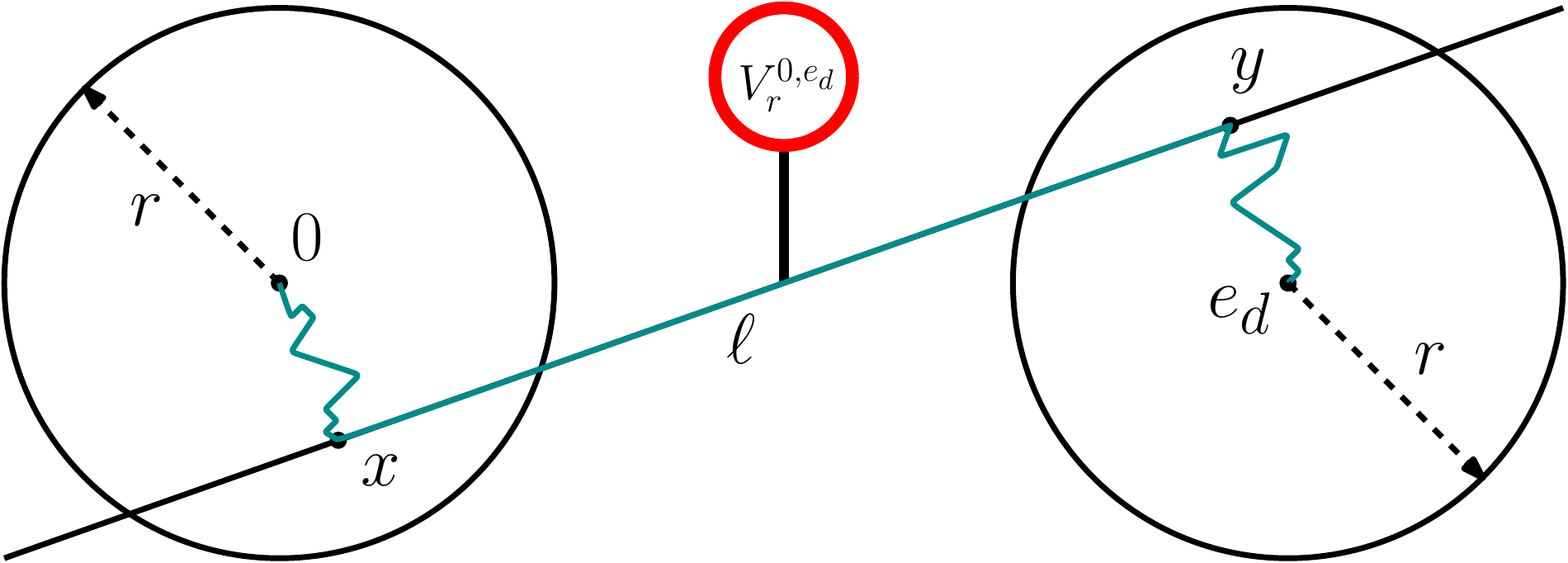}
\caption{(Planted on $\ell$ is a speed limit sign.) Denoting by $(\ell,v)$ the fastest road passing through both balls $\overline{B}(0,r)$ and $\overline{B}(e_d,r)$, with $v=V^{0,e_d}_r$ by definition, let $x$ (resp. $y$) be the orthogonal projection of $0$ (resp. $e_d$) onto $\ell$.
The points $x$ and $y$ can be connected  in time $|x-y|/v\leq1/v$, and connecting points $0$ and $x$ (resp. $y$ and $e_d$) takes time at most $T^*_{0,r}$ \big(resp. $T^*_{e_d,r}$\big).}\label{figproplb}
\end{figure}
Now, the three events in the above intersection are all increasing, in the sense that if they hold for $\Pi$, then they also hold for any atomic measure $\Pi'$ such that $\Pi'(A)\geq\Pi(A)$ for all measurable sets $A\subset\mathbb{L}_d\times\mathbb{R}_+^*$.
Therefore, by the FKG inequality for Poisson processes (see, e.g, \cite[Theorem 20.4]{lastpenrose}), we have
\[\mathbb{P}\left(T^*_{0,r}\leq\frac{t}{3}~;~V^{0,e_d}_r\geq\frac{3}{t}~;~T^*_{e_d,r}\leq\frac{t}{3}\right)\geq\mathbb{P}\left(T^*_{0,r}\leq\frac{t}{3}\right)\cdot\mathbb{P}\left(V^{0,e_d}_r\geq\frac{3}{t}\right)\cdot\mathbb{P}\left(T^*_{e_d,r}\leq\frac{t}{3}\right).\]
Next, by \eqref{eqkahn} together with invariance and scaling, there exists a constant $c>0$ such that
\[\mathbb{P}\left(T^*_{0,r}>\frac{t}{3}\right)\leq J\exp\left[-\frac{ct^{\gamma-1}}{r^{\gamma-d}}\right]\quad\text{and}\quad\mathbb{P}\left(T^*_{e_d,r}>\frac{t}{3}\right)\leq J\exp\left[-\frac{ct^{\gamma-1}}{r^{\gamma-d}}\right]\quad\text{for every $t\in\mathbb{R}_+^*$.}\]
At this point, let us set
\[r(t)=\left(\frac{ct^{\gamma-1}}{\ln(1+2J)}\right)^\frac{1}{\gamma-d},\quad\text{so that}\quad J\exp\left[-\frac{ct^{\gamma-1}}{r(t)^{\gamma-d}}\right]\leq\frac{1}{2}.\]
With that choice, we get
\[\mathbb{P}(T(0,e_d)\leq t)\geq\frac{1}{4}\cdot\mathbb{P}\left(V^{0,e_d}_r\geq\frac{3}{t}\right).\]
Finally, the quantity $\mathbb{P}\left(V^{0,e_d}_r\geq3/t\right)$ is the probability for a Poisson random variable with mean $p$ proportional to $\mu_d\left[\overline{B}(0,r)~;~\overline{B}(e_d,r)\right]\cdot(3/t)^{-(\gamma-1)}$ to be positive.
By virtue of Lemma \ref{lemdroitesboules}, the parameter $p$ is bounded between constants times $r^{2(d-1)}\cdot t^{\gamma-1}$: since
\[r(t)^{2(d-1)}\cdot t^{\gamma-1}=\frac{c}{\ln(1+2J)}\cdot t^\sigma\]
goes to $0$ as $t\to0^+$, the probability $\mathbb{P}\left(V^{0,e_d}_r\geq3/t\right)=1-e^{-p(t)}$ is asymptotic to $p(t)$, and thus we obtain the result of the proposition.
\end{proof}

\paragraph{Upper bound.}

We now turn to upper bounding the quick connection probability.
The idea is that, for $0$ and $e_d$ to be connected quickly, there must exist fast roads at all scales around both points.
More precisely, for every $r\in{]0,1]}$, any $\Pi$-path connecting $0$ and $e_d$ has to travel an Euclidean distance $r$ inside the ball $\overline{B}(0,r)$ \big(resp. $\overline{B}(e_d,r)$\big), at speed at most $V^0_r$ \big(resp. $V^{e_d}_r$\big).
This is illustrated in Figure \ref{figkahn}.
Therefore, on the event $(T(0,e_d)\leq t)$, we must have $V^0_r\geq r/t$ and $V^{e_d}_r\geq r/t$ for all $r\in{]0,1]}$.
As mentioned in Section \ref{secrappels}, the typical order of $V^x_r$ is $r^{(d-1)/(\gamma-1)}$, so --- intuitively --- the condition $V^x_r\geq r/t$ is restrictive only when $r$ is of order at least $t^{(\gamma-1)/(\gamma-d)}$.
Discretising scales, we set $r_k=2^{-k}$ for all $k\in\mathbb{N}$, and let $K$ be the largest integer $k$ such that $r_k\geq t^{(\gamma-1)/(\gamma-d)}$.
According to the previous discussion, we have
\begin{equation}\label{eqproofbs}
\mathbb{P}(T(0,e_d)\leq t)\leq\mathbb{P}\left(\text{$V^0_{r_k}\geq\frac{r_k}{t}$ and $V^{e_d}_{r_k}\geq\frac{r_k}{t}$ for all $k\in\llbracket0,K\rrbracket$}\right).
\end{equation}
In the following subparagraph, we explain how to control the probability of such an event.

\subparagraph{Multiscale lemma.}
In \eqref{eqproofbs}, the random variables $V^0_{r_k}$ and $V^{e_d}_{r_k}$ for $k\in\llbracket0,K\rrbracket$ are correlated; still, in order to control the right hand side, the idea is to exploit some independence.
Such independence is gained by examining the possible configurations of the process leading to the realisation of the event at stake (the proof of Lemma \ref{lemkahn} below should shed light on this sentence).
We derive the following ``multiscale lemma''.
\begin{lem}\label{lemkahn}
Let $(r_k)_{k\in\mathbb{N}}$ and $(v_k)_{k\in\mathbb{N}}$ be decreasing sequences of positive real numbers.
\begin{enumerate}
\item\label{lemkahn1} Let $x\in\mathbb{R}^d$.
For any integers $0\leq n\leq K$, we have
\begin{equation}\label{eqlemkahn1}
\mathbb{P}\left(\text{$V^x_{r_k}\geq v_k$ for all $k\in\llbracket n,K\rrbracket$}\right)\leq\sum_{j=0}^{K-n}\sum_{n=k_0<\ldots<k_{j+1}=K+1}\prod_{i=0}^j\mathbb{P}\left(V^x_{r_{k_{i+1}-1}}\geq v_{k_i}\right).
\end{equation}
\item\label{lemkahn2} Let $x\neq y\in\mathbb{R}^d$. 
For any integers $0\leq m,n\leq K$, we have
\begin{multline}\label{eqlemkahn2}
\mathbb{P}\left(\text{$V^x_{r_k}\geq v_k$ for all $k\in\llbracket m,K\rrbracket$ and $V^y_{r_k}\geq v_k$ for all $k\in\llbracket n,K\rrbracket$}\right)\\
\leq\sum_{j=0}^{2K+1-m-n}\sum_{(m,n)=(k_0,l_0)\prec\ldots\prec(k_{j+1},l_{j+1})=(K+1,K+1)}\prod_{i=0}^j\varpi_{k_i,l_i}^{k_{i+1},l_{i+1}},
\end{multline}
where $(k_{i+1},l_{i+1})\succ(k_i,l_i)$ means $k_{i+1}\geq k_i$, $l_{i+1}\geq l_i$, and $(k_{i+1},l_{i+1})\neq(k_i,l_i)$; and with
\[\varpi_{k_i,l_i}^{k_{i+1},l_{i+1}}=\begin{cases}
\mathbb{P}\left(V^{x,y}_{r_{k_{i+1}-1},r_{l_{i+1}-1}}\geq v_{k_i\wedge l_i}\right)&\text{if $k_{i+1}>k_i$ and $l_{i+1}>l_i$},\\
\mathbb{P}\left(V^x_{r_{k_{i+1}-1}}\geq v_{k_i\wedge l_i}\right)&\text{if $k_{i+1}>k_i$ and $l_{i+1}=l_i$},\\
\mathbb{P}\left(V^y_{r_{l_{i+1}-1}}\geq v_{k_i\wedge l_i}\right)&\text{if $k_{i+1}=k_i$ and $l_{i+1}>l_i$}.
\end{cases}\]
\end{enumerate}
\end{lem}
The first point was developed by Kahn in \cite[proof of Theorem 5.1]{kahn}, in a context where the scales $r_k$ are of order at least $|x-y|$, so that looking at both balls $\overline{B}(x,r_k)$ and $\overline{B}(y,r_k)$ is redundant.
Here in \eqref{eqproofbs}, we want to look at scales $r_k$ below $1$ around points $0$ and $e_d$, and thus we cannot afford to lose the information on the second point.
\begin{proof}
\begin{enumerate}
\item Fix $K\in\mathbb{N}$, and let $a_n=\mathbb{P}\left(\text{$V^x_{r_k}\geq v_k$ for all $k\in\llbracket n,K\rrbracket$}\right)$ for every $n\in\llbracket0,K+1\rrbracket$.
We claim that the sequence $(a_n)_{n\in\llbracket0,K+1\rrbracket}$ satisfies the following recurrence relation: for every $n\in\llbracket0,K\rrbracket$, we have
\begin{equation}\label{eqreclemkahn1}
a_n\leq\sum_{n<p\leq K+1}\mathbb{P}\left(V^x_{r_{p-1}}\geq v_n\right)\cdot a_p.
\end{equation}
Since $a_{K+1}=1$, this yields \eqref{eqlemkahn1}.
To prove \eqref{eqreclemkahn1}, let $n\in\llbracket0,K\rrbracket$.
Since $n\leq K$, on the event $A_n$ we have $V^x_{r_n}\geq v_n$.
Looking at the largest integer $p\geq n$ such that $V^x_{r_p}\geq v_n$, we write
\[a_n=\sum_{n\leq p<K}\mathbb{P}\left(\text{$V^x_{r_p}\geq v_n>V^x_{r_{p+1}}$ ; $V^x_{r_k}\geq v_k$ for all $k\in\llbracket p+1,K\rrbracket$}\right)+\mathbb{P}\left(V^x_{r_K}\geq v_n\right).\]
Let us call $B_p$ the event
\[\left(\text{$V^x_{r_p}\geq v_n>V^x_{r_{p+1}}$ ; $V^x_{r_k}\geq v_k$ for all $k\in\llbracket p+1,K\rrbracket$}\right).\]
On $B_p$, we have $X_p:=\Pi\left(\left.\left[\overline{B}(x,r_p)\right]\middle\backslash\left[\overline{B}(x,r_{p+1})\right]\right.\times[v_n,\infty[\right)>0$, and $X_p$ is independent of the random variables $\left(V^x_{r_k},~k\in\llbracket p+1,K\rrbracket\right)$: indeed, these random variables are all measurable with respect to the $\sigma$-algebra generated by the restriction of the Poisson random measure $\Pi$ to $\left[\overline{B}(x,r_{p+1})\right]\times\mathbb{R}_+^*$. 
Thus, we obtain
\[\mathbb{P}(B_p)\leq\mathbb{P}(X_p>0)\cdot\mathbb{P}\left(\text{$V^x_{r_k}\geq v_k$ for all $k\in\llbracket p+1,K\rrbracket$}\right)\leq\mathbb{P}\left(V^x_{r_p}\geq v_n\right)\cdot a_{p+1}.\]
Plugging this bound into the previous equation proves \eqref{eqreclemkahn1}, upon reindexing the sum.
\item The idea is the same as before, but things are a bit trickier since we are dealing with two points: at each scale $r_k$ around points $x$ and $y$, the speed $v_k$ can come from the same road.

Fix $K\in\mathbb{N}$.
For every $m,n\in\llbracket0,K+1\rrbracket$, denote by $A_{m,n}$ the event 
\[\left(\text{$V^x_{r_k}\geq v_k$ for all $k\in\llbracket m,K\rrbracket$ and $V^y_{r_k}\geq v_k$ for all $k\in\llbracket n,K\rrbracket$}\right),\]
and let $a_{m,n}=\mathbb{P}(A_{m,n})$.
We claim that the sequence $(a_{m,n})_{m,n\in\llbracket0,K+1\rrbracket}$ satisfies the following recurrence relation: for every $m,n\in\llbracket0,K+1\rrbracket$ with $(m,n)\prec(K+1,K+1)$, we have
\begin{equation}\label{eqreclemkahn2}
a_{m,n}\leq\sum_{(m,n)\prec(p,q)\preceq(K+1,K+1)}\varpi_{m,n}^{p,q}\cdot a_{p,q}.
\end{equation}
Recall that $(m,n)\prec(p,q)$ means $(m,n)\preceq(p,q)$ and ${(p,q)\neq(m,n)}$, where $(m,n)\preceq(p,q)$ stands for $p\geq m$ and $q\geq n$.
We also recall the definition of the $\varpi_{m,n}^{p,q}$ terms:
\[\varpi_{m,n}^{p,q}=\begin{cases}
\mathbb{P}\left(V^{x,y}_{r_{p-1},r_{q-1}}\geq v_{m\wedge n}\right)&\text{if $p>m$ and $q>n$},\\
\mathbb{P}\left(V^x_{r_{p-1}}\geq v_{m\wedge n}\right)&\text{if $p>m$ and $q=n$},\\
\mathbb{P}\left(V^y_{r_{q-1}}\geq v_{m\wedge n}\right)&\text{if $p=m$ and $q>n$}.
\end{cases}\]
Since $a_{K+1,K+1}=1$, the fact that $(a_{m,n})_{m,n\in\llbracket0,K+1\rrbracket}$ satisfies the above recurrence relation yields \eqref{eqlemkahn2}.
To prove \eqref{eqreclemkahn2}, let $m,n\in\llbracket0,K+1\rrbracket$ be such that $(m,n)\prec(K+1,K+1)$.
By invariance, we have the symmetry ${a_{m,n}=a_{n,m}}$, so we might as well assume $m\leq n$.
In particular, we have $m\leq K$, and on the event $A_{m,n}$ we have $V^x_{r_m}\geq v_m$.
Looking at the largest integer $p\geq m$ such that $V^x_{r_p}\geq v_m$, we write
\begin{equation}\label{eqlemkahn2proof1}
a_{m,n}=\sum_{m\leq p<K}\mathbb{P}\left(V^x_{r_p}\geq v_m>V^x_{r_{p+1}}~;~A_{p+1,n}\right)+\mathbb{P}\left(V^x_{r_K}\geq v_m~;~A_{K+1,n}\right).
\end{equation}
Let us call $B_p$ the event
\[\left(V^x_{r_p}\geq v_m>V^x_{r_{p+1}}~;~A_{p+1,n}\right).\]
On $B_p$, either (a) a road with speed limit $v\geq v_m$ passes trough both balls $\overline{B}(x,r_p)$ and $\overline{B}(y,r_n)$, or (b) the ball $\overline{B}(x,r_p)$ is traversed by a road with speed limit $v\geq v_m$ that does not pass through $\overline{B}(y,r_n)$.
\begin{itemize}
\item In case (a), looking at the largest integer $q\geq n$ such that $V^{x,y}_{r_p,r_q}\geq v_m$, we write
\[\mathbb{P}\left(B_p~;~V^{x,y}_{r_p,r_n}\geq v_m\right)\leq\sum_{n\leq q<K}\mathbb{P}\left(B_p~;~V^{x,y}_{r_p,r_q}\geq v_m>V^{x,y}_{r_p,r_{q+1}}\right)+\mathbb{P}\left(B_p~;~V^{x,y}_{r_p,r_K}\geq v_m\right).\]
On the event $\left(B_p~;~V^{x,y}_{r_p,r_q}\geq v_m>V^{x,y}_{r_p,r_{q+1}}\right)$, we have 
\[X_{p,q}:=\Pi\left(\left(\left[\overline{B}(x,r_p)\right]\middle\backslash\left[\overline{B}(x,r_{p+1})\right]\right)\cap\left(\left[\overline{B}(y,r_q)\right]\middle\backslash\left[\overline{B}(x,r_{q+1})\right]\right)\times[v_m,\infty[\right)>0,\]
and the event $A_{p+1,q+1}$ is realised.
Since $A_{p+1,q+1}$ belongs to the $\sigma$-algebra generated by the restriction of the Poisson random measure $\Pi$ to $\left(\left[\overline{B}(x,r_{p+1})\right]\cup\left[\overline{B}(y,r_{q+1})\right]\right)\times\mathbb{R}_+^*$, it is independent of $X_{p,q}$, and thus we obtain
\[\begin{split}
\mathbb{P}\left(B_p~;~V^{x,y}_{r_p,r_q}\geq v_m>V^{x,y}_{r_p,r_{q+1}}\right)&\leq\mathbb{P}(X_{p,q}>0)\cdot\mathbb{P}(A_{p+1,q+1})\\
&\leq\mathbb{P}\left(V^{x,y}_{r_p,r_q}\geq v_m\right)\cdot\mathbb{P}(A_{p+1,q+1})=\varpi_{m,n}^{p+1,q+1}\cdot a_{p+1,q+1}.
\end{split}\]
Similarly, we have
\[\mathbb{P}\left(B_p~;~V^{x,y}_{r_p,r_K}\geq v_m\right)\leq\varpi_{m,n}^{p+1,K+1}\cdot a_{p+1,K+1}.\]
We get
\[\mathbb{P}\left(B_p~;~V^{x,y}_{r_p,r_n}\geq v_m\right)\leq\sum_{n\leq q<K}\varpi_{m,n}^{p+1,q+1}\cdot a_{p+1,q+1}+\varpi_{m,n}^{p+1,K+1}\cdot a_{p+1,K+1}=\sum_{n<q\leq K+1}\varpi_{m,n}^{p+1,q}\cdot a_{p+1,q}.\]
\item In case (b), we have
\[Y_p:=\Pi\left(\left.\left(\left[\overline{B}(x,r_p)\right]\middle\backslash\left[\overline{B}(x,r_{p+1})\right]\right)\middle\backslash\left[\overline{B}(y,r_n)\right]\right.\times[v_m,\infty[\right)>0,\]
and the event $A_{p+1,n}$ is realised.
Since $A_{p+1,n}$ belongs to the $\sigma$-algebra generated by the restriction of the Poisson random measure $\Pi$ to $\left(\left[\overline{B}(x,r_{p+1})\right]\cup\left[\overline{B}(y,r_n)\right]\right)\times\mathbb{R}_+^*$, it is independent of $Y_p$, and thus we obtain
\[\mathbb{P}\left(B_p~;~V^x_{r_p}\geq v_m>V^y_{r_n}\right)\leq\mathbb{P}(Y_p>0)\cdot\mathbb{P}(A_{p+1,n})\leq\mathbb{P}\left(V^x_{r_p}\geq v_m\right)\cdot\mathbb{P}(A_{p+1,n})=\varpi_{m,n}^{p+1,n}\cdot a_{p+1,n}.\]
\end{itemize}
With the above distinction, we see that
\[\mathbb{P}(B_p)\leq\sum_{n<q\leq K+1}\varpi_{m,n}^{p+1,q}\cdot a_{p+1,q}+\varpi_{m,n}^{p+1,n}\cdot a_{p+1,n}=\sum_{q=n}^{K+1}\varpi_{m,n}^{p+1,q}\cdot a_{p+1,q}.\]
Similarly, we have
\[\mathbb{P}\left(V^x_{r_K}\geq v_m~;~A_{K+1,n}\right)\leq\sum_{q=n}^{K+1}\varpi_{m,n}^{K+1,q}\cdot a_{K+1,q}.\]
Plugging these bounds into \eqref{eqlemkahn2proof1}, we get
\[a_{m,n}\leq\sum_{m\leq p<K}\sum_{q=n}^{K+1}\varpi_{m,n}^{p+1,q}\cdot a_{p+1,q}+\sum_{q=n}^{K+1}\varpi_{m,n}^{K+1,q}\cdot a_{K+1,q}=\sum_{m<p\leq K+1}\sum_{q=n}^{K+1}\varpi_{m,n}^{p,q}\cdot a_{p,q},\]
which proves \eqref{eqreclemkahn2}.
\end{enumerate}
\end{proof}

In order to extract some quantitative information from the previous lemma, consider the following technical proposition, which will be used several times in the paper.

\begin{prop}\label{propmultiscale}
Let $(r_k)_{k\in\mathbb{N}}$ and $(v_k)_{k\in\mathbb{N}}$ be decreasing sequences of positive real numbers.
\begin{enumerate}
\item Let $x\in\mathbb{R}^d$.
For any integers $0\leq n\leq K$, we have
\[\mathbb{P}\left(\text{$V^x_{r_k}\geq v_k$ for all $k\in\llbracket n,K\rrbracket$}\right)\leq r_K^{d-1}\cdot v_n^{-(\gamma-1)}\cdot\exp\left[\sum_{k=n+1}^Kr_{k-1}^{d-1}\cdot v_k^{-(\gamma-1)}\right].\]
\item Let $x\neq y\in\mathbb{R}^d$.
Assume that $\left(r_k=|x-y|\cdot2^{-k}\right)_{k\in\mathbb{N}}$, and $(v_k=r_k/t)_{k\in\mathbb{N}}$, for some $t\in\mathbb{R}_+^*$.
There exists positive constants $C$ and $c$, which do not depend on $x$, $y$ or $t$, such that for every $K\in\mathbb{N}$:
\[\mathbb{P}\left(\text{$V^x_{r_k}\geq v_k$ and $V^y_{r_k}\geq v_k$ for all $k\in\llbracket0,K\rrbracket$}\right)\leq C\cdot\frac{r_K^{d-1}\cdot r_K^{d-1}}{|x-y|^{d-1}}\cdot v_0^{-(\gamma-1)}\cdot\exp\left[c\cdot r_K^{-(\gamma-d)}\cdot t^{\gamma-1}\right].\]
\end{enumerate}
\end{prop}
\begin{rem}
In the other direction, we have the obvious lower bounds
\[\mathbb{P}\left(\text{$V^x_{r_k}\geq v_k$ for all $k\in\llbracket n,K\rrbracket$}\right)\geq\mathbb{P}\left(V^x_{r_K}\geq v_n\right)=1-\exp\left[-r_K^{d-1}\cdot v_n^{-(\gamma-1)}\right],\]
and
\[\mathbb{P}\left(\text{$V^x_{r_k}\geq v_k$ and $V^y_{r_k}\geq v_k$ for all $k\in\llbracket0,K\rrbracket$}\right)\geq\mathbb{P}\left(V^{x,y}_{r_K}\geq v_0\right)\geq1-\exp\left[c'\cdot\frac{r_K^{d-1}\cdot r_K^{d-1}}{|x-y|^{d-1}}\cdot v_0^{-(\gamma-1)}\right],\]
where $c'>0$ is a constant obtained from Lemma \ref{lemdroitesboules}.
Under mild assumptions, these lower bounds match the above upper bounds up to multiplicative constants.
Thus, informally, if one asks $\Pi$ to put fast roads at all scales near one or two points, then the less expensive way for that to be realised is to have one very fast road passing very close to the prescribed points, and which alone is responsible for all these events.
This was formulated by Kahn for the one-point case in \cite[Remark 5.1]{kahn}.
\end{rem}
\begin{proof}
Let $K\in\mathbb{N}$.
\begin{enumerate}
\item By item \ref{lemkahn1} of the multiscale lemma, we have
\[\mathbb{P}\left(\text{$V^x_{r_k}\geq v_k$ for all $k\in\llbracket n,K\rrbracket$}\right)\leq\sum_{j=0}^{K-n}\sum_{n=k_0<\ldots<k_{j+1}=K+1}\prod_{i=0}^j\mathbb{P}\left(V^x_{r_{k_{i+1}-1}}\geq v_{k_i}\right).\]
Using the inequality $\mathbb{P}(\mathrm{Poisson}(\lambda)>0)\leq\lambda$, we have
\[\mathbb{P}\left(V^x_{r_{k_{i+1}-1}}\geq v_{k_i}\right)\leq r_{k_{i+1}-1}^{d-1}\cdot v_{k_i}^{-(\gamma-1)}\quad\text{for all integers $k_{i+1}>k_i$,}\]
hence
\[\prod_{i=0}^j\mathbb{P}\left(V^x_{r_{k_{i+1}-1}}\geq v_{k_i}\right)\leq r_K^{d-1}\cdot v_n^{-(\gamma-1)}\cdot\prod_{i=1}^jr_{k_i-1}^{d-1}\cdot v_{k_i}^{-(\gamma-1)}.\]
It follows that
\[\begin{split}
\sum_{0=k_0<\ldots<k_{j+1}=K+1}\prod_{i=0}^j\mathbb{P}\left(V^x_{r_{k_{i+1}-1}}\geq v_{k_i}\right)&\leq r_K^{d-1}\cdot v_n^{-(\gamma-1)}\cdot\sum_{n=k_0<\ldots<k_{j+1}=K+1}\prod_{i=1}^jr_{k_i-1}^{d-1}\cdot v_{k_i}^{-(\gamma-1)}\\
&\leq r_K^{d-1}\cdot v_n^{-(\gamma-1)}\cdot\frac{1}{j!}\cdot\sum_{n+1\leq k_1,\ldots,k_j\leq K}\prod_{i=1}^jr_{k_i-1}^{d-1}\cdot v_{k_i}^{-(\gamma-1)}\\
&=r_K^{d-1}\cdot v_n^{-(\gamma-1)}\cdot\frac{1}{j!}\cdot\left[\sum_{k=n+1}^Kr_{k-1}^{d-1}\cdot v_k^{-(\gamma-1)}\right]^j.
\end{split}\]
Finally, summing over $j$ we obtain
\[\sum_{j=0}^{K-n}\sum_{n=k_0<\ldots<k_{j+1}=K+1}\prod_{i=0}^j\mathbb{P}\left(V^x_{r_{k_{i+1}-1}}\geq v_{k_i}\right)\leq r_K^{d-1}\cdot v_n^{-(\gamma-1)}\cdot\exp\left[\sum_{k=n+1}^Kr_{k-1}^{d-1}\cdot v_k^{-(\gamma-1)}\right].\]

\item By item \ref{lemkahn2} of the multiscale lemma, we have (with the notation introduced in Lemma \ref{lemkahn}):
\[\mathbb{P}\left(\text{$V^x_{r_k}\geq v_k$ and $V^y_{r_k}\geq v_k$ for all $k\in\llbracket 0,K\rrbracket$}\right)\leq\sum_{j=0}^{2K+1}\sum_{(0,0)=(k_0,l_0)\prec\ldots\prec(k_{j+1},l_{j+1})=(K+1,K+1)}\prod_{i=0}^j\varpi_{k_i,l_i}^{k_{i+1},l_{i+1}}.\]
To continue, recall the assumptions: we have $r_k=|x-y|\cdot2^{-k}$ for all $k\in\mathbb{N}$, and $v_k=r_k/t$ for all $k\in\mathbb{N}$, for some $t\in\mathbb{R}_+^*$.
We claim that there exists a constant $C>0$, not depending on $x$, $y$ or $t$, such that
\begin{equation}\label{eqpreviouspoint}
\varpi_{m,n}^{p,q}\leq C\cdot\frac{r_{p-1}^{d-1}\cdot r_{q-1}^{d-1}\cdot t^{\gamma-1}}{r_m^{(d+\gamma-2)/2}\cdot r_n^{(d+\gamma-2)/2}}\quad\text{for all integers $(p,q)\succ(m,n)$},
\end{equation}
where for convenience we have set $r_{-1}=2|x-y|$.
Indeed, let $m,n,p,q\in\mathbb{N}$ be such that $(p,q)\succ(m,n)$.
\begin{itemize}
\item If $p>m$ and $q>n$, then $\varpi_{m,n}^{p,q}=\mathbb{P}\left(V^{x,y}_{r_{p-1},r_{q-1}}\geq v_{m\wedge n}\right)$.
Using the inequality ${\mathbb{P}(\mathrm{Poisson}(\lambda)>0)\leq\lambda}$ and Lemma \ref{lemdroitesboules}, we see that there exists a constant $C>0$ such that
\[\varpi_{m,n}^{p,q}\leq C\cdot\frac{r_{p-1}^{d-1}\cdot r_{q-1}^{d-1}}{|x-y|^{d-1}}\cdot v_{m\wedge n}^{-(\gamma-1)}=C\cdot\frac{r_{p-1}^{d-1}\cdot r_{q-1}^{d-1}\cdot t^{\gamma-1}}{|x-y|^{d-1}\cdot r_{m\wedge n}^{\gamma-1}}\leq C\cdot\frac{r_{p-1}^{d-1}\cdot r_{q-1}^{d-1}\cdot t^{\gamma-1}}{r_m^{(d+\gamma-2)/2}\cdot r_n^{(d+\gamma-2)/2}}.\]
\item If $p>m$ and $q=n$, then $\varpi_{m,n}^{q,n}=\mathbb{P}\left(V^x_{r_{p-1}}\geq v_{m\wedge n}\right)$.
Using the inequality $\mathbb{P}(\mathrm{Poisson}(\lambda)>0)\leq\lambda$, we see that
\[\varpi_{m,n}^{p,q}\leq r_{p-1}^{d-1}\cdot v_{m\wedge n}^{-(\gamma-1)}=2^{-(d-1)}\cdot\frac{r_{p-1}^{d-1}\cdot r_{q-1}^{d-1}\cdot t^{\gamma-1}}{r_n^{d-1}\cdot r_{m\wedge n}^{\gamma-1}}\leq2^{-(d-1)}\cdot\frac{r_{p-1}^{d-1}\cdot r_{q-1}^{d-1}\cdot t^{\gamma-1}}{r_m^{(d+\gamma-2)/2}\cdot r_n^{(d+\gamma-2)/2}}.\]
\item If $p=m$ and $q>n$, then $\varpi_{m,n}^{q,n}=\mathbb{P}\left(V^x_{r_{q-1}}\geq v_{m\wedge n}\right)$.
In the same way as above, we get
\[\varpi_{m,n}^{p,q}\leq2^{-(d-1)}\cdot\frac{r_{p-1}^{d-1}\cdot r_{q-1}^{d-1}\cdot t^{\gamma-1}}{r_m^{(d+\gamma-2)/2}\cdot r_n^{(d+\gamma-2)/2}}.\]
\end{itemize}
Now armed with \eqref{eqpreviouspoint}, we obtain
\[\begin{split}
\prod_{i=0}^j\varpi_{k_i,l_i}^{k_{i+1},l_{i+1}}&\leq C\cdot\frac{r_K^{d-1}\cdot r_K^{d-1}\cdot t^{\gamma-1}}{r_0^{(d+\gamma-2)/2}\cdot r_0^{(d+\gamma-2)/2}}\cdot\prod_{i=1}^j\left(C\cdot\frac{r_{k_i-1}^{d-1}\cdot r_{l_i-1}^{d-1}\cdot t^{\gamma-1}}{r_{k_i}^{(d+\gamma-2)/2}\cdot r_{l_i}^{(d+\gamma-2)/2}}\right)\\
&=C\cdot\frac{r_K^{d-1}\cdot r_K^{d-1}}{|x-y|^{d-1}}\cdot v_0^{-(\gamma-1)}\cdot\prod_{i=1}^j\frac{2^{d+\gamma-2}\cdot C\cdot t^{\gamma-1}}{r_{k_i-1}^{(\gamma-d)/2}\cdot r_{l_i-1}^{(\gamma-d)/2}}.
\end{split}\]
It follows that
\begin{eqnarray*}
\lefteqn{\sum_{(0,0)=(k_0,l_0)\prec\ldots\prec(k_{j+1},l_{j+1})=(K+1,K+1)}\prod_{i=0}^j\varpi_{k_i,l_i}^{k_{i+1},l_{i+1}}}\\
&\leq&C\cdot\frac{r_K^{d-1}\cdot r_K^{d-1}}{|x-y|^{d-1}}\cdot v_0^{-(\gamma-1)}\cdot\sum_{(0,0)=(k_0,l_0)\prec\ldots\prec(k_{j+1},l_{j+1})=(K+1,K+1)}\prod_{i=1}^j\frac{2^{d+\gamma-2}\cdot C\cdot t^{\gamma-1}}{r_{k_i-1}^{(\gamma-d)/2}\cdot r_{l_i-1}^{(\gamma-d)/2}}\\
&\leq&C\cdot\frac{r_K^{d-1}\cdot r_K^{d-1}}{|x-y|^{d-1}}\cdot v_0^{-(\gamma-1)}\cdot\frac{1}{j!}\cdot\sum_{(0,0)\prec(k_1,l_1);\ldots;(k_j,l_j)\prec(K+1,K+1)}\prod_{i=1}^j\frac{2^{d+\gamma-2}\cdot C\cdot t^{\gamma-1}}{r_{k_i-1}^{(\gamma-d)/2}\cdot r_{l_i-1}^{(\gamma-d)/2}}\\
&=&C\cdot\frac{r_K^{d-1}\cdot r_K^{d-1}}{|x-y|^{d-1}}\cdot v_0^{-(\gamma-1)}\cdot\frac{1}{j!}\cdot\left[\sum_{(0,0)\prec(k,l)\prec(K+1,K+1)}\frac{2^{d+\gamma-2}\cdot C\cdot t^{\gamma-1}}{r_{k-1}^{(\gamma-d)/2}\cdot r_{l-1}^{(\gamma-d)/2}}\right]^j\\
&\leq&C\cdot\frac{r_K^{d-1}\cdot r_K^{d-1}}{|x-y|^{d-1}}\cdot v_0^{-(\gamma-1)}\cdot\frac{1}{j!}\cdot\left[2^{d+\gamma-2}\cdot C\cdot\left(\sum_{k=0}^{K+1}r_{k-1}^{-(\gamma-d)/2}\right)^2\cdot t^{\gamma-1}\right]^j.
\end{eqnarray*}
Finally, summing over $j$ we obtain
\begin{multline*}
\sum_{j=0}^{2K+1}\sum_{(0,0)=(k_0,l_0)\prec\ldots\prec(k_{j+1},l_{j+1})=(K+1,K+1)}\prod_{i=0}^j\varpi_{k_i,l_i}^{k_{i+1},l_{i+1}}\\
\leq C\cdot\frac{r_K^{d-1}\cdot r_K^{d-1}}{|x-y|^{d-1}}\cdot v_0^{-(\gamma-1)}\cdot\exp\left[2^{d+\gamma-2}\cdot C\cdot\left(\sum_{k=0}^{K+1}r_{k-1}^{-(\gamma-d)/2}\right)^2\cdot t^{\gamma-1}\right].
\end{multline*}
The result follows, because $\sum_{k=0}^{K+1}r_{k-1}^{-(\gamma-d)/2}\leq r_K^{-(\gamma-d)/2}\cdot\sum_{l\geq0}2^{-(\gamma-d)/2\cdot l}$.
\end{enumerate}
\end{proof}

\subparagraph{The upper bound.}

With the previous results, we can now proceed with the upper bound for the quick connection probability.
Mercifully, all the work has been done already.
\begin{prop}\label{propbs}
There exists a constant $C>0$ such that
\[\text{$\mathbb{P}(T(0,e_d)\leq t)\leq C\cdot t^\sigma$ for all $t\in[0,1]$,\quad where $\sigma=\frac{(\gamma-1)(d+\gamma-2)}{\gamma-d}$.}\]
\end{prop}
\begin{proof}
Let $t\in{]0,1]}$.
Set $r_k=2^{-k}$ and $v_k=r_k/t$ for all $k\in\mathbb{N}$, and let ${K=\left\lfloor\log_2t^{-(\gamma-1)/(\gamma-d)}\right\rfloor}$ be the largest integer $k$ such that $r_k\geq t^{(\gamma-1)/(\gamma-d)}$.
By \eqref{eqproofbs}, we have
\[\mathbb{P}(T(0,e_d)\leq t)\leq\mathbb{P}\left(\text{$V^0_{r_k}\geq v_k$ and $V^{e_d}_{r_k}\geq v_k$ for all $k\in\llbracket0,K\rrbracket$}\right).\]
By Proposition \ref{propmultiscale}, there exists positive constants $C$ and $c$, not depending on $t$, such that
\[\mathbb{P}\left(\text{$V^0_{r_k}\geq v_k$ and $V^{e_d}_{r_k}\geq v_k$ for all $k\in\llbracket0,K\rrbracket$}\right)\leq C\cdot\frac{r_K^{d-1}\cdot r_K^{d-1}}{1^{d-1}}\cdot v_0^{-(\gamma-1)}\cdot\exp\left[c\cdot r_K^{-(\gamma-d)}\cdot t^{\gamma-1}\right].\]
The result follows, since $r_K\leq2t^{(\gamma-1)/(\gamma-d)}$, $v_0=1/t$, and $r_K^{-(\gamma-d)}\leq t^{-(\gamma-1)}$.
\end{proof}

\section{Lebesgue measure of $T$-balls}\label{secvolboules}

We pursue our analysis with the Lebesgue measure as mass distribution on $\left(\mathbb{R}^d,T\right)$, by examining the Lebesgue measure of balls for the metric $T$.
Recalling the notation ${\overline\Gamma(x,t)=\left\{y\in\mathbb{R}^d:T(x,y)\leq t\right\}}$ for the $T$-ball centered at $x\in\mathbb{R}^d$ with radius $t>0$, we are interested in the behaviour of the quantity $\lambda\left(\overline\Gamma(x,t)\right)$ as $t\to0^+$.

Note that for fixed $x\in\mathbb{R}^d$ and $t>0$, by invariance and scaling, we have the equality in distribution
\begin{equation}\label{eqvolboul}
\lambda\left(\overline\Gamma(x,t)\right)\overset{\text{\tiny law}}{=}\lambda\left(\overline\Gamma(0,1)\right)\cdot t^{s^*},
\end{equation}
where $s^*=(\gamma-1)d/(\gamma-d)$ is the Hausdorff dimension of $\left(\mathbb{R}^d,T\right)$.
Therefore, for ``typical'' points $x$ we should have $\lambda\left(\overline\Gamma(x,t)\right)\approx t^{s^*}$ for small $t$.
On the other hand, $T$-balls around points on roads should be bigger than typical balls, since $\Pi$-paths travelling in time $t$ from a point on a road can first use this road for a time $t/2$, and then still have time $t/2$ left to move around.
Heuristically, as illustrated in Figure \ref{figballsroads}, those balls should have Lebesgue measure $\approx t^{s_*}$, with ${s_*=s^*-(d-1)/(\gamma-d)}$.

\begin{figure}[ht]
\centering
\includegraphics[width=0.6\linewidth]{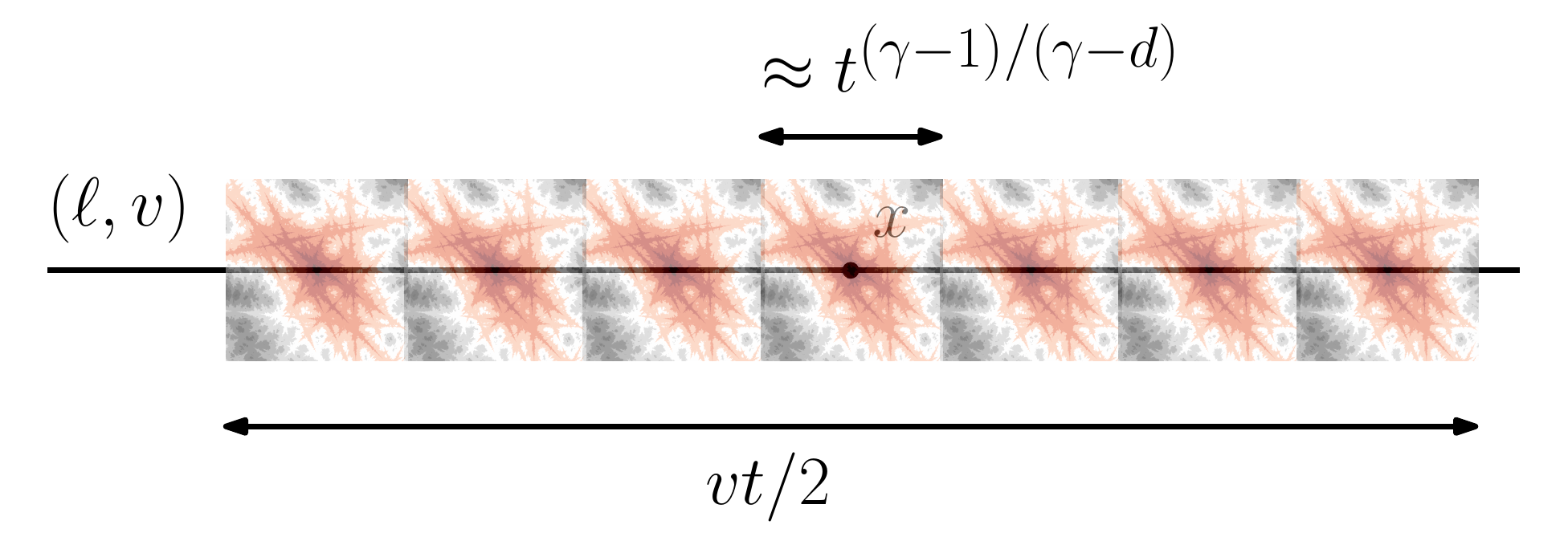}
\caption{Fix a realisation of $\Pi$. Let $(\ell,v)$ be a road of $\Pi$, and let $x\in\ell$. The $T$-ball $\overline\Gamma(x,t)$ contains the union $B_t=\bigcup_{y\in\ell:|x-y|\leq vt/2}\overline\Gamma(y,t/2)$.
On the picture, a ``typical'' $T$-ball of radius $t/2$ has been represented around $x$, and replicated along the line $\ell$.
Heuristically, this $T$-ball has Lebesgue measure $\approx t^{s^*}$, and thus spans an Euclidean distance $\approx t^{s^*/d}=t^{(\gamma-1)/(\gamma-d)}$.
Therefore, the union $B_t$ should have Lebesgue measure ${\approx t\cdot t^{-(\gamma-1)/(\gamma-d)}\cdot t^{s^*}=t^{s_*}}$ (the first two factors account for the number of ``disjoint'' typical $T$-balls of radius $t/2$ it takes to span an Euclidean distance $vt/2$, and the third term for the Lebesgue measure of a typical $T$-ball with radius $t/2$).}\label{figballsroads}
\begin{tabular}{lr}
\includegraphics[width=0.25\linewidth]{figures/simu4.png}&\includegraphics[width=0.25\linewidth]{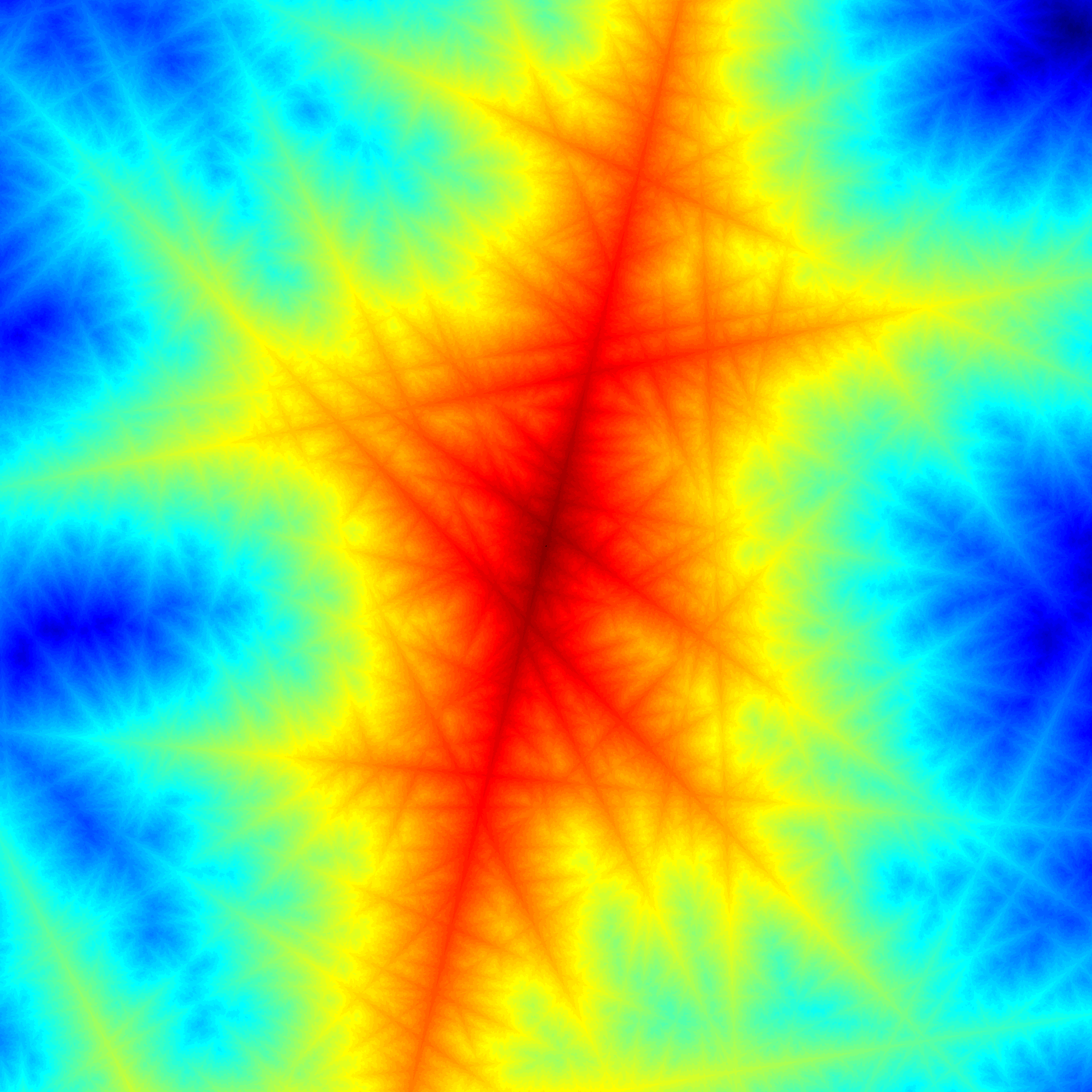}
\end{tabular}
\caption{Simulations by Arvind Singh. On the left hand side, a typical $T$-ball. On the right hand side, the origin is conditioned to be traversed by a road, and we seem to observe a shape similar to the one in Figure \ref{figballsroads}.}\label{figsimu'}
\end{figure}

In this section, the above heuristics are made rigorous in two different ways: in a first subsection we adopt a distributional point of view, before turning to the proof of Theorem \ref{thmvolboules} in Subsection \ref{subsecthmvolboul}.

\subsection{Distributional point of view}

Here we adopt a distributional point of view, describing the law of $\lambda\left(\overline\Gamma(X,t)\right)$ as $t\to0^+$, when $X$ is a point picked ``uniformly at random'', possibly on a road.
The main result of this subsection is Proposition \ref{propballsroads}

\paragraph{Typical points.}
Fixing $R>0$, let $X$ be a random variable with uniform distribution on $\overline{B}(0,R)$, independent of $\Pi$.
For every $t>0$, by invariance and scaling, we have the equality in distribution
\[\lambda\left(\overline\Gamma(X,t)\right)\overset{\text{\tiny law}}{=}\lambda\left(\overline\Gamma(0,1)\right)\cdot t^{s^*}.\]
In particular, trivially we have the convergence in distribution
\[\frac{\lambda\left(\overline\Gamma(X,t)\right)}{t^{s^*}}\underset{t\to0^+}{\overset{\text{\tiny law}}{\longrightarrow}}\lambda\left(\overline\Gamma(0,1)\right).\]
Note that in the above setting, almost surely $X$ is not on a road of $\Pi$. 
Let us now consider the case where $X$ is a point chosen ``uniformly at random'' on a road of $\Pi$.

\paragraph{Typical points on roads.}
Fixing $R>0$ and $v_0\in\mathbb{R}_+^*$, let us write here $N^R_{v_0}$ in place of $\Pi\left(\left[\overline{B}(0,R)\right]\times[v_0,\infty[\right)$ for short.
Define a random variable $X$ as follows: on the event $\left(N^R_{v_0}>0\right)$, pick a road $(\ell,v)$ of $\Pi$ in the set ${\left[\overline{B}(0,R)\right]\times[v_0,\infty[}$ uniformly at random, and let $X$ be a uniform point on the line segment $\ell\cap\overline{B}(0,R)$; on the complementary event $\left(N^R_{v_0}=0\right)$, just set $X=0$ (unimportant).
What we prove here is the following.

\begin{prop}\label{propballsroads}
There exists a constant $\Theta>0$, not depending on $R$ or $v_0$, such that the following holds: conditionally on the event $\left(N^R_{v_0}>0\right)$, we have the convergence in probability
\[\frac{\lambda\left(\overline\Gamma(X,t)\right)}{V(X)\cdot t^{s_*}}\underset{t\to0^+}{\longrightarrow}\Theta,\]
where $V(X)$ corresponds to the speed limit of the road on which lies $X$.
\end{prop}
As we shall see in the proof, the constant $\Theta$ emerges from a subadditivity argument.
\begin{proof}
It suffices to prove that for each $\varepsilon>0$,
\[\mathbb{E}\left[\frac{1}{N^R_{v_0}}\sum_{\substack{(\ell,v)\in\Pi\\\text{$\ell$ hits $\overline{B}(0,R)$ and $v\geq v_0$}}}\fint_{\ell\cap\overline{B}(0,R)}\mathbf{1}\left(\left|\Theta-\frac{\lambda\left(\overline\Gamma(x,t)\right)}{v\cdot t^{s_*}}\right|>\varepsilon\right)\mathrm{d}x~;~N^R_{v_0}>0\right]\underset{t\to0^+}{\longrightarrow}0,\]
where the symbol $\fint_{\ell\cap\overline{B}(0,R)}\mathrm{d}x$ denotes averaging in $x$ over the line segment $\ell\cap\overline{B}(0,R)$.
In the above expectation, bounding $\mathbf{1}\left(N^R_{v_0}>0\right)$ by $N^R_{v_0}$, we find that it suffices to prove that as $t\to0^+$, we have
\[\mathbb{E}\left[\sum_{\substack{(\ell,v)\in\Pi\\\text{$\ell$ hits $\overline{B}(0,R)$ and $v\geq v_0$}}}\fint_{\ell\cap\overline{B}(0,R)}\mathbf{1}\left(\left|\Theta-\frac{\lambda\left(\overline\Gamma(x,t)\right)}{v\cdot t^{s_*}}\right|>\varepsilon\right)\mathrm{d}x\right]\longrightarrow0.\]
Let us write $\varphi(\ell,v;\Pi)$ for
\[\fint_{\ell\cap\overline{B}(0,R)}\mathbf{1}\left(\left|\Theta-\frac{\lambda\left(\overline\Gamma(x,t)\right)}{v\cdot t^{s_*}}\right|>\varepsilon\right)\mathrm{d}x\cdot\mathbf{1}\left(\text{$\ell$ hits $\overline{B}(0,R)$ and $v\geq v_0$}\right).\]
By the Slivnyak--Mecke theorem (see, e.g, \cite[Theorem 4.1]{lastpenrose}), we have
\[\mathbb{E}\left[\sum_{(\ell,v)\in\Pi}\varphi(\ell,v;\Pi)\right]=c\int_{\mathbb{L}_d\times\mathbb{R}_+^*}\varphi\left(\ell,v;\Pi+\delta_{(\ell,v)}\right)\mathrm{d}\left[\mu_d\otimes v^{-\gamma}\mathrm{d}v\right](\ell,v),\]
where $c=\upsilon_{d-1}^{-1}\cdot(\gamma-1)$ is the normalising constant of \eqref{eqnorm}. 
This yields
\begin{multline*}
\mathbb{E}\left[\sum_{\substack{(\ell,v)\in\Pi\\\text{$\ell$ hits $\overline{B}(0,R)$ and $v\geq v_0$}}}\fint_{\ell\cap\overline{B}(0,R)}\mathbf{1}\left(\left|\Theta-\frac{\lambda\left(\overline\Gamma(x,t)\right)}{v\cdot t^{s_*}}\right|>\varepsilon\right)\mathrm{d}x\right]\\
=c\int_{\left[\overline{B}(0,R)\right]}\int_{v_0}^\infty\left\{\fint_{\ell\cap\overline{B}(0,R)}\mathbb{P}\left(\left|\Theta-\frac{\lambda\left(\overline{\Gamma_{\ell,v}}(x,t)\right)}{v\cdot t^{s_*}}\right|>\varepsilon\right)\mathrm{d}x\right\}\frac{\mathrm{d}v}{v^\gamma}\mathrm{d}\mu_d(\ell).
\end{multline*}
The notation $\overline{\Gamma_{\ell,v}}(x,t)$ in the last integral denotes the closed ball of radius $t$ centered at $x$ for the random metric $T_{\ell,v}$ induced by the process ${\Pi_{\ell,v}=\Pi+\delta_{(\ell,v)}}$, i.e, obtained using the road $(\ell,v)$ on top of the roads of $\Pi$.
By the bounded convergence theorem, we see that it suffices to prove the following to complete the proof of the proposition.
\begin{lem}
Fix $(\ell,v)\in\mathbb{L}_d\times\mathbb{R}_+^*$, and $x\in\ell$. 
Setting $\Pi_{\ell,v}=\Pi+\delta_{(\ell,v)}$, consider the random metric ${T_{\ell,v}:\mathbb{R}^d\times\mathbb{R}^d\rightarrow\mathbb{R}_+}$ induced by $\Pi_{\ell,v}$, which is given by $T_{\ell,v}(y,z)=\tau\left(\Pi_{\ell,v};y,z\right)$ for all $y,z\in\mathbb{R}^d$ --- where $\tau$ is the measurable function of \eqref{eqmeasurabilityT} ---, and consider the corresponding balls ${\overline{\Gamma_{\ell,v}}(x,t)=\left\{y\in\mathbb{R}^d:T_{\ell,v}(x,y)\leq t\right\}}$.
There exists a constant $\Theta>0$, not depending on $x$ or $(\ell,v)$, such that
\[\frac{\lambda\left(\overline{\Gamma_{\ell,v}}(x,t)\right)}{v\cdot t^{s_*}}\underset{t\to0^+}{\overset{\mathbb{P}}{\longrightarrow}}\Theta.\]
\end{lem}
\begin{proof}
For ease of writing, we consider the case where $\ell=\ell_d$, and $x=0$.
Let $t>0$.
On the event $G_t=\left(V^0_{vt}<v\right)$, we have
\[\overline{\Gamma_{\ell_d,v}}(0,t)=\bigcup_{|y|\leq vt}\overline\Gamma(ye_d,t-|y|/v).\]
Indeed, the converse inclusion is always true; and for the direct inclusion, let $z\in\overline{\Gamma_{\ell_d,v}}(0,t)$.
Denote by $ye_d$ the last exit point from $\ell_d$ of the $\Pi_{\ell_d,v}$-geodesic $\xi$ connecting $0$ to $z$, with $y\in\mathbb{R}$.
On $G_t$, the $\Pi_{\ell_d,v}$-geodesic $\xi$ must remain inside the ball $\overline{B}(0,vt)$; in particular, we have $|y|\leq vt$.
Moreover, since on that event $(\ell_d,v)$ is the fastest road of $\Pi_{\ell_d,v}$ passing through $\overline{B}(0,vt)$, the $\Pi_{\ell_d,v}$-geodesic $\xi$ must connect $0$ to $ye_d$ using the road $(\ell_d,v)$: we deduce that $T_{\ell_d,v}(0,ye_d)=|y|/v$.
Finally, since $\xi$ does not use the road $(\ell_d,v)$ between $ye_d$ and $z$, we have $T_{\ell_d,v}(ye_d,z)=T(ye_d,z)$.
We conclude that $T(ye_d,z)=T_{\ell_d,v}(0,z)-|y|/v\leq t-|y|/v$.

Since $\mathbb{P}(G_t)\rightarrow1$ as $t\to0^+$, to prove the Lemma it suffices to show that
\[\frac{\lambda\left(\bigcup_{|y|\leq vt}\overline\Gamma(ye_d,t-|y|/v)\right)}{v\cdot t^{s_*}}\underset{t\to0^+}{\overset{\mathbb{P}}{\longrightarrow}}\Theta.\]
Now, by scaling, we have the equality in distribution
\[\lambda\left(\bigcup_{|y|\leq vt}\overline\Gamma(ye_d,t-|y|/v)\right)\overset{\text{\tiny law}}{=}\lambda\left(\bigcup_{|y|\leq u}\overline\Gamma\left(ye_d,1-|y|/u\right)\right)\cdot t^{s^*},\quad\text{where $u=vt^{-(d-1)/(\gamma-d)}$},\]
hence
\[\frac{\lambda\left(\bigcup_{|y|\leq vt}\overline\Gamma(ye_d,t-|y|/v)\right)}{v\cdot t^{s_*}}\overset{\text{\tiny law}}{=}\frac{\lambda\left(\bigcup_{|y|\leq u}\overline\Gamma\left(ye_d,1-|y|/u\right)\right)}{u}.\]
To conclude, let us show that
\[\frac{\lambda\left(\bigcup_{|y|\leq u}\overline\Gamma(ye_d,1-|y|/u)\right)}{u}\underset{u\to\infty}{\overset{\mathbb{P}}{\longrightarrow}}\Theta.\]

\small\textsc{Exploiting subadditivity\textbf{.}}
\normalsize For each $t>0$, let
\[\Lambda^t_{a,b}=\lambda\left(\bigcup_{a\leq y\leq b}\overline\Gamma(ye_d,t)\right)\quad\text{for all $a\leq b\in\mathbb{R}$.}\]
We claim that the random variables $\left(\Lambda^t_{m,n}\right)_{m\leq n\in\mathbb{N}}$ satisfy the following:
\begin{enumerate}[label=(\roman*)]
\item for each $m\leq n\in\mathbb{N}$, we have $\Lambda^t_{0,n}\leq\Lambda^t_{0,m}+\Lambda^t_{m,n}$,
\item for each $q\in\mathbb{N}^*$, the sequence $\left(\Lambda^t_{qn,q(n+1)}\right)_{n\in\mathbb{N}}$ is stationary and mixing,
\item the distribution of $\left(\Lambda^t_{m,m+n}\right)_{n\in\mathbb{N}}$ does not depend on the integer $m\in\mathbb{N}$,
\item we have $\mathbb{E}\left[\Lambda^t_{0,1}\right]<\infty$, and there exists a constant $c=c(t)>0$ such that $\mathbb{E}\left[\Lambda^t_{0,n}\right]\geq c\cdot n$ for all $n\in\mathbb{N}^*$.
\end{enumerate}
These are the assumptions of the subadditive ergodic theorem as it is stated in \cite[Theorem 6.4.1]{durrett}.
Therefore, by the subadditive ergodic theorem, there exists a constant $\theta=\theta(t)>0$ such that
\begin{equation}\label{eqsubaddet}
\frac{\Lambda^t_{0,u}}{u}\underset{u\to\infty}{\longrightarrow}\theta(t)\quad\text{almost surely and in $L^1$}.
\end{equation}
The subadditivity (i) is clear.
(iii) and the stationarity in (ii) follow from invariance, and the mixing in (ii) follows from Proposition \ref{propmixing}.
It remains to check (iv).
To ease notation, let us take $t=1$, and simply write $\Lambda_{a,b}$ instead of $\Lambda^1_{a,b}$.
We begin with the second part of (iv): we have
\[\begin{split}
\mathbb{E}[\Lambda_{0,n}]&=\int_{\mathbb{R}^d}\mathbb{P}(T(z;[0,n]e_d)\leq1)\mathrm{d}z\\
&\geq\int_{\mathbb{R}^{d-1}\times[0,n]}\mathbb{P}(T((w,z_d);(0,z_d))\leq1)\mathrm{d}(w,z_d)\\
&=\int_0^n\int_{\mathbb{R}^{d-1}}\mathbb{P}(T((w,0);0)\leq1)\mathrm{d}w\mathrm{d}z_d=\int_{\mathbb{R}^{d-1}}\mathbb{P}\left(T(0,e_d)\leq|w|^{-\frac{\gamma-d}{\gamma-1}}\right)\mathrm{d}w\cdot n,
\end{split}\]
where we used invariance and scaling for the last two equalities.
Finally, for the first part of (iv), write
\[\mathbb{E}[\Lambda_{0,1}]=\int_{\mathbb{R}^d}\mathbb{P}(T(z;[0,1]e_d)\leq1)\mathrm{d}z\leq\upsilon_d2^d+\int_{|z|\geq2}\mathbb{P}(T(z;[0,1]e_d)\leq1)\mathrm{d}z.\]
To bound the last term, let $|z|\geq2$.
As for upper bounding the quick connection probability, we argue that if there exists $y\in[0,1]e_d$ such that  $T(z,y)\leq1$, then we must have $V^y_\rho\geq\rho$ and $V^z_\rho\geq\rho$ for all $\rho\in{]0,|y-z|]}$.
By the triangle inequality, it follows that $V^0_{2\rho}\geq V^y_\rho\geq\rho$ and $V^z_{2\rho}\geq V^z_\rho\geq\rho$ for all $\rho\in[1,|z|/2]$.
Thus, setting $\rho_k=|z|/2\cdot2^{-k}$ for all $k\in\mathbb{N}$, and letting $K=\left\lfloor\log_2(|z|/2)\right\rfloor$ be the largest integer $k$ such that $\rho_k\geq1$, we have:
\[\mathbb{P}(T(z;[0,1]e_d)\leq1)\leq\mathbb{P}\left(\text{$V^0_{2\rho_k}\geq\rho_k$ and $V^z_{2\rho_k}\geq\rho_k$ for all $k\in\llbracket0,K\rrbracket$}\right).\]
Using Proposition \ref{propmultiscale} with $\left(r_k=2\rho_k=|z|\cdot2^{-k}\right)_{k\in\mathbb{N}}$ and $(v_k=\rho_k=r_k/2)_{k\in\mathbb{N}}$, we obtain that there exists a constant $C>0$ such that
\[\mathbb{P}(T(z;[0,1]e_d)\leq1)\leq\frac{C}{|z|^{d+\gamma-2}}\quad\text{for all $|z|\geq2$,}\]
which achieves to show that $\mathbb{E}[\Lambda_{0,1}]<\infty$.

\small\textsc{Concluding the proof\textbf{.}}
\normalsize We claim that the following is a consequence of \eqref{eqsubaddet}:
\begin{equation}\label{eqconvprob}
\frac{\lambda\left(\bigcup_{0\leq y\leq u}\overline\Gamma(ye_d,1-y/u)\right)}{u}\underset{u\to\infty}{\overset{\mathbb{P}}{\longrightarrow}}\int_0^1\theta(1-r)\mathrm{d}r.
\end{equation}
To prove this, fix $n\in\mathbb{N}^*$.
For every $u>0$, we have
\[\bigcup_{0\leq y\leq u}\overline\Gamma\left(ye_d,1-\frac{y}{u}\right)\subset\bigcup_{0\leq k<n}\bigcup_{ku/n\leq y\leq(k+1)u/n}\overline\Gamma\left(ye_d,1-\frac{k}{n}\right)=:\bigcup_{0\leq k<n}R^n_k(u)\]
and
\[\bigcup_{0\leq y\leq u}\overline\Gamma\left(ye_d,1-\frac{y}{u}\right)\supset\bigcup_{0\leq k<n}\bigcup_{ku/n\leq y\leq(k+1)u/n}\overline\Gamma\left(ye_d,1-\frac{k+1}{n}\right)=:\bigcup_{0\leq k<n}L^n_k(u).\]
Then, the Bonferroni inequalities yield the bounds
\begin{equation}\label{eqbonferroni}
\sum_{0\leq k<n}\lambda(L^n_k(u))-\sum_{0\leq i<j<n}\lambda\left(L^n_i(u)\cap L^n_j(u)\right)\leq\lambda\left(\bigcup_{0\leq y\leq u}\overline\Gamma\left(ye_d,1-\frac{y}{u}\right)\right)\leq\sum_{0\leq k<n}\lambda\left(R^n_k(u)\right).
\end{equation}
Now, let us first control the cross term $\chi_u:=\sum_{0\leq i<j<n}\lambda\left(L^n_i(u)\cap L^n_j(u)\right)$ in the left hand side: we have
\[\begin{split}
\mathbb{E}\left[\chi_u\right]&\leq\sum_{0\leq i<j<n}\mathbb{E}\left[\lambda\left(\left(\bigcup_{iu/n\leq y\leq(i+1)u/n}\overline\Gamma(ye_d,1)\right)\cap\left(\bigcup_{ju/n\leq y\leq(j+1)u/n}\overline\Gamma(ye_d,1)\right)\right)\right]\\
&\leq\frac{n(n-1)}{2}\cdot\mathbb{E}\left[\lambda\left(\left(\bigcup_{-u\leq y\leq0}\overline\Gamma(ye_d,1)\right)\cap\left(\bigcup_{0\leq y\leq u}\overline\Gamma(ye_d,1)\right)\right)\right],
\end{split}\]
where we used invariance for the last inequality.
By the inclusion-exclusion formula and by invariance, we can write
\[\mathbb{E}\left[\lambda\left(\left(\bigcup_{-u\leq y\leq0}\overline\Gamma(ye_d,1)\right)\cap\left(\bigcup_{0\leq y\leq u}\overline\Gamma(ye_d,1)\right)\right)\right]=\mathbb{E}[\Lambda_{-u,u}]-\mathbb{E}[\Lambda_{-u,0}]-\mathbb{E}[\Lambda_{0,u}]=\mathbb{E}[\Lambda_{0,2u}]-2\mathbb{E}[\Lambda_{0,u}].\]
By subadditivity, we have $\mathbb{E}[\Lambda_{0,u}]/u\rightarrow\theta(1)$ as $u\to\infty$, and it follows that the above is $o(u)$.
In particular, by Markov's inequality, we have $\chi_u/u\overset{\mathbb{P}}{\rightarrow}0$ as $u\to\infty$.
Next, for each $k\in\llbracket0,n\llbracket$, by \eqref{eqsubaddet} and invariance, we have
\[\frac{\lambda(R^n_k(u))}{u}\underset{u\to\infty}{\overset{\mathbb{P}}{\longrightarrow}}\frac{\theta(1-k/n)}{n}\quad\text{and}\quad\frac{\lambda(L^n_k(u))}{u}\underset{u\to\infty}{\overset{\mathbb{P}}{\longrightarrow}}\frac{\theta(1-(k+1)/n)}{n}.\]
Therefore, dividing \eqref{eqbonferroni} by $u$, the right and left hand sides respectively converge, as $u\to\infty$, to
\[\sum_{0\leq k<n}\frac{\theta(1-k/n)}{n}\quad\text{and}\quad\sum_{0\leq k<n}\frac{\theta(1-(k+1)/n)}{n}-0\quad\text{in probability.}\]
Since these quantities can be made arbitrarily close to $\int_0^1\theta(1-r)\mathrm{d}r$ by picking $n$ large enough, we obtain \eqref{eqconvprob}.
By invariance, we get as well
\[\frac{\lambda\left(\bigcup_{-u\leq y\leq0}\overline\Gamma(ye_d,1-|y|/u)\right)}{u}\underset{u\to\infty}{\overset{\mathbb{P}}{\longrightarrow}}\int_0^1\theta(1-r)\mathrm{d}r=\int_{-1}^0\theta(1-|r|)\mathrm{d}r.\]
Using one last time the inclusion-exclusion formula, we conclude that
\[\frac{\lambda\left(\bigcup_{|y|\leq u}\overline\Gamma(ye_d,1-|y|/u)\right)}{u}\underset{u\to\infty}{\overset{\mathbb{P}}{\longrightarrow}}\int_{-1}^0\theta(1-|r|)\mathrm{d}r+\int_0^1\theta(1-r)\mathrm{d}r-0=\int_{-1}^1\theta(1-|r|)\mathrm{d}r=:\Theta,\]
since --- as above --- the cross term is negligible.
\end{proof}
\end{proof}

\subsection{Almost sure results}\label{subsecthmvolboul}

In this subsection, with a different point of view we establish Theorem \ref{thmvolboules}, which states that almost surely the following holds:
\begin{enumerate}
\item\label{thmvbitem1} for $\lambda$-almost every $x\in\mathbb{R}^d$, we have $\lambda\left(\overline\Gamma(x,t)\right)=t^{s^*+o(1)}$ as $t\to0^+$,
\item\label{thmvbitem2} for every $x\in\mathcal{L}$, we have $\lambda\left(\overline\Gamma(x,t)\right)=t^{s_*+o(1)}$ as $t\to0^+$,
\item\label{thmvbitem3} for every $x\in\mathbb{R}^d$, we have $t^{s^*+o(1)}\leq\lambda\left(\overline\Gamma(x,t)\right)\leq t^{s_*-o(1)}$ as $t\to0^+$.
\end{enumerate}
We recall here that $s^*=(\gamma-1)d/(\gamma-d)$, and $s_*=s^*-(d-1)/(\gamma-d)$.
To prove Theorem \ref{thmvolboules}, we derive the upper bound in item \ref{thmvbitem1} (Proposition \ref{propboultypbs}), the lower bound in item \ref{thmvbitem2} (Proposition \ref{propballsroadslb}), and both bounds of item \ref{thmvbitem3} (Proposition \ref{propallballslb} for the lower bound, and Proposition \ref{propallballsub} for the upper bound).
Let us start with the following << upper bound for typical points >>.

\begin{prop}\label{propboultypbs}
Almost surely, for $\lambda$-almost every $x\in\mathbb{R}^d$, there exists a constant $C=C(\omega,x)>0$ such that \[\lambda\left(\overline\Gamma(x,t)\right)\leq  C\cdot t^{s^*}\cdot\ln(1/t)^2\quad\text{for all $t\in{]0,1]}$.}\]
In particular, for $\lambda$-almost every $x\in\mathbb{R}^d$, we have $\lambda\left(\overline\Gamma(x,t)\right)=t^{s^*+o(1)}$ as $t\to0^+$.
\end{prop}
Such a statement does not come as a surprise, since the Lebesgue measure is a mass distribution on $\left(\mathbb{R}^d,T\right)$ for which the energy integrals $\left(\iint T(x,y)^{-s}\mathrm{d}y\mathrm{d}x\right)_{0\leq s<s^*}$ are finite (see the discussion in \cite[Chapter 8]{mattila}).
\begin{proof}
Recall the equality in distribution \eqref{eqvolboul}.
We claim that the random variable $\lambda\left(\overline\Gamma(0,1)\right)$ is integrable.
Indeed, we have (by Fubini's theorem and scaling)
\[\mathbb{E}\left[\lambda\left(\overline\Gamma(0,1)\right)\right]=\int_{\mathbb{R}^d}\mathbb{P}(T(0,y)\leq1)\mathrm{d}y=\int_{\mathbb{R}^d}\mathbb{P}\left(T(0,e_d)\leq|y|^{-\frac{\gamma-d}{\gamma-1}}\right)\mathrm{d}y,\]
and by Theorem \ref{thmquickco}, there exists a constant $C>0$ such that
\[\mathbb{P}\left(T(0,e_d)\leq|y|^{-\frac{\gamma-d}{\gamma-1}}\right)\leq\frac{C}{|y|^{d+\gamma-2}}\quad\text{for all $|y|\geq1$.}\]
Now, let $x\in\mathbb{R}^d$, and set $t_n=2^{-n}$ for all $n\in\mathbb{N}$.
By \eqref{eqvolboul} and Markov's inequality, we have
\[\mathbb{P}\left(\lambda\left(\overline\Gamma(x,t_n)\right)\geq t_n^{s^*}\cdot n^2\right)=\mathbb{P}\left(\lambda\left(\overline\Gamma(0,1)\right)\geq n^2\right)\leq\frac{\mathbb{E}\left[\lambda\left(\overline\Gamma(0,1)\right)\right]}{n^2}\quad\text{for all $n\in\mathbb{N}^*$.}\]
Therefore, by the Borel-Cantelli lemma and by Fubini's theorem, we obtain
\begin{eqnarray*}
\lefteqn{\mathbb{E}\left[\lambda\left\{x\in\mathbb{R}^d:\text{$\lambda\left(\overline\Gamma(x,t_n)\right)\geq t_n^{s^*}\cdot n^2$ for infinitely many $n$}\right\}\right]}\\
&=&\int_{\mathbb{R}^d}\mathbb{P}\left(\text{$\lambda\left(\overline\Gamma(x,t_n)\right)\geq t_n^{s^*}\cdot n^2$ for infinitely many $n$}\right)\mathrm{d}x\\
&=&0.
\end{eqnarray*}
Therefore, almost surely, for $\lambda$-almost every $x\in\mathbb{R}^d$, there exists a constant $C=C(\omega,x)>0$ such that
\[\lambda\left(\overline\Gamma(x,t_n)\right)\leq C\cdot t_n^{s^*}\cdot n^2\quad\text{for all $n\in\mathbb{N}^*$.}\]
The result of the proposition readily follows.
\end{proof}

In the other direction, we have the following << lower bound for all points >>, as a simple consequence of Proposition \ref{propunifcontrolT}.
\begin{prop}\label{propallballslb}
Almost surely, for every $R>0$, there exists a positive function $r$ (depending on $\omega$ and $R$), defined over some interval $]0,t_0[$ and satisfying $r(t)=t^{(\gamma-1)/(\gamma-d)+o(1)}$ as $t\to0^+$, such that for every $t\in{]0,t_0[}$, we have 
\[\overline\Gamma(x,t)\supset\overline{B}(x,r(t))\quad\text{for all $x\in\overline{B}(0,R)$.}\]
In particular, for every $x\in\mathbb{R}^d$, we have $\lambda\left(\overline\Gamma(x,t)\right)\geq t^{s^*+o(1)}$ as $t\to0^+$.
\end{prop}
\begin{proof}
By Proposition \ref{propunifcontrolT}, almost surely, for every $R>0$, there exists a constant $C=C(\omega,R)>0$ such that
\[T(x,y)\leq C\cdot|x-y|^\frac{\gamma-d}{\gamma-1}\cdot\ln\left(\frac{4R}{|x-y|}\right)^\frac{1}{\gamma-1}\quad\text{for all $x\neq y\in\overline{B}(0,R)$.}\]
Now, consider the function $\phi:r\in{]0,2R]}\mapsto r^{(\gamma-d)/(\gamma-1)}\cdot\ln(4R/r)^{1/(\gamma-1)}$.
A straightforward analysis shows that $\phi$ is strictly increasing over the interval $]0,r_0[$, where $r_0=4Re^{1/(\gamma-d)}$; and that if we denote by $\psi:{]0,\phi(r_0)[}\rightarrow{]0,r_0[}$ its inverse, then we have $\psi(t)=t^{(\gamma-1)/(\gamma-d)+o(1)}$ as $t\to0^+$.
Setting $r(t)=\psi(t/C)$, so that $C\cdot\phi(r(t))=t$, for $t$ small enough so as to have $r(t)\leq R/2$, we get $T^*_{x,r(t)}\leq t$ for every $x\in\overline{B}(0,R/2)$, i.e. $\overline\Gamma(x,t)\supset\overline{B}(x,r(t))$.
This yields the first assertion; and the second follows by taking the Lebesgue measure.
\end{proof}

The previous << lower bound for all points >> effortlessly yields the following << lower bound for points on roads >>, with the argument of Figure \ref{figballsroads}.

\begin{prop}\label{propballsroadslb}
Almost surely, for every $x\in\mathcal{L}$, we have $\lambda\left(\overline\Gamma(x,t)\right)\geq t^{s_*+o(1)}$ as $t\to0^+$.
\end{prop}
\begin{proof}
Fix a realisation of $\Pi$.
Let $(\ell,v)$ be a road of $\Pi$, and let $x\in\ell$.
Let $R>|x|$: by Proposition \ref{propallballslb}, there exists a positive function $r$, defined over some interval $]0,t_0[$ and satisfying $r(t)=t^{(\gamma-1)/(\gamma-d)+o(1)}$ as $t\to0^+$, such that for every $t\in{]0,t_0[}$, we have $\overline\Gamma(x,t)\supset\overline{B}(x,r(t))$ for all $x\in\overline{B}(0,R)$.
As explained in Figure \ref{figballsroads}, we have
\[\overline\Gamma(x,t)\supset\bigcup_{y\in\ell:|x-y|\leq vt/2}\overline\Gamma\left(y,\frac{t}{2}\right),\]
thus
\[\overline\Gamma(x,t)\supset\bigcup_{y\in\ell:|x-y|\leq vt/2}\overline{B}\left(y,r\left(\frac{t}{2}\right)\right)\quad\text{for all sufficiently small $t$.}\]
Placing points $(y_i)_{i\in\mathbb{Z}}$ to be $2r(t/2)$-apart on $\ell$ so that the balls $\left(\overline{B}(y_i,r(t/2));~i\in\mathbb{Z}\right)$ have disjoint interiors, we obtain
\[\lambda\left(\overline\Gamma(x,t)\right)\geq\#\left\{i\in\mathbb{Z}:|x-y_i|\leq\frac{vt}{2}\right\}\cdot\upsilon_dr\left(\frac{t}{2}\right)^d\geq\left(\frac{vt}{2r(t/2)}-1\right)\cdot\upsilon_dr\left(\frac{t}{2}\right)^d\quad\text{for all sufficiently small $t$.}\]
This lower bound is of order $t\cdot r(t/2)^{d-1}=t^{s_*+o(1)}$ as $t\to0^+$, concluding the proof.
\end{proof}

Finally, to complete the proof of Theorem \ref{thmvolboules}, it remains to prove the following << upper bound for all points >>.

\begin{prop}\label{propallballsub}
Almost surely, for every $x\in\mathbb{R}^d$, we have $\lambda\left(\overline\Gamma(x,t)\right)\leq t^{s_*-o(1)}$ as $t\to0^+$.
\end{prop}
\begin{proof}
For positive real numbers $R$ and $t$, consider the random variable $\Lambda^R_t=\sup_{x\in\overline{B}(0,R)}\lambda\left(\overline\Gamma(x,t)\right)$\footnote{In all rigour, work with the actual $\Pi$-measurable random variable $\overline\Lambda^R_t=\sup_{x\in\overline{B}(0,R)\cap\mathbb{Q}}\lambda\left(\overline\Gamma(x,2t)\right)$, which upper bounds $\Lambda^R_t$ by virtue of Fatou's lemma.}.
To prove the proposition, we will show that almost surely, for each $s\in[0,s_*[$, we have $\Lambda^R_t\leq t^s$ for all sufficiently small $t$.

Let $s\in[0,s_*[$.
By scaling, we have the equality in distribution
\[\Lambda^R_t\overset{\text{\tiny law}}{=}\Lambda^{R'}_{1}\cdot t^{s^*},\quad\text{where $R'=t^{-\frac{\gamma-1}{\gamma-d}}\cdot R$.}\]
Therefore,
\[\mathbb{P}\left(\Lambda^R_t>t^s\right)=\mathbb{P}\left(\Lambda^{R'}_1>(R/R')^{(\gamma-d)(s-s^*)/(\gamma-1)}\right).\]
Since we have $(\gamma-d)(s-s^*)/(\gamma-1)<-(d-1)/(\gamma-1)$, there exists $\varepsilon>0$ such that
\[\mathbb{P}\left(\Lambda^{R'}_1>(R/R')^{(\gamma-d)(s-s^*)/(\gamma-1)}\right)\leq\mathbb{P}\left(\Lambda^{R'}_1>R'^{\frac{d-1}{\gamma-1}+\varepsilon}\right)\quad\text{for all sufficiently large $R$}.\]
Now, suppose we show that
\begin{equation}\label{eqproofpropallbub}
\sum_{n\geq0}\mathbb{P}\left(\Lambda^{R_n}_1>R_n^{\frac{d-1}{\gamma-1}+\varepsilon}\right)<\infty,\quad\text{where $R_n=2^n$ for all $n\in\mathbb{N}$.}
\end{equation}
Then, going back up the previous reasoning, we get 
\[\sum_{n\geq0}\mathbb{P}\left(\Lambda^R_{t_n}>t_n^s\right)<\infty,\quad\text{with $t_n=(R\cdot2^{-n})^\frac{\gamma-d}{\gamma-1}$ for all $n\in\mathbb{N}$,}\]
and by the Borel-Cantelli lemma we conclude that almost surely, we have $\Lambda^R_{t_n}\leq t_n^s$ for all sufficiently large $n$.
The result of the proposition readily follows.

In view of the above considerations, let us fix a realisation of $\Pi$, and consider the following deterministic lemma.
\begin{lem}\label{lemdet}
Let $\varepsilon\in{]0,(\gamma-d)/(\gamma-1)[}$, set $\alpha_i=(1-\varepsilon)^i\cdot((d-1)/(\gamma-1)+\varepsilon)$ for all $i\in\mathbb{N}$, and let $j\in\mathbb{N}^*$ be large enough so that $\alpha_jd\leq\varepsilon$.
Assume that there exists $N_0,\ldots,N_{j-1}\in\mathbb{N}$ such that for every $R\geq1$, the following holds, for each $i\in\llbracket-1,j\llbracket$:
\begin{enumerate}
\item[$(i=-1)$] $V^0_{2R}\leq R^{(d-1)/(\gamma-1)+\varepsilon}$,
\item[$(i=0)$] for every $x\in\overline{B}(0,R)$, there are at most $N_0$ roads of $\Pi_{R^{\alpha_1}}$ that pass through the ball $\overline{B}(x,R^{\alpha_0})$,
\item[$(1\leq i<j)$] for every $x\in\overline{B}(0,R+R^{\alpha_0}+\ldots+R^{\alpha_{i-1}})$, there are at most $N_i$ roads of $\Pi_{R^{\alpha_{i+1}}}$ that pass through the ball $\overline{B}(x,2R^{\alpha_i})$.
\end{enumerate}
Then, there exists a constant $C>0$ such that
\[\Lambda^R_1\leq C\cdot R^{\frac{d-1}{\gamma-1}+2\varepsilon}\quad\text{for every $R\geq1$}.\]
\end{lem}
To prove \eqref{eqproofpropallbub}, it will then remain to show that the assumptions of the lemma are satisfied with very high probability.
\begin{proof}
Let $R\geq1$.
The idea is that the information of items $(i,~i\in\llbracket-1,j\llbracket)$ bootstraps into the right upper bound for $\Lambda^R_1$.

Fix $x\in\overline{B}(0,R)$: we wish to upper bound $\lambda\left(\overline\Gamma(x,1)\right)$.
\begin{enumerate}
\item[$(i=-1)$] First, as $V^0_{2R}\leq R^{(d-1)/(\gamma-1)+\varepsilon}\leq R$, the $T$-ball $\overline\Gamma(x,1)$ must be included in $\overline{B}(0,2R)$.
Then, since the speed of $\Pi$-paths inside $\overline{B}(0,2R)$ cannot exceed $R^{(d-1)/(\gamma-1)+\varepsilon}=R^{\alpha_0}$, actually $\overline\Gamma(x,1)$ is included in $\overline{B}(x,R^{\alpha_0})$.

\item[$(i=0)$] By assumption, there are at most $N_0$ roads of $\Pi_{R^{\alpha_1}}$ that pass through the ball $\overline{B}(x,R^{\alpha_0})$: let us denote by $\mathcal{L}_\varnothing$ the corresponding set of lines.
Note that it could be the case that no road of $\Pi_{R^{\alpha_1}}$ passes through $x$: to allow for the following argument, let us fix an arbitrary road $\left(\ell_0,R^{\alpha_0}\right)$ that passes through $x$, and add $\ell_0$ to $\mathcal{L}_\varnothing$.
Now, for each line $\ell\in\mathcal{L}_\varnothing$, we place points to be $2R^{\alpha_1}$ apart on $\ell\cap\overline{B}(x,R^{\alpha_0})$, so that the Euclidean balls with radius $R^{\alpha_1}$ around those points cover the line segment.
Let us denote \emph{all} these points by $x_1,\ldots,x_{\kappa(\varnothing)}$: we have 
\[\kappa(\varnothing)\leq\#\mathcal{L}_\varnothing\cdot\left(\frac{R^{\alpha_0}}{R^{\alpha_1}}+1\right)\leq(N_0+1)\cdot2R^{\varepsilon\alpha_0},\]
and the balls $\left(\overline{B}(x_k,2R^{\alpha_1});~k\in\llbracket1,\kappa(\varnothing)\rrbracket\right)$ cover $\overline\Gamma(x,1)$.
Indeed, given $y\in\overline\Gamma(x,1)$, denote by $z_0$ the last exit point from $\bigcup_{\ell\in\mathcal{L}_\varnothing}\ell$ of the $\Pi$-geodesic $\xi$ that connects $x$ to $y$ (this $z_0$ would be ill-defined if we had not added $\ell_0$ to $\mathcal{L}_\varnothing$).
As an element of $\bigcup_{\ell\in\mathcal{L}_\varnothing}\ell\cap\overline{B}(x,R^{\alpha_0})$, the point $z_0$ lies in the ball $\overline{B}(x_{k_0},R^{\alpha_1})$ for some $k_0\in\llbracket1,\kappa(\varnothing)\rrbracket$.
Since the speed of $\xi$ cannot exceed $R^{\alpha_1}$ between $z_0$ and $y$, the point $y$ must lie in ${\overline{B}(z_0,R^{\alpha_1})\subset\overline{B}(x_{k_0},2R^{\alpha_1})}$.

\item[$(i=1)$] For each $u\in\llbracket1,\kappa(\varnothing)\rrbracket$, we proceed as follows.
As $x_u\in\overline{B}(0,R+R^{\alpha_0})$, by assumption there are at most $N_1$ roads with speed limit greater than or equal to $R^{\alpha_2}$ that pass through the ball $\overline{B}(x_u,2R^{\alpha_1})$: let us denote by $\mathcal{L}_u$ the corresponding set of lines.
For each line $\ell\in\mathcal{L}_u$, we place points to be $2R^{\alpha_2}$ apart on $\ell\cap\overline{B}(x_u,2R^{\alpha_1})$, so that the Euclidean balls with radius $R^{\alpha_2}$ around those points cover the line segment.
Let us denote \emph{all} these points by $x_{u1},\ldots,x_{u\kappa(u)}$ (where $uk$ is to be understood as the concatenation of the word $u$ with the letter $k$).
We have 
\[\kappa(u)\leq\#\mathcal{L}_u\cdot\left(\frac{2R^{\alpha_1}}{R^{\alpha_2}}+1\right)\leq N_1\cdot3R^{\varepsilon\alpha_1},\]
and the balls $\left(\overline{B}\left(x_{uk},2R^{\alpha_2}\right);~u\in\llbracket1,\kappa(\varnothing)\rrbracket,k\in\llbracket1,\kappa(u)\rrbracket\right)$   cover $\overline\Gamma(x,1)$.
Indeed, let $y\in\overline\Gamma(x,1)$.
We know that $z_0$, the last exit point from the set $\bigcup_{\ell\in\mathcal{L}_\varnothing}\ell$ of the $\Pi$-geodesic $\xi$ connecting $x$ to $y$, lies in the ball $\overline{B}(x_u,R^{\alpha_1})$ for some $u\in\llbracket1,\kappa(\varnothing)\rrbracket$.
Moreover, by the definition of $z_0$, the speed of $\xi$ cannot exceed $R^{\alpha_1}$ between $z_0$ and $y$.
Then, denote by $z_1$ the last exit point from $\bigcup_{\ell\in\mathcal{L}_u}\ell$ of the $\Pi$-geodesic segment between $z_0$ and $y$.
As an element of $\bigcup_{\ell\in\mathcal{L}_u}\ell\cap\overline{B}(x_u,2R^{\alpha_1})$, the point $z_1$ lies in the ball $\overline{B}(x_{uk},R^{\alpha_2})$ for some $k\in\llbracket1,\kappa(u)\rrbracket$.
Finally, since the speed of $\xi$ cannot exceed $R^{\alpha_2}$ between $z_1$ and $y$, the point $y$ must lie in ${\overline{B}(z_1,R^{\alpha_2})\subset\overline{B}(x_{uk},2R^{\alpha_2})}$.

\item[$\ldots$] Continuing like this by induction, we construct a plane tree $\mathbb{T}$ in which the root $\varnothing$ has $\kappa(\varnothing)$ children $1,\ldots,\kappa(\varnothing)$, each $u\in\llbracket1,\kappa(\varnothing)\rrbracket$ has $\kappa(u)$ children $u1,\ldots,u\kappa(u)$, etc.
By construction, for each $i\in\llbracket1,j\rrbracket$, the balls $\left(B(x_u,2R^{\alpha_i});~u\in\mathbb{T}:|u|=i\right)$ cover $\overline\Gamma(x,1)$.
\end{enumerate}
Since the balls $\left(B(x_u,2R^{\alpha_j});~u\in\mathbb{T}:|u|=j\right)$ cover $\overline\Gamma(x,1)$, we have
\[\begin{split}
\lambda\left(\overline\Gamma(x,1)\right)&\leq\#\{u\in\mathbb{T}:|u|=j\}\cdot\upsilon_d(2R^{\alpha_j})^d\\
&\leq(N_0+1)\cdot2R^{\varepsilon\alpha_0}\cdot N_1\cdot3R^{\varepsilon\alpha_1}\cdot\ldots\cdot N_{j-1}\cdot3R^{\varepsilon\alpha_{j-1}}\cdot\upsilon_d(2R^{\alpha_j})^d\\
&=:C\cdot R^{\varepsilon(\alpha_0+\ldots+\alpha_{j-1})+\alpha_jd}.
\end{split}\]
This bound is uniform in $x\in\overline{B}(0,R)$, and the constant $C$ does not depend on $R$.
The result of the lemma follows, since
\[\varepsilon(\alpha_0+\ldots+\alpha_{j-1})+\alpha_jd\leq\varepsilon\sum_{i\geq0}(1-\varepsilon)^i\cdot\left(\frac{d-1}{\gamma-1}+\varepsilon\right)+\varepsilon=\frac{d-1}{\gamma-1}+2\varepsilon.\]
\end{proof}
Let $R_n=2^n$ for all $n\in\mathbb{N}$.
Note that we may assume without loss of generality that $\varepsilon$ in \eqref{eqproofpropallbub} is less than ${(\gamma-d)/(\gamma-1)}$.
To prove \eqref{eqproofpropallbub}, in view of the previous lemma, it remains to find integers $N_0,\ldots,N_{j-1}$ for which the assumptions of the lemma are satisfied, on events whose complements have summable probability.
For each $i\in\llbracket0,j\llbracket$, fix $N_i\in\mathbb{N}$ large enough so that $(1-\alpha_i)d-\alpha_i((1-\varepsilon)(\gamma-1)-(d-1))(N_i+1)<0$.
We show below that there exists constants $C_i,c_i>0$ such that item $(i)$ of Lemma \ref{lemdet} holds with probability at least $1-C_i2^{-c_in}$.
This will complete the proof of the proposition.
\begin{enumerate}
\item[$(i=-1)$] Using the inequality $\mathbb{P}(\mathrm{Poisson}(\lambda)>0)\leq\lambda$, we get
\[\mathbb{P}\left(V^0_{2R_n}>R_n^{\frac{d-1}{\gamma-1}+\varepsilon}\right)\leq(2R_n)^{d-1}\cdot R_n^{-((d-1)+\varepsilon(\gamma-1))}=2^{d-1}\cdot2^{-\varepsilon(\gamma-1)n}=:C_{-1}\cdot2^{-c_{-1}\cdot n}\quad\text{for all $n\in\mathbb{N}$.}\]
\item[$(i=0)$] For each $n\in\mathbb{N}$, denote by $\left(\overline{B}(x,R_n^{\alpha_0});~x\in\mathcal{X}_0\right)$ a covering of $\overline{B}(0,R_n)$ by balls of radius $R_n^{\alpha_0}$, with centres ${x\in\overline{B}(0,R_n)}$ more than $R_n^{\alpha_0}$ apart from each other.
We have $\#\mathcal{X}_0\leq\left(2R_n\cdot R_n^{-\alpha_0}+1\right)^d\leq\left(3R_n^{1-\alpha_0}\right)^d$.
Now, suppose that there exists an $x\in\overline{B}(0,R_n)$ for which more than $N_0$ roads of $\Pi_{R_n^{\alpha_1}}$ pass through the ball $\overline{B}(x,R_n^{\alpha_0})$.
The point $x$ lies in the ball $\overline{B}(x',R_n^{\alpha_0})$ for some $x'\in\mathcal{X}_0$, and the Poisson random variable $\Pi\left(\left[\overline{B}(x',2R_n^{\alpha_0})\right]\times[R_n^{\alpha_1},\infty[\right)$ then exceeds $N_0$.
By a union bound we get, using the inequality ${\mathbb{P}\left(\mathrm{Poisson}(\lambda)>N_0\right)\leq\frac{\lambda^{N_0+1}}{(N_0+1)!}}$:
\[\begin{split}
\mathbb{P}\left(\exists x'\in\mathcal{X}_0:\Pi\left(\left[\overline{B}(x',2R_n^{\alpha_0})\right]\times[R_n^{\alpha_1},\infty[\right)>N_0\right)&\leq\left(3R_n^{1-\alpha_0}\right)^d\cdot\frac{\left[(2R_n^{\alpha_0})^{d-1}\cdot R_n^{-\alpha_1(\gamma-1)}\right]^{N_0+1}}{(N_0+1)!}\\
&=:C_0\cdot R_n^{(1-\alpha_0)d+(\alpha_0(d-1)-\alpha_1(\gamma-1))(N_0+1)}.
\end{split}\]
The constant $C_0$ does not depend on $n$, and by the definition of $N_0$, the exponent
\[(1-\alpha_0)d+(\alpha_0(d-1)-\alpha_1(\gamma-1))(N_0+1)=(1-\alpha_0)d-\alpha_0((1-\varepsilon)(\gamma-1)-(d-1))(N_0+1)=:-c_0\]
is negative.
\item[$(1\leq i<j)$] For each $n\in\mathbb{N}$, denote by $\left(\overline{B}(x,R_n^{\alpha_i});~x\in\mathcal{X}_i\right)$ a covering of $\overline{B}\left(0,R_n+R_n^{\alpha_0}+\ldots+R_n^{\alpha_{i-1}}\right)$ by balls of radius $R_n^{\alpha_i}$, with centres ${x\in\overline{B}\left(0,R_n+R_n^{\alpha_0}+\ldots+R_n^{\alpha_{i-1}}\right)}$ more than $R_n^{\alpha_i}$ apart from each other.
We have 
\[\#\mathcal{X}_i\leq\left(2\cdot\frac{R_n+R_n^{\alpha_0}+\ldots+R_n^{\alpha_{i-1}}}{R_n^{\alpha_i}}+1\right)^d\leq\left((2i+3)R_n^{1-\alpha_i}\right)^d.\]
Now, suppose that there exists an $x\in\overline{B}\left(0,R_n+R_n^{\alpha_0}+\ldots+R_n^{\alpha_{i-1}}\right)$ for which more than $N_i$ roads of $\Pi_{R_n^{\alpha_{i+1}}}$ pass through the ball $\overline{B}(x,2R_n^{\alpha_i})$.
The point $x$ lies in the ball $\overline{B}(x',R_n^{\alpha_i})$ for some $x'\in\mathcal{X}_i$, and the Poisson random variable $\Pi\left(\left[\overline{B}(x',3R_n^{\alpha_i})\right]\times\left[R_n^{\alpha_{i+1}},\infty\right[\right)$ then exceeds $N_i$.
By a union bound we get, using the inequality $\mathbb{P}\left(\mathrm{Poisson}(\lambda)>N_i\right)\leq\frac{\lambda^{N_i+1}}{(N_i+1)!}$:
\[\begin{split}
\mathbb{P}\left(\exists x'\in\mathcal{X}_i:\Pi\left(\left[\overline{B}(x,3R_n^{\alpha_i})\right]\times\left[R_n^{\alpha_{i+1}},\infty\right[\right)>N_i\right)&\leq\left((2i+3)R_n^{1-\alpha_i}\right)^d\cdot\frac{\left[(3R_n^{\alpha_i})^{d-1}\cdot R_n^{-\alpha_{i+1}(\gamma-1)}\right]^{N_i+1}}{(N_i+1)!}\\
&=:C_i\cdot R_n^{(1-\alpha_i)d+(\alpha_i(d-1)-\alpha_{i+1}(\gamma-1))(N_i+1)}.
\end{split}\]
The constant $C_i$ does not depend on $n$, and by the definition of $N_i$, the exponent
\[(1-\alpha_i)d+(\alpha_i(d-1)-\alpha_{i+1}(\gamma-1))(N_i+1)=(1-\alpha_i)d-\alpha_i((1-\varepsilon)(\gamma-1)-(d-1))(N_i+1)=:-c_i\]
is negative.
\end{enumerate}
\end{proof}

\section{Discussion}\label{secdisc}

To conclude this paper, we mention and discuss some natural problems our results leave open.

\paragraph{Gauge function and exact Hausdorff measure.}
In our analysis, we used the Lebesgue measure as mass distribution on $\left(\mathbb{R}^d,T\right)$.
\small\textsc{Question\textbf{:}} \normalsize is the Lebesgue measure a Hausdorff measure for some gauge function?

\paragraph{Complete description of the multifractal spectrum.}
The results of Theorem \ref{thmvolboules} are short of describing completely the local behaviour around all points in the measured metric space $\left(\mathbb{R}^d,T,\lambda\right)$.
\small\textsc{Question\textbf{:}} \normalsize can one describe the behaviour of $\lambda\left(\overline\Gamma(x,t)\right)$ as $t\to0^+$ for all $x\in\mathbb{R}^d$?

In the following, we make a modest first step towards that by considering the simpler question: << are there points $x$ other than points on roads for which $\lambda\left(\overline\Gamma(x,t)\right)=t^{s+o(1)}$ as $t\to0^+$, with an exponent $s<s^*$? >>.
Note that if we could answer no to the latter, then with Theorem \ref{thmvolboules} we would have the complete description mentioned above.

\begin{rem}
One of our original motivations with this question was to recover the roads of the process $\Pi$, given the measured metric space $\left(\mathbb{R}^d,T,\lambda\right)$.
Because of Proposition \ref{propult} below, it turns out that points on roads cannot in general be identified as the only points having untypically big $T$-balls around them.
However, the road network may be recovered from the metric space $\left(\mathbb{R}^d,T\right)$ through the speed-limit function $V$, as follows: almost surely, for every $x\in\mathbb{R}^d$, we have
\[\varliminf_{\substack{y\to x\\y\neq x}}\frac{T(x,y)}{|x-y|}=\frac{1}{V(x)}\in[0,\infty].\]
\end{rem}

In the right direction for a negative answer to the last question, we have the following.
\begin{prop}
Almost surely, the set
\[\bigcup_{s<s^*}\left\{x\in\mathbb{R}^d:\text{$\lambda\left(\overline\Gamma(x,t)\right)\geq t^s$ for all sufficiently small $t$}\right\}\]
has Euclidean Hausdorff dimension $1$.
\end{prop}
\begin{proof}
By Theorem \ref{thmvolboules}, this set contains $\mathcal{L}$, and thus has Euclidean Hausdorff dimension at least $1$.
Now consider the upper bound.
It suffices to prove that for each $s<s^*$ and every $R>0$, almost surely the set 
\[\left\{x\in\overline{B}(0,R):\text{$\lambda\left(\overline\Gamma(x,t)\right)\geq t^s$ for all sufficiently small $t$}\right\}\]
has Euclidean Hausdorff dimension at most $1$ (note that by Theorem \ref{thmvolboules}, this set is empty for $s<s_*$).
The idea is that, if $T$-balls with small radius around $x\in\mathbb{R}^d$ are untypically big, then there must be untypically fast roads passing near $x$ at all scales, and we can control the probability of this happening with the multiscale lemma.
More precisely now, let $x\in\mathbb{R}^d$, and suppose that there exists an $s\in[s_*,s^*[$ such that $\lambda\left(\overline\Gamma(x,t)\right)\geq t^s$ for all sufficiently small $t$.
Fixing $s'\in{]s,s^*[}$, set $r(t)=t^{s'/d}$, so that $\upsilon_dr(t)^d<t^s$ for all sufficiently small $t$.
For those $t$, the $T$-ball $\overline\Gamma(x,t)$ cannot be contained in $\overline{B}(x,r(t))$, hence we must have $V^x_{r(t)}\geq r(t)/t=r(t)^\alpha$, where $\alpha=1-d/s'\in{]0,(d-1)/(\gamma-1)[}$.
To conclude, it suffices to prove that for each $\alpha\in{]0,(d-1)/(\gamma-1)[}$, almost surely, the Euclidean Hausdorff dimension of the set 
\[\left\{x\in\overline{B}(0,R):\text{$V^x_r\geq r^\alpha$ for all sufficiently small $r$}\right\}\]
is at most $1$.
For each $k\in\mathbb{N}$, set $\rho_k=2^{-k}$, and let $\left(\overline{B}(x,\rho_k);~x\in\mathcal{X}_k\right)$ be a covering of $\overline{B}(0,R)$ by balls of radius $\rho_k$, with centres $x\in\overline{B}(0,R)$ more than $\rho_k$ apart from each other.
In particular, we have ${\#\mathcal{X}_k\leq(2R/\rho_k+1)^d}$.
Fix $n\in\mathbb{N}$.
For each $K\geq n$, the balls $\left(\overline{B}(x,\rho_K);~x\in\mathcal{X}_K:\text{$V^x_{2\rho_k}\geq\rho_k^\alpha$ for all $k\in\llbracket n,K\rrbracket$}\right)$ cover the set 
\[\left\{x\in\overline{B}(0,R):\text{$V^x_r\geq r^\alpha$ for all $r\in{]0,\rho_n]}$}\right\}.\]
Indeed, each element $x$ of that set lies in the ball $\overline{B}(x',\rho_K)$ for some $x'\in\mathcal{X}_K$, and for every $k\in\llbracket n,K\rrbracket$ we have $V^{x'}_{2\rho_k}\geq V^x_{\rho_k}\geq\rho_k^\alpha$.
Therefore, the $1$-dimensional Euclidean Hausdorff measure of the above set is upper bounded by
\[\varliminf_{K\to\infty}\#\left\{x\in\mathcal{X}_K:\text{$V^x_{2\rho_k}\geq\rho_k^\alpha$ for all $k\in\llbracket n,K\rrbracket$}\right\}\cdot2\rho_K=:H.\]
We wish to see that $H$ is almost surely finite.
By Fatou's lemma, we have
\[\mathbb{E}[H]\leq\varliminf_{K\to\infty}\mathbb{E}\left[\#\left\{x\in\mathcal{X}_K:\text{$V^x_{2\rho_k}\geq\rho_k^\alpha$ for all $k\in\llbracket n,K\rrbracket$}\right\}\cdot2\rho_K\right],\]
with
\[\mathbb{E}\left[\#\left\{x\in\mathcal{X}_K:\text{$V^x_{2\rho_k}\geq\rho_k^\alpha$ for all $k\in\llbracket n,K\rrbracket$}\right\}\cdot2\rho_K\right]=\#\mathcal{X}_K\cdot\mathbb{P}\left(\text{$V^0_{2\rho_k}\geq\rho_k^\alpha$ for all $k\in\llbracket n,K\rrbracket$}\right)\cdot2\rho_K\]
for each $K\geq n$.
Using Proposition \ref{propmultiscale} with ${(r_k=2\rho_k)_{k\in\mathbb{N}}}$ and ${(v_k=\rho_k^\alpha)_{k\in\mathbb{N}}}$, we obtain
\[\begin{split}
\mathbb{P}\left(\text{$V^0_{2\rho_k}\geq\rho_k^\alpha$ for all $k\in\llbracket n,K\rrbracket$}\right)&\leq(2\rho_K)^{d-1}\cdot\rho_n^{-\alpha(\gamma-1)}\cdot\exp\left[\sum_{k=n+1}^K(2\rho_{k-1})^{d-1}\cdot\rho_k^{-\alpha(\gamma-1)}\right]\\
&\leq(2\rho_K)^{d-1}\cdot\rho_n^{-\alpha(\gamma-1)}\cdot\exp\left[4^{d-1}\cdot\sum_{k\geq0}\rho_k^{(d-1)-\alpha(\gamma-1)}\right].
\end{split}\]
As $(d-1)-\alpha(\gamma-1)>0$, we have $\sum_{k\geq0}\rho_k^{(d-1)-\alpha(\gamma-1)}<\infty$, and we get 
\[\mathbb{P}\left(\text{$V^0_{2\rho_k}\geq\rho_k^\alpha$ for all $k\in\llbracket n,K\rrbracket$}\right)\leq C\cdot\rho_K^{d-1},\]
where the constant $C$ does not depend on $K$.
Thus, we obtain
\[\#\mathcal{X}_K\cdot\mathbb{P}\left(\text{$V^0_{2\rho_k}\geq\rho_k^\alpha$ for all $k\in\llbracket n,K\rrbracket$}\right)\cdot2\rho_K\leq(2R/\rho_K+1)^d\cdot C\cdot\rho_K^{d-1}\cdot2\rho_K\leq(2R+\rho_K)^d\cdot C\cdot2,\]
and it follows that
\[\mathbb{E}[H]\leq\varliminf_{K\to\infty}\#\mathcal{X}_K\cdot\mathbb{P}\left(\text{$V^x_{2\rho_k}\geq\rho_k^\alpha$ for all $k\in\llbracket n,K\rrbracket$}\right)\cdot2\rho_K\leq(2R)^d\cdot C\cdot2<\infty.\]
\end{proof}

The result of the previous proposition does not imply though that the set
\[\bigcup_{s<s^*}\left\{x\in\mathbb{R}^d:\text{$\lambda\left(\overline\Gamma(x,t)\right)\geq t^s$ for all sufficiently small $t$}\right\}\]
is reduced to $\mathcal{L}$.
In fact it is not, at least in the planar case.

\begin{prop}\label{propult}
Almost surely, there exists $x\in\left.\mathbb{R}^2\middle\backslash\mathcal{L}\right.$ such that $\lambda\left(\overline\Gamma(x,t)\right)\geq t^{s_*+o(1)}$ as $t\to0^+$.
\end{prop}
\begin{proof}
It suffices to prove that there exists $x\in\left.\overline{B}(0,1)\middle\backslash\mathcal{L}\right.$ such that ${V^x_r\geq\ln(1/r)^{-1}}$ for all sufficiently small $r$.
Indeed, suppose this has been done.
By Proposition \ref{propallballslb}, there exists a positive function $r$, defined over some interval $]0,t_0[$ and satisfying $r(t)=t^{(\gamma-1)/(\gamma-d)+o(1)}$ as $t\to0^+$, such that for every $t\in{]0,t_0[}$, we have 
\[\overline\Gamma(y,t)\supset\overline{B}(y,r(t))\quad\text{for all $y\in\overline{B}(0,2)$.}\]
For all sufficiently small $t$, we have both $\Gamma(x,t/2)\supset\overline{B}(0,r(t/2))$ and $V^x_{r(t/2)}\geq\ln(1/r(t/2))^{-1}$.
In particular, the ball $\overline{B}(x,r(t/2))$ is traversed by a road $(\ell,v)\in\Pi$ with speed limit $v\geq\ln(1/r(t/2))^{-1}$.
Denoting by $x'\in\overline{B}(x,r(t/2))$ the orthogonal projection of $x$ onto the line $\ell$, we have $T(x,x')\leq t/2$, and thus
\[\overline\Gamma(x,t)\supset\overline\Gamma(x',t/2)\supset\bigcup_{y\in\ell:|x'-y|\leq vt/2}\overline\Gamma\left(y,\frac{t}{2}\right)\supset\bigcup_{y\in\ell:|x'-y|\leq vt/2}\overline{B}\left(y,r\left(\frac{t}{2}\right)\right)\quad\text{for all sufficiently small $t$.}\]
Placing points $(y_i)_{i\in\mathbb{Z}}$ to be $2r(t/2)$-apart on $\ell$ so that the balls $\left(\overline{B}(y_i,r(t/2));~i\in\mathbb{Z}\right)$ have disjoint interiors, we get
\[\lambda\left(\bigcup_{y\in\ell:|x'-y|\leq vt/2}\overline{B}\left(y,r\left(\frac{t}{2}\right)\right)\right)\geq\#\left\{i\in\mathbb{Z}:|x'-y_i|\leq\frac{vt}{2}\right\}\cdot\upsilon_dr\left(\frac{t}{2}\right)^d\geq\left(\frac{vt}{2r(t/2)}-1\right)\cdot\upsilon_dr\left(\frac{t}{2}\right)^d.\]
We thus obtain
\[\lambda\left(\overline\Gamma(x,t)\right)\geq\left(\frac{\ln(1/r(t/2))^{-1}\cdot t}{2r(t/2)}-1\right)\cdot\upsilon_dr\left(\frac{t}{2}\right)^d\quad\text{for all sufficiently small $t$,}\]
which achieves to show that $\lambda\left(\overline\Gamma(x,t)\right)\geq t^{s_*+o(1)}$ as $t\to0^+$.

Now, we turn to the proof of the claim.
The idea, to obtain a point $x\in\left.\overline{B}(0,1)\middle\backslash\mathcal{L}\right.$ such that ${V^x_r\geq\ln(1/r)^{-1}}$ for all sufficiently small $r$, is to construct it using a compactness argument, forcing it to be close enough to fast roads, yet away from them (we are indebted to Jonas Kahn for the birth of this idea).
Denote by $(\ell_1,v_1);(\ell_2,v_2);\ldots$ the roads of $\Pi$ passing through the ball $\overline{B}(0,1)$, with $v_1>v_2>\ldots$.
The following steps are represented in Figure \ref{figsegmemb}.

\begin{figure}[ht]
\centering
\begin{tabular}{cc}
\includegraphics[width=0.4\linewidth]{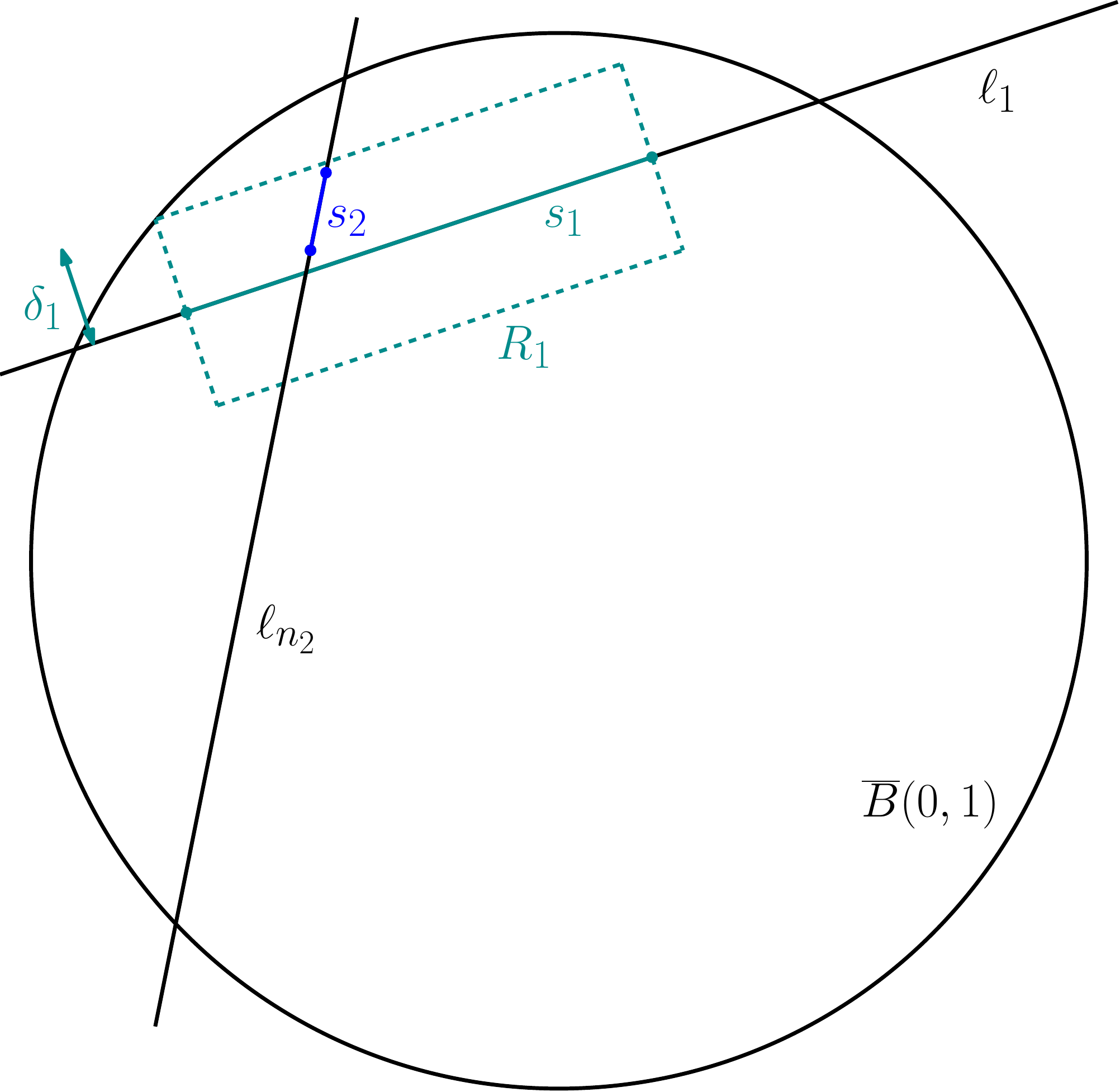}&\includegraphics[width=0.4\linewidth]{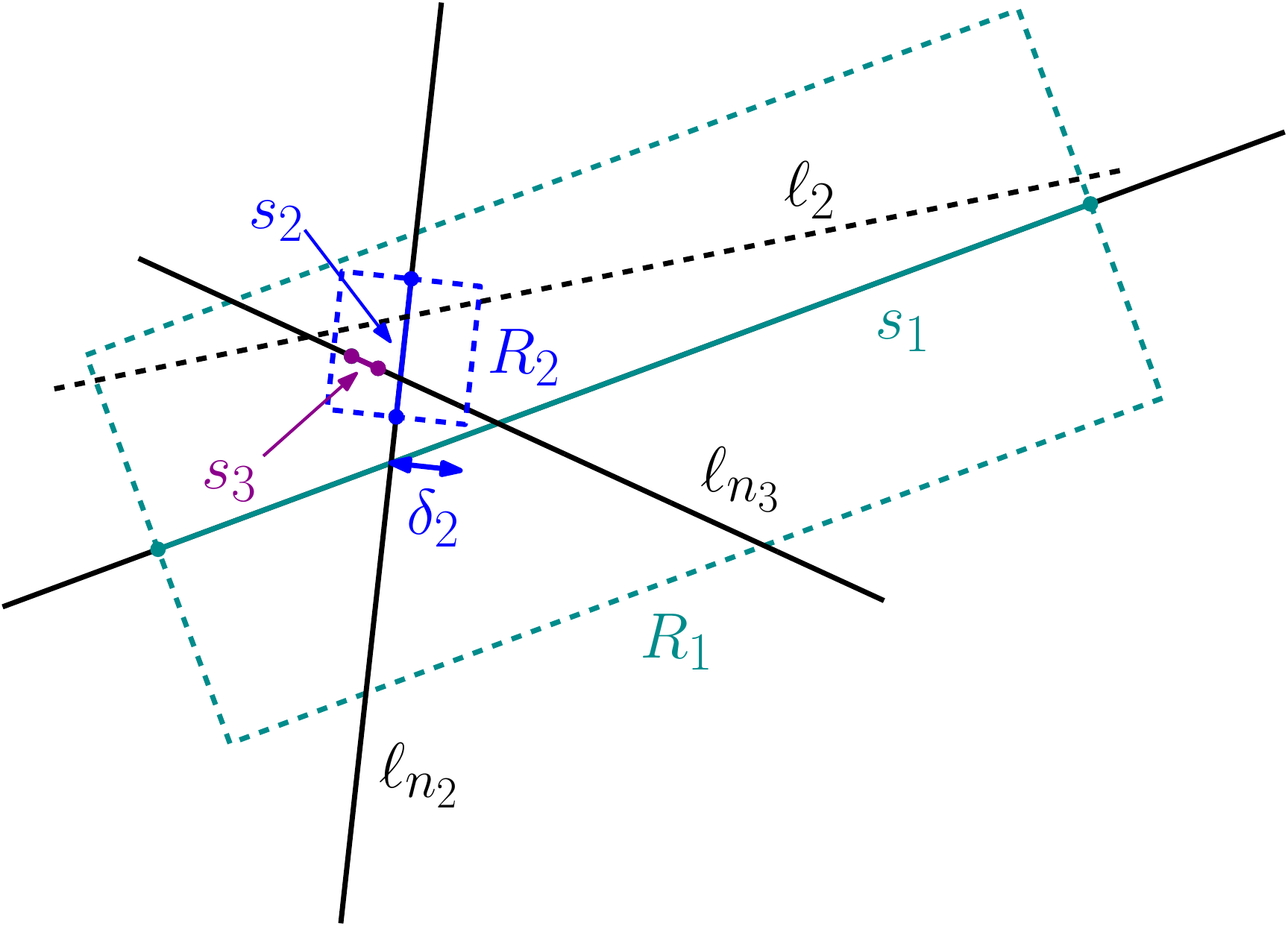}
\end{tabular}
\caption{\small\textsc{Left\textbf{:}} \normalsize the line segment $s_1$, the line $\ell_{n_2}$, the rectangle $R_1$ and the line segment $s_2$. \small\textsc{Right\textbf{:}} \normalsize zooming  in on the rectangle $R_1$, we represent the line segment $s_2$, the line $\ell_{n_3}$, also $\ell_2$ which is potentially there, then the rectangle $R_2$ and finally the line segment $s_3$.}\label{figsegmemb}
\end{figure}

\begin{enumerate}
\item[$(k=1)$] Let $n_1=1$.
Almost surely, the line $\ell_1$ intersects the interior of the ball $\overline{B}(0,1)$, and we can fix a segment $s_1$ of this line that is included in that interior.
\item[$(k=2)$] Let $n_2=\min\left\{n>1:\text{$\ell_n$ meets $s_1$}\right\}$ be the index of the next line hitting $s_1$, and define $R_1$ as a rectangle along that line segment, with half-side length $0<\delta_1\leq\exp\left[-v_{n_2}^{-1}\right]$ small enough so that $R_1\subset\overline{B}(0,1)$.
Zooming in on the rectangle $R_1$, we can fix a segment $s_2$ of the line $\ell_{n_2}$ that: (a) is included in the interior of $R_1$, and (b) avoids $\ell_1$.
\item[$(k=3)$] Let $n_3=\min\left\{n>n_2:\text{$\ell_n$ meets $s_2$}\right\}$ be the index of the next line hitting $s_2$, and define $R_2$ as a rectangle along that line segment, with half-side length $0<\delta_2\leq\exp\left[-v_{n_3}^{-1}\right]$ small enough so as to have $R_2$: (a) included in the rectangle $R_1$, and (b) avoiding the line $\ell_1$.
Zooming in on the rectangle $R_2$, fix a segment $s_3$ of the line $\ell_{n_3}$ that: (a) is included in the interior of $R_2$, and (b) avoids the lines $\ell_{n_2}$ and $\ell_2$ (possibly $n_2$ differs from $2$, and we want $s_3$ to avoid the line $\ell_2$).
\item[$\ldots$] Repeating the procedure inductively, we construct rectangles $R_1\supset R_2\supset\ldots$ such that, for each $k\in\mathbb{N}^*$, the rectangle $R_k$ avoids the lines $\ell_1,\ldots,\ell_{k-1}$, and any point $x\in R_k$ is at distance at most $\delta_k\leq\exp\left[-v_{n_{k+1}}^{-1}\right]$ from $\ell_{n_k}$.
\end{enumerate}
By compactness, the intersection $\bigcap_{k\geq1}R_k$ is non-empty (see, e.g, \cite[special case after Theorem 26.9]{munkres}): let $x$ be an element of that intersection.
By construction, the point $x$ belongs to $\overline{B}(0,1)$, does not lie on any of the lines $\ell_1,\ell_2,\ldots$, and satisfies $d(x,\ell_{n_k})\leq\exp\left[-v_{n_{k+1}}^{-1}\right]$ for all $k\in\mathbb{N}^*$.
Thus, the point $x$ is not on a road, and we claim that $V^x_r\geq\ln(1/r)^{-1}$ for all sufficiently small $r$.
Indeed, for $r\in\left]0,\exp\left[-v_1^{-1}\right]\right]$, taking $k\in\mathbb{N}^*$ such that $\exp\left[-v_{n_{k+1}}^{-1}\right]\leq r\leq\exp\left[-v_{n_k}^{-1}\right]$, we have $d(x,\ell_{n_k})\leq\exp\left[-v_{n_{k+1}}^{-1}\right]\leq r$, hence $V^x_r\geq v_{n_k}\geq\ln(1/r)^{-1}$.
\end{proof}

\begin{appendix}
\section*{Appendix}\label{appendix}
\begin{proof}[Proof of Lemma \ref{lemdroitesboules}]
Although the proof we present here is based on covering arguments, the result can be obtained by direct computations, with the description of $\mu_d$ as a pushforward measure recalled in Section \ref{secrappels}.
Let $x\neq y\in\mathbb{R}^d$ and $r,s>0$, and without loss of generality assume $r\geq s$.

\small\textsc{Lower bound\textbf{.}}
\normalsize For $r\leq|x-y|$, let $\left(\overline{B}(y',s)\right)_{y'\in\mathcal{Y}}$ be a covering of $\partial B(x,|x-y|)$ by balls of radius $s$, with centres $y'\in\partial B(x,|x-y|)$ more than $s$ apart from each other.
Since the balls $\left(\overline{B}(y',s/2);~y'\in\mathcal{Y}\right)$ are disjoint and included in the annulus $\left\{z\in\mathbb{R}^d:|x-y|-s/2\leq|x-z|\leq|x-y|+s/2\right\}$, we have $\#\mathcal{Y}\cdot\upsilon_d(s/2)^d\leq\upsilon_d\cdot d(|x-y|+s/2)^{d-1}\cdot s$, i.e, since $s\leq r\leq|x-y|$:
\[\#\mathcal{Y}\leq\left(\frac{2}{s}\right)^d\cdot d\left(\frac{3|x-y|}{2}\right)^{d-1}\cdot s=:C\cdot\frac{|x-y|^{d-1}}{s^{d-1}}.\]
Integrating
\[\mathbf{1}\left(\text{$\ell$ hits $\overline{B}(x,r)$}\right)\leq\frac{\#\left\{y'\in\mathcal{Y}:\text{$\ell$ hits $\overline{B}(x,r)$ and $\overline{B}(y',s)$}\right\}}{2}\quad\text{for all $\ell\in\mathbb{L}_d$}\]
with respect to $\mu_d$ yields, using invariance:
\[\upsilon_{d-1}r^{d-1}\leq\frac{\#\mathcal{Y}\cdot\mu_d\left[\overline{B}(x,r)~;~\overline{B}(y,s)\right]}{2}.\]
We get
\[\mu_d\left[\overline{B}(x,r)~;~\overline{B}(y,s)\right]\geq\frac{2\upsilon_{d-1}}{C}\cdot\frac{r^{d-1}\cdot s^{d-1}}{|x-y|^{d-1}}.\]

\small\textsc{Upper bound\textbf{.}}
\normalsize If $r\geq|x-y|$, then we have
\[\mu_d\left[\overline{B}(x,r)~;~\overline{B}(y,s)\right]\leq\mu_d\left[\overline{B}(y,s)\right]=\upsilon_{d-1}s^{d-1}\leq\upsilon_{d-1}\cdot\frac{r^{d-1}\cdot s^{d-1}}{|x-y|^{d-1}}.\]
In the case $r\leq|x-y|$, let $\mathcal{Y}$ be a maximal $3s$-separated subset of $\partial B(x,3|x-y|)$.
Since the balls $\left(\overline{B}(y',3s);~y'\in\mathcal{Y}\right)$ cover $\partial B(x,3|x-y|)$, the balls $\left(\overline{B}(y',4s);~y'\in\mathcal{Y}\right)$ cover the annulus $\left\{z\in\mathbb{R}^d:3|x-y|-s\leq|x-z|\leq3|x-y|+s\right\}$, and we have $\#\mathcal{Y}\cdot\upsilon_d(4s)^d\geq\upsilon_d\cdot d(3|x-y|-s)^{d-1}\cdot2s$, i.e, since $s\leq r\leq|x-y|$:
\[\#\mathcal{Y}\geq(4s)^{-d}\cdot d(2|x-y|)^{d-1}\cdot2s=:c\cdot\frac{|x-y|^{d-1}}{s^{d-1}}.\]
Now, the balls $\left(\overline{B}(y',s);~y'\in\mathcal{Y}\right)$ are at distance at least $s$ from one another, and included in the annulus
\[\left\{z\in\mathbb{R}^d:3|x-y|-s\leq|x-z|\leq3|x-y|+s\right\}.\]
The intersection of a line hitting $\overline{B}(x,r)$ with that annulus is the disjoint union of two line segments, each of length at most $2\sqrt{3}\cdot s$.
Therefore, any line passing through the ball $\overline{B}(x,r)$ cannot hit more than $2\left(1+2\sqrt{3}\right)$ of the $\overline{B}(y',s)$, and we have
\[\mathbf{1}\left(\text{$\ell$ hits $\overline{B}(x,r)$}\right)\geq\frac{\#\left\{y'\in\mathcal{Y}:\text{$\ell$ hits $\overline{B}(x,r)$ and $\overline{B}(y',s)$}\right\}}{2\left(1+2\sqrt{3}\right)}\quad\text{for all $\ell\in\mathbb{L}_d$.}\]
Integration with respect to $\mu_d$ yields, upon using invariance:
\[\upsilon_{d-1}r^{d-1}\geq\frac{\#\mathcal{Y}\cdot\mu_d\left[\overline{B}(x,r)~;~\overline{B}(y,s)\right]}{2\left(1+2\sqrt{3}\right)},\]
and we get
\[\mu_d\left[\overline{B}(x,r)~;~\overline{B}(y,s)\right]\leq\frac{2\left(1+2\sqrt{3}\right)\upsilon_{d-1}}{c}\cdot\frac{r^{d-1}\cdot s^{d-1}}{|x-y|^{d-1}}.\]
\end{proof}

\begin{proof}[Proof of $(\flat\flat)$ and $(*)$]
Fix $R>0$ and $v_0\in\mathbb{R}_+^*$.
Let $\delta>0$, and let $\left(\overline{B}(x,\delta)\right)_{x\in\mathcal{X}}$ be a covering of $\overline{B}(0,R)$ by balls of radius $\delta$, with centres $x\in\overline{B}(0,R)$ more than $\delta$ apart so that the $\left(\overline{B}(x,\delta/2)\right)_{x\in\mathcal{X}}$ are disjoint.
In particular, we have $\#\mathcal{X}\leq(2R/\delta+1)^d$.
On the underlying probability space, we have the following inclusions:
\[A_2:=\left(\text{there is a crossing of two roads of $\Pi_{v_0}$ inside $\overline{B}(0,R)$}\right)\subset\bigcup_{x\in\mathcal{X}}\left(\Pi\left(\left[\overline{B}(x,\delta)\right]\times[v_0,\infty[\right)\geq2\right),\]
and
\[A_3:=\left(\text{there is a crossing of three roads of $\Pi_{v_0}$ inside $\overline{B}(0,R)$}\right)\subset\bigcup_{x\in\mathcal{X}}\left(\Pi\left(\left[\overline{B}(x,\delta)\right]\times[v_0,\infty[\right)\geq3\right).\]
By union bounds, we get
\[\mathbb{P}\left(\bigcup_{x\in\mathcal{X}}\left(\Pi\left(\left[\overline{B}(x,\delta)\right]\times[v_0,\infty[\right)\geq2\right)\right)\leq\left(\frac{2R}{\delta}+1\right)^d\cdot\frac{\left(\delta^{d-1}\cdot v_0^{-(\gamma-1)}\right)^2}{2!}=:a_2(\delta)\]
and
\[\mathbb{P}\left(\bigcup_{x\in\mathcal{X}}\left(\Pi\left(\left[\overline{B}(x,\delta)\right]\times[v_0,\infty[\right)\geq3\right)\right)\leq\left(\frac{2R}{\delta}+1\right)^d\cdot\frac{\left(\delta^{d-1}\cdot v_0^{-(\gamma-1)}\right)^3}{3!}=:a_3(\delta).\]
As $\delta\to0^+$, we have $a_2(\delta)=O\left(\delta^{d-2}\right)$, hence $A_2$ is negligible in dimensions $d\geq3$; and $a_3(\delta)=O\left(\delta^{2d-3}\right)$, hence $A_3$ is negligible in all dimensions $d\geq2$.
Upon taking countable unions over $R$ and $v_0$, we obtain $(\flat\flat)$ and $(*)$.
\end{proof}
\end{appendix}
%
%

\begin{acks}[Acknowledgements]
I warmly thank Nicolas Curien and Arvind Singh for their constant support and guidance, and for their helpful comments on draft versions of this paper.
I would also like to thank Jonas Kahn for the stimulating discussions we had while he was visiting in Orsay, and Wilfrid Kendall for his feedback on an earlier version of this paper.
Finally, I am thankful to the referees for their careful reading of the paper: their remarks have helped improve the presentation a great deal.
\end{acks}
\begin{funding}
I acknowledge support from the ERC Advanced Grant 740943 \textsc{GeoBrown}.
\end{funding}


\bibliographystyle{imsart-number} 
\bibliography{biblio.bib}       


\end{document}